\documentclass[11pt,reqno]{amsart}
\usepackage{amssymb, hyperref}
\usepackage{amsthm}
\usepackage{amsmath}
\usepackage{graphicx}
\usepackage[all,2cell]{xy}
\usepackage{verbatim}
\usepackage{tikz}
\usepackage{tikz-cd}
\usepackage{hyperref}
\usepackage{array}
\usepackage[capitalise]{cleveref}
\usepackage{epsfig,graphicx}
\usepackage{tikz}
\usepackage{tikz-cd}
\usepackage{float}
\usepackage{capt-of}
\tikzset{node distance=2cm, auto}
\usepackage[vcentermath]{youngtab}
\usepackage{ytableau}
\usepackage{hyperref}
\usepackage{etoolbox}
\usetikzlibrary{graphs}
\usepackage{listings}
\usetikzlibrary{backgrounds}
\usetikzlibrary{arrows,decorations.markings}
\tikzstyle{vertex}=[circle, draw, inner sep=0pt, minimum size=6pt]

\usetikzlibrary{matrix,arrows,decorations.pathmorphing}
\usepackage{mathrsfs}
\usepackage{caption}
\numberwithin{equation}{section}
\newtheorem*{theorem*}{Theorem}
\newtheorem*{corollary*}{\bf Corollary}
\newtheorem*{remark*}{\bf Remark}
\usepackage{lipsum}
\usepackage{lineno}
\newtheorem{theorem}{Theorem}[section]
\newtheorem{corollary}[theorem]{Corollary}

\newtheorem{example}[theorem]{Example}
\newtheorem{lemma}[theorem]{Lemma}

\newtheorem{proposition}[theorem]{Proposition}
\newtheorem{remark}[theorem]{Remark}

\newcommand{\eat}[1]{}

\title{on Torus quotients of Schubert varieties in Orthogonal Grassmannian}
\author[A. Nayek]{Arpita Nayek}
\address{Arpita Nayek\\
	Department of Mathematics\\ IIT Bombay, Powai, Mumbai\\ 400076, India.}
\email{arpitan@math.iitb.ac.in}
\author[P. Saha]{Pinakinath Saha}
\address{Pinakinath Saha\\
	Department of Mathematics\\ IIT Bombay, Powai, Mumbai \\ 400076, India.}
\email{psaha@math.iitb.ac.in}
\subjclass[2010]{14M15}

\begin{document}
	\begin{abstract}
		Let $G=Spin(8n, \mathbb{C})(n\ge 1)$ and $T_{G}$ be a maximal torus of $G.$ Let $P^{\alpha_{4n}}(\supset T_{G})$ be the maximal parabolic subgroup of $G$ corresponding to the simple root $\alpha_{4n}.$ Let $X$ be a Schubert variety in $G/P^{\alpha_{4n}}$ admitting semi-stable point with respect to the $T$-linearized very ample line bundle $\mathcal{L}(2\omega_{4n}).$ Let $R=\bigoplus_{k \in \mathbb{Z}_{\geq 0}}R_k,$ where $R_k=H^{0}(X, \mathcal{L}^{\otimes k}(2\omega_{4n}))^{T_{G}}.$ In this article, we prove that for $n=1$ and $X=G/P^{\alpha_4},$ the graded $\mathbb{C}$-algebra $R$ is generated by $R_1.$ As a consequence, we prove that the GIT quotient of $G/P^{\alpha_{4}}$ is projectively normal with respect to the descent of the $T_{G}$-linearized very ample line bundle $\mathcal{L}(2\omega_{4})$ and is isomorphic to the projective space $(\mathbb{P}^{2},\mathcal{O}_{\mathbb{P}^{2}}(1))$ as a polarized variety. Further,  we prove that $R$ is generated by $R_1$ and $R_2$ for some Schubert varieties in $G/P^{\alpha_{4n}}$ (for $n \geq 2$). As a consequence, we prove that the GIT quotient of those Schubert varieties are projectively normal with respect to the descent of the $T_G$-linearized very ample line bundle $\mathcal{L}(4\omega_{4n}).$  
		
		Moreover, for $G = Spin(2n,\mathbb{C})(n \ge 4)$ (respectively, $G=Sp(2n, \mathbb{C}) (n\ge 2)$) and a maximal torus $T_G$ of $G,$  we prove that the GIT quotient of $G/P^{\alpha_{1}}$ is projectively normal with respect to the descent of the $T_G$-linearized very ample line bundle $\mathcal{L}(2\omega_{1})$ and is isomorphic to the projective space $(\mathbb{P}^{n-2},\mathcal{O}_{\mathbb{P}^{n-2}}(1))$ (respectively, $(\mathbb{P}^{n-1},\mathcal{O}_{\mathbb{P}^{n-1}}(1))$ as a polarized variety.

	\end{abstract}
	
	\keywords{Orthogonal Grassmannian, GIT-quotient, Line bundle, Semi-stable point}
	\subjclass[2010]{14M15}	
	\maketitle
	\section{Introduction}\label{section1}
	
	Let $V= \mathbb{C}^{2n}$ together with a non-degenerate symmetric bilinear form $(-,-).$ Let $\{e_{1},\ldots , e_{2n} \}$ be the standard basis of $V.$ Let $E=\begin{pmatrix}
		0& J\\
		J&0\\
	\end{pmatrix}$ be the matrix of
	the form $(-,- )$ with respect to the standard basis $\{e_{1},\ldots , e_{2n} \},$ where
	
	$J=\begin{pmatrix}
		0&0&\cdots& 0& 1\\
		0&0&\cdots& 1& 0\\
		\vdots&\vdots &\cdots&\vdots &\vdots\\
		0&1&\cdots& 0& 0\\
		1&0&\cdots& 0& 0\\
	\end{pmatrix}$
	of size $n\times n.$ 
	
	We may realize $G=SO(V)$ as the fixed
	point set ${H}^{\sigma},$  where $H=SL(V)$ and $\sigma : H\longrightarrow H$ is given by $$\sigma(A)= E(A^{t})^{-1}E,$$
	where $A\in H$ and $A^{t}$ denote transpose of $A.$
	
	Let $T_{H}$ (respectively, $B_{H}$) be the maximal torus in $H$ consisting of diagonal matrices (respectively, the Borel subgroup in $H$ consisting of upper triangular matrices). We see easily that $T_{H},$ $B_{H}$ are stable under $\sigma.$ Let $T_{G}=T_{H}^{\sigma},$ and $B_{G}=B_{H}^{\sigma}.$ Then $T_{G}$ is a maximal torus in $G$ and $B_{G}$ is a Borel subgroup in $G.$

	Let $N_{G}(T_{G})$ (respectively, $N_{H}(T_{H})$) be the normalizer in $G$ (respectively, $H$) of $T_{G}$ (respectively, $T_{H}$). Then we have $N_{G}(T_{G})\subseteq N_{H}(T_{H}).$ Further, $N_{H}(T_{H})$ is stable under $\sigma,$ and we have $N_{G}(T_{G})={N_{H}(T_{H})}^{\sigma}.$ Hence, we have a map $$N_{G}(T_{G})/T_{G}\longrightarrow N_{H}(T_{H})/T_{H}.$$
	Thus, we obtain a homomorphism 
	$$W_{G}\longrightarrow W_{H},$$
	where $W_{G},$ $W_{H}$ denote the Weyl groups of $G,$ $H$ respectively (with respect to $T_{G},$  $T_{H}$
	respectively). 
	
	The Weyl group $W_{H}$ is identified with $S_{2n},$ the symmetric group on the $2n$ letters $\{1,2,\ldots,2n\}.$
	Further, $\sigma$ induces an involution $W_{H}\longrightarrow W_{H}$ given by
	$$w=(a_{1},\ldots , a_{2n})\mapsto \sigma(w)=(c_{1},\ldots , c_{2n}),$$
	where $w=(a_{1},\ldots , a_{2n})$ is written in one line notation of the permutation and $c_{i} = 2n + 1-a_{2n+1-i}.$ Thus, we have $W_{G} =\{ w \in W_{H}^{\sigma}: ~w~ \text {is an even permutation in}~ W_{H}  \}.$ In one line notation, we have
	$W_{G}=\{(a_{1},\ldots,  a_{2n}) \in  S_{2n} : a_{i}=2n+ 1-a_{2n+1-i}, 1 \le i \le 2n \text{ and } ~m_{w}~ \text {is even}\},$
	where $m_{w}=\#\{ i\in \{1,\ldots, n\}: a_{i} > n\}.$ Thus, $w=(a_{1},\ldots, a_{2n})\in  W_{G}$ is known once $(a_{1}, \ldots, a_{n})$ is known.

	Let $X(T_{H})$(respectively, $X(T_{G})$) be the group of characters of $T_{H}$ (respectively, $T_{G}$). For $1\le i\le 2n,$ the elements $\epsilon_{i}\in X(T_{H})$ are defined by 
	$$\epsilon_{i}(t)=t_i,$$ where $t=\begin{pmatrix}t_1&0&\cdots&0&0\\
		0&t_2&\cdots&0&0\\
		\vdots&\vdots&\cdots &\vdots&\vdots\\
		0&0&\cdots &t_{2n-1}&0\\
		0&0&\cdots&0&t_{2n}\\
		
	\end{pmatrix}\in T_{H}.$ 
	
	Then the set $R_{H}=\{\epsilon_{i}-\epsilon_j: i\neq j\}$ can be identified with the set of roots of $H$ with respect to $T_{H}$ and that $R_{H}^{+}=\{\epsilon_{i}-\epsilon_{j}: 1\le i<j\le 2n\}$ is the set of positive roots with respect to $B_{H}.$ For $1\le i\le 2n-1,$ we denote $\beta_i=\epsilon_{i}-\epsilon_{i+1}.$ Then $S_{H}=\{\beta_{i}: 1\le i\le 2n-1\}$ is the set of simple roots of $R_{H}.$ 
	
	Note that $\sigma$ induces an involution 
	\begin{center}$\sigma: X(T_{H})\longrightarrow X(T_{H}),$\end{center} on $X(T_{H})$ defined by \begin{center}$\chi\mapsto\sigma(\chi),$\end{center} where $\sigma(\chi)(t)=\chi(\sigma(t))$ for $t \in T_{H}.$ 
	
	Then we have $\sigma(\epsilon_{i})=-\epsilon_{2n+1-i}$ for $1\le i\le 2n.$ Since $T_{G}\subset T_{H},$ there is a surjective map $\varphi: X(T_{H})\longrightarrow X(T_{G}),$ defined by $$\varphi(\epsilon_{i})=-\varphi(\epsilon_{2n+1-i})$$ for $1\le i\le 2n.$

	Then $\sigma$ leaves $R_{H}$ (respectively, $R_{H}^{+}$) stable. Let $R_G$ (respectively, $R_G^{+}$) be the set of roots (respectively, positive roots) of $G$ with respect to $T_{G}$ (respectively, $B_G$). Then from the explicit nature of the adjoint representation of $G$ on $\mathfrak{g}$ (Lie algebra of $G$) it follows that $R_{G}=\varphi(R_{H}\setminus R_{H}^{\sigma})$ and $R_{G}^{+}=\varphi(R_{H}^{+}\setminus {R_{H}^{+}}^{\sigma}).$ In particular, we have  $R_G$(respectively, $R_{G}^{+}$) can be  identified with the orbit space of $R_{H}$ (respectively, $R_{H}^{+}$) under the action of $\sigma$ minus the fixed point set under $\sigma.$ We see now that $R^{+}(G)$ can be identified with the subset $\{\epsilon_{i}-\epsilon_j: 1\le i<j\le n\}\cup \{\epsilon_{i} + \epsilon_{j}:1\le i<j\le n\}$ of $X(H)$ through $\varphi.$ For $1\le i\le n-1,$ we denote $\alpha_{i}=\epsilon_{i}-\epsilon_{i+1}$ and $\alpha_{n}=\epsilon_{n-1}+\epsilon_{n}.$ Then $S_{G}=\{\alpha_{i}: 1\le i\le n\}$ is the set of simple roots of $R_{G}.$ 
	
	Let $Q^{\beta_{n}}$ (respectively, $P^{\alpha_{n}}$) be the maximal parabolic subgroup of $H$ (respectively, $G$) corresponding to the simple root $\beta_{n}$ (respectively, $\alpha_{n}$). Let $W_{Q^{\beta_{n}}}$ (respectively, $W_{P^{\alpha_{n}}}$) be the Weyl group of the $Q^{\beta_{n}}$ (respectively, $P^{\alpha_{n}}$). Then the Schubert varieties in $H/Q^{\beta_{n}}$ (respectively, $G/P^{\alpha_{n}}$) are parameterized by the minimal coset representatives of $W_{H}/W_{Q^{\beta_{n}}}$ (respectively, $W_{G}/W_{P^{\alpha_{n}}}$). We denote the set of all minimal coset representatives of $W_{H}/W_{Q^{\beta_{n}}}$ (respectively, $W_{G}/W_{P^{\alpha_{n}}}$) by $W^{Q^{\beta_{n}}}$ (respectively, $W^{P^{\alpha_{n}}}$).
	
	For $v\in W_{H}^{Q^{\beta_{n}}},$ let $X_{H}(v):=\overline{B_{H}vQ^{\beta_{n}}/Q^{\beta_{n}}}$  denote the Schubert variety in $H/Q^{\beta_{n}}$  corresponding to $v.$ For $w\in {W_{G}}^{P^{\alpha_{n}}},$ let $X_{G}(w):=\overline{B_{G}wP^{\alpha_{n}}/P^{\alpha_{n}}}$ denote the Schubert variety in $G/P^{\alpha_{n}}$ corresponding to $w.$ 
	
	Let $\omega'_{n}$ (respectively, $\omega_{n}$) be the fundamental weight corresponding to the simple root $\beta_{n}$ (respectively, $\alpha_{n}$) in $H$ (respectively, in $G$). Note that the very ample generator of the Picard group of $H/Q^{\beta_{n}}$ (respectively, of $G/P^{\alpha_{n}}$) is the line bundle $\mathcal{L}(\omega'_{n})$ (respectively, $\mathcal{L}(\omega_n)$). Further, we note that the restriction of the very ample line bundle $\mathcal{L}(\omega'_{n})$ via the inclusion $G/P^{\alpha_{n}}\subset H/Q^{\alpha_{n}}$ is the very ample line bundle $\mathcal{L}(2\omega_{n}).$

	In \cite{NS}, the GIT quotients of the Schubert varieties in $H/Q^{\beta_{n}}$ for the action of $T_{H}$ is considered. In \cite{NS}, we have shown that for $n\ge 3,$ there exists a Schubert variety $X_{H}(v)$ in $H/Q^{\beta_{n}}$ admitting semi-stable point such that the GIT quotient of $X_{H}(v)$ is not projectively normal with respect to the descent of the $T_{H}$-linearized very ample line bundle $\mathcal{L}(2\omega'_{n}).$ As a consequence, we have concluded that for any Schubert variety $X_{H}(w)$ in $H/Q^{\beta_{n}}$ containing $X_{H}(v),$ the GIT quotient of $X_{H}(w)$ is not projectively normal with respect to the descent of the $T_H$-linearized very ample line bundle $\mathcal{L}(2\omega'_{n}).$ In particular, the GIT quotient of $H/Q^{\beta_{n}}(n \ge 3)$ is not projectively normal with respect to the descent of the $T_H$-linearized very ample line bundle $\mathcal{L}(2\omega'_{n}).$
	On the other hand, it is easy to see that for $n=1$ or $2,$ GIT quotient of $H/Q^{\beta_{n}}$ is projectively normal with respect to the descent of the $T_H$-linearized very ample line bundle $\mathcal{L}(2\omega'_{n}).$ 
	
	In this article, we make an attempt to study the GIT quotients of the Schubert varieties in $G/P^{\alpha_{4n}}$ for the action of $T_{G},$ where $G=Spin(8n,\mathbb{C}).$ Let $X_{G}(w)$ be a Schubert variety in $G/P^{\alpha_{4n}}$ admitting semi-stable point with respect to the $T_{G}$-linearized very ample line bundle $\mathcal{L}(2\omega_{n}).$ Let $R=\bigoplus_{k\in \mathbb{Z}_{\ge 0}} R_{k},$ where $R_{k}=H^0(X_{G}(w), \mathcal{L}^{\otimes k}(2\omega_{n}))^{T_G}.$ Note that $R$ is a graded $\mathbb{C}$-algebra.
	
	Then it is interesting to ask the following questions:
	\begin{itemize}
		\item [(1)] What is the degree bounds of the generators of the graded $\mathbb{C}$-algebra $R?$
	\end{itemize}
	\begin{itemize}
		\item [(2)] Is the graded $\mathbb{C}$-subalgebra $\bigoplus_{k\in \mathbb{Z}_{\ge0}}R_{2k}$ of $R$ generated by $R_{2}?$
	\end{itemize} 
	
	In this article, we give a partial answer to the above question. The main theorems of this article are the following:
	\begin{theorem}\label{thm1.1}
		Let $G=Spin(8,\mathbb{C}).$ Then the GIT quotient of $G/P^{\alpha_{4}}$ is projectively normal with respect to the descent of the $T_G$-linearized very ample line bundle $\mathcal{L}(2\omega_{4})$ and is isomorphic to the projective space $(\mathbb{P}^{2}, \mathcal{O}_{\mathbb{P}^2}(1))$ as a polarized variety. 
	\end{theorem}
	
	\begin{proposition}\label{prop1}
		Let $G=Spin(2n,\mathbb{C})(n\ge 4).$ Then the GIT quotient of $G/P^{\alpha_{1}}$ is projectively normal with respect to the descent of the $T_G$-linearized very ample line bundle $\mathcal{L}(2\omega_{1})$ and is isomorphic to the projective space $(\mathbb{P}^{n-2},\mathcal{O}_{\mathbb{P}^{n-2}}(1))$ as a polarized variety.
	\end{proposition}
	
	\begin{proposition}\label{prop2}
		Let $G=Sp(2n,\mathbb{C})(n\ge 2)$ and $T_G$ be a maximal torus of $G.$ Then the GIT quotient of $G/P^{\alpha_{1}}$ is projectively normal with respect to the descent of the $T_G$-linearized very ample line bundle $\mathcal{L}(2\omega_{1})$ and is isomorphic to the projective space $(\mathbb{P}^{n-1},\mathcal{O}_{\mathbb{P}^{n-1}}(1))$ as a polarized variety.
	\end{proposition}

	The layout of the article is as follows. In Section 2, we recall some preliminaries on algebraic groups, Lie algebras, Standard Monomial Theory and Geometric Invariant Theory. In Section 3, we prove \cref{thm1.1}. In Section 4,  we prove that graded $\mathbb{C}$-algebra $R$ is generated by $R_{1}$ and $R_{2}$ for some Schubert varieties $X_{G}(w)$ in $G/P^{\alpha_{4n}}(n\ge 2)$ admitting semi-stable point with respect to the $T_{G}$-linearized very ample line bundle $\mathcal{L}(2\omega_{4n}).$ As a consequence, we prove  projective normality results for those Schubert varieties with respect to the descent of the $T_{G}$-linearized very ample line bundle $\mathcal{L}(4\omega_{4n}).$ In Section 5, we prove \cref{prop1} and \cref{prop2}.

	\section{Notation and Preliminaries}\label{section2}
	In this section, we set up some notation and recall some preliminaries. We refer to \cite{Hum1}, \cite{Hum2}, \cite{Jan}, \cite{LS}, \cite{Mumford}, \cite{New} for preliminaries in algebraic groups, Lie algebras, Standard Monomial Theory, and Geometric Invariant Theory.

	Let $G, T_G, B_G, R_G$ and $R_G^{+}$ be as in the previous section. 
	
	Let $\mathfrak{g}$ be the Lie algebra of $G.$ Let $\mathfrak{h}_{\mathfrak{g}} \subset \mathfrak{g}$ be the Lie algebra of $T_{G}$ and $\mathfrak{b}_{\mathfrak{g}}$ be the Lie algebra of $B_{G}$. Let $Y(T_{G})$ denote the group of one-parameter subgroups of $T_{G}.$

	Let $E_1:=X(T_{G}) \otimes \mathbb{R},$ $E_2:=Y(T_{G}) \otimes \mathbb{R}$ and $\langle - , -\rangle: E_1 \times E_2 \longrightarrow \mathbb{R}$ be the canonical non-degenerate bilinear form. Let $\bar{C}:=\{\lambda \in E_2 : \langle \alpha, \lambda \rangle \geq 0 \text{ for all } \alpha \in R_{G}^{+}\}.$ Note that for each $\alpha \in R_{G},$ there is a homomorphism $\phi_{\alpha}:SL(2,\mathbb{C}) \longrightarrow G.$ We define $\check{\alpha}:\mathbb{G}_m \longrightarrow G$ by $\check{\alpha}(t)=\phi_{\alpha}(\begin{pmatrix}
		t & 0\\
		0 & t^{-1}
	\end{pmatrix}).$ We also have $s_{\alpha}(\chi)=\chi - \langle \chi, \check{\alpha} \rangle \alpha$ for all $\alpha \in R_{G}$ and $\chi \in E_1.$ 
	
	Let $\{\omega_i: i= 1,2, \ldots, n\} \subset E_1$ be the fundamental weights; i.e. $\langle \omega_i, \check{\alpha_j} \rangle=\delta_{ij}$ for all $i,j=1, 2, \ldots, n.$ There is a natural partial order $\leq$ on $X(T_{G})$ defined by $\psi\le \chi $ if and only if $\chi -\psi$ is a non-negative integral linear combination of simple roots.

	Let $\mathcal{L}$ be a $T_{G}$-linealized ample line bundle on $G/P^{\alpha_{n}}.$ We also denote the restriction of the line bundle $\mathcal{L}$ on $X_{G}(w)$ by $\mathcal{L}.$ 
	
	A point $p \in X_{G}(w)$ is said to be a semi-stable point with respect to the $T_{G}$-linearized line bundle $\mathcal{L}$ if there is a $T_{G}$-invariant section $s \in H^0(X_{G}(w),\mathcal{L}^{\otimes m})$ for some positive integer $m$ such that $s(p)\neq 0.$ We denote the set of all semi-stable points of $X_{G}(w)$ with respect to $\mathcal{L}$ by $X_{G}(w)^{ss}_{T_{G}}(\mathcal{L}).$ 
	
	A point $p$ in $X_{G}(w)^{ss}_{T_{G}}(\mathcal{L})$ is said to be a stable point if $T_{G}$-orbit of $p$ is closed in $X_{G}(w)^{ss}_{T_{G}}(\mathcal{L})$ and stabilizer of $p$ in $T_{G}$ is finite. We denote the set of all stable points of $X_{G}(w)$ with respect to $\mathcal{L}$ by $X_G(w)^{s}_{T_{G}}(\mathcal{L}).$ 
	
	Now, we recall the definition of projective normality of a projective variety.
	
	Let $X$ be a projective variety in $\mathbb{P}^{m}.$ We denote by $\hat{X}$ the affine cone of $X.$ Then $X$ is said to be projectively normal if $\hat{X}$ is normal (see \cite[Chapter I, Exercise 3.18, p.23]{R}). For the practical purpose we need the following fact about projective normality of a polarized variety.
	
	A polarized variety $(X,\mathcal{L}),$ where $\mathcal{L}$ is a very-ample line bundle is said to be projectively normal if its homogeneous coordinate ring $\bigoplus_{k\in \mathbb{Z}_{\geq 0}} H^{0}(X,\mathcal{L}^{\otimes k})$ is integrally closed and is generated as a $\mathbb{C}$-algebra by $H^0(X,\mathcal{L})$ (see \cite[Chapter II, Exercise 5.14, p.126]{R}).

	Let us recall some well-known properties of $Spin(2n, \mathbb{C}).$ They can be found in the standard references such as \cite{LS}.
	
	Let us denote the length of $w$ as an element of  $W_{H}$ (respectively, $W_{G}$) by $\ell(w, W_{H})$ (respectively, $\ell(w, W_{G})$). Then for $w=(a_{1},\ldots , a_{2n} )\in W_{G},$ we have $\ell(w, W_{G})=\frac{1}{2}(\ell(w,W_{H})- m_{w}).$ Further, the Bruhat-Chevalley order $~``\le"~$ on $W^{P^{\alpha_{n}}}$ is induced by the Bruhat-Chevalley order $~``\le"~$ on $W^{Q^{\beta_{n}}}.$
	
	For $1\le d\le 2n,$ define the set $I_{d,2n}:=\{\underline{i}=(i_{1},\ldots, i_{d}):1 \le i_1 < \cdots < i_d \le 2n \}.$ Then there is a partial order $"\le "$ on $I_{d,2n}$ defined by $\underline{i}\le \underline{j}$ if and only if  $i_{t}\le j_{t}$ for all $1\le t\le d.$ Further, there is a natural order preserving indentification of $W^{Q^{\beta_{n}}}$
	with $I_{n, 2n},$ the correspondence is given by $w\in W^{Q^{\beta_{n}}}$ mapping to $(w(1),\ldots ,w(n)).$

	Let $Z$ be the subgroup of $H$ consisting of matrices of the form $\begin{pmatrix} Id_{n} & 0\\
		Y& Id_{n}
	\end{pmatrix},$
	where $Y\in M_{n}$ the $n\times n$ matrices with entries in $\mathbb{C}.$ Then the canonical morphism $\psi_{H}:H\longrightarrow H/Q^{\beta_{n}}$ induces a morphism $\psi_{H}: Z\longrightarrow H/Q^{\beta_{n}}.$ Further, $\psi_{H}$ is an open immersion and $\psi_{H}$ is identified with the opposite big cell $O_{H}^{-}$ in $H/Q^{\beta_{n}}.$ Moreover, we have $O_{H}^{-}$ is $\sigma$-stable and we have an identification $Sk M_{n}\simeq (O_{H}^{-})^{\sigma},$ where $SkM_{n}$ is the space of skew symmetric $n\times n$ matrices. Further, $\psi_{H}$ induces an inclusion $\psi_{H}:G/P^{\alpha_{n}}\longrightarrow H/Q^{\beta_{n}}.$

	Note that $\mathcal{L}(\omega'_{n})$ is the tautological line bundle on $H/Q^{\beta_{n}}$($\simeq G_{n, 2n},$ the Grassmannian of $n$-dimensional subspaces of $V=\mathbb{C}^{2n}$). Then $H^0(H/Q^{\beta_{n}}, \mathcal{L}(\omega'_{n}))$ is the dual of the irreducible $H$-module with highest weight $\omega_{n}.$ Let $Id=(1,..., n)\in I_{n,2n}$ denote $f:= p_{Id}.$ Then we have
	$f$ is a lowest weight vector in $H^0(H/Q^{\beta_{n}}, \mathcal{L}(\omega'_{n}))$ and $O_{H}^{-}$ is the principal open set of $(H/Q^{\beta_{n}})_{f}.$

	Let $Y\in M_{n}.$ Given two $s$-tuples $A:=(a_{1},...,a_{s}),$ and $B:=(b_{1}, ..., b_{s})$ in $I_{s, n},$ $s\le n.$  Let $p(A, B)(Y)$ be the $s$-minors of $Y$ with row and column indices given by $A,B$ respectively. Consider the following identification 
	\begin{center}
		$M_{n}\simeq \bigg\{\begin{pmatrix}
			Id_{n}\\JY
		\end{pmatrix}: Y\in M_{n} \bigg\}\simeq {O_{H}}^{-},$
	\end{center} 
	where $J$ is the anti-diagonal matrix $(1,...,1)$ of size $n\times n.$
	
	Let $\underline{i}:=(i_{1},...,i_{n})\in I_{n,2n},$ and $p_{\underline{i}}$ be the associated Pl\"ucker co-ordinate on $G_{n,2n}(\simeq H/Q^{\beta_{n}}).$ Let $f_{\underline{i}}$ denote the restriction of $p_{\underline{i}}$ to $M_{n}$ under the above identification. Then $f_{\underline{i}}(Y)=p({\underline{i}}(A), {\underline{i}}(B)),$ where ${\underline{i}}(A), {\underline{i}}(B)$ are given as follows:
	
	Let $r$ be such that $i_{r}\le n, i_{r+1}>n.$ Let $s=n-r.$ Then $\underline{i}(A)$ is the $s$-tuple given by $(2n+1-i_{n},...,2n+1-i_{r+1}),$ while $\underline{i}(B)$ is the $s$-tuple given by the complement of $(i_{1},...,i_{r})$ in $(1,...,n).$ Then the pair $(\underline{i}(A), \underline{i}(B))$ is the canonical dual pair associated to $\underline{i}.$
	
	Let $\mathcal{L}$ be the restriction of $\mathcal{L}(\omega'_n)$ to $G/P^{\alpha_{n}}.$ For $\underline{i}\in I_{n,2n},$ let $p'_{\underline{i}}$ be the restriction of $p_{\underline{i}}$ to $G/P^{\alpha_{n}}.$ We note that $p'_{\underline{i}}\in H^0(G/P^{\alpha_{n}}, \mathcal{L}).$ Let $p(\underline{i}(A), \underline{i}(B))$ be the restriction of $p_{\underline{i}}$ to $M_{n},$ $M_{n}$ being identified with the opposite cell $O^{-}_{H}$ as above. Let $p'(\underline{i}(A), \underline{i}(B))$ be the restriction of $p(\underline{i}(A), \underline{i}(B))$ to $SkM_{n}.$ Let $f=p_{Id}$ be a lowest weight vector in $H^0(H/Q^{\beta_{n}}, \mathcal{L}(\omega'_{n}))$ and $f'$ the restriction of $f$ to $G/P^{\alpha_{n}}.$ Let $Y\in SkM_{n}.$ Consider $\underline{i}\in I_{n,2n}$ such that $\underline{i}(A)=\underline{i}(B).$ Then $p'(\underline{i}(A), \underline{i}(A))(Y)$ is a principal minor of the skew symmetric matrix $Y.$ Hence, $p'(\underline{i}(A), \underline{i}(A))(Y)$ is a square denoting $q_{\underline{i}}(Y)$ the corresponding Pfaffian (i.e., $q_{\underline{i}}^{2}(Y)=p'(\underline{i}(A), \underline{i}(A))(Y)$), we obtain a regular function 
	$q_{\underline{i}}: SkM_{n}\longrightarrow \mathbb{C}.$ Let $r=\#\underline{i}(A).$ Then we have that $q_{\underline{i}}$ is non zero if only if $r$ is even (since the determinant of a skew symmetric $r\times r$ matrix is zero, if $r$ is odd). Thus, for $\underline{i}\in I_{n, 2n}$ such that $\underline{i}(A)=\underline{i}(B),$ we have a regular function $q_{\underline{i}}: SkM_{n}\longrightarrow \mathbb{C}$ such that $q^{2}_{\underline{i}}=p'(\underline{i}(A), \underline{i}(A)).$ Further, $q_{\underline{i}}$ is non-zero if and only if $\#\underline{i}(A)$ is even. Also, we have $\mathcal{L}^{2}=\mathcal{L}(\omega'_{n})$  and the identification $Z^{\sigma}\simeq O^{-}_{G}\simeq SkM_{n}.$ Thus, the opposite cell in $H/Q^{\beta_{n}}$ induces the opposite cell in $G/P^{\alpha_{n}}.$

	Now, we recall the following theorems due to Shrawan Kumar describes which $T_G$-linearized very ample line bundle $\mathcal{L}_{\lambda}$ on $G/P^{\alpha_{4n}}$ descends to the GIT quotient $T_{G}\backslash\backslash(G/P^{\alpha_{4n}})_{T_{G}}^{ss}(\mathcal{L}_{\lambda})$ (see \cite[Theorem 3.10, p.764]{Kum}), where $G=Spin(8n, \mathbb{C})$ and $\mathcal{L}_{\lambda}$ is the line bundle associated to the regular dominant character $\lambda$ of $P^{\alpha_{4n}}.$ We note that Kumar's result in \cite{Kum} is more
	general than what is presented here.
	
	\begin{theorem}\label{thm2.1}
		$\mathcal{L}_{\lambda}$ descends to a line bundle on the GIT quotient $T_{G}\backslash\backslash(G/P^{\alpha_{4n}})_{T_{G}}^{ss}(\mathcal{L}_{\lambda})$ if and only if
		\begin{itemize}
			\item[(i)] for $G=Spin(8, \mathbb{C}):$ $\lambda=2m\omega_4$ for some $m\in \mathbb{N},$
			\item[(ii)] for $G=Spin(8n, \mathbb{C})$ $(n \geq 2):$ $\lambda=4m\omega_{4n}$ for some $m\in \mathbb{N}.$
		\end{itemize}
	\end{theorem}
	\begin{theorem}\label{thm2.2}
		\begin{itemize}
			\item[(i)] Assume that $G=Spin(2n, \mathbb{C}) (n \geq 4)$ and $P^{\alpha_{1}}$ is the maximal parabolic subgroup of $G$ corresponding to the simple root $\alpha_{1}$ containing maximal torus $T_{G}$ of $G.$ Then $\mathcal{L}_{\lambda}$ descends to a line bundle on the GIT quotient $T_{G}\backslash\backslash(G/P^{\alpha_{1}})_{T_{G}}^{ss}(\mathcal{L}_{\lambda})$ if and only if $\lambda=2m\omega_1$ for some $m\in \mathbb{N},$
			
			\item[(ii)] Assume that $G=Sp(2n, \mathbb{C})$ $(n \geq 2)$ and $P^{\alpha_{1}}$ is the maximal parabolic subgroup of $G$ corresponding to the simple root $\alpha_{1}$ containing  maximal torus $T_{G}$ of $G.$ Then $\mathcal{L}_{\lambda}$ descends to a line bundle on the GIT quotient $T_{G}\backslash\backslash(G/P^{\alpha_{1}})_{T_{G}}^{ss}(\mathcal{L}_{\lambda})$ if and only if $\lambda=2m\omega_1$ for some $m\in \mathbb{N}.$
		\end{itemize}
	\end{theorem}

	Now, we recall some definitions and facts on standard monomial for $Spin(2n, \mathbb{C})$ and $Sp(2n, \mathbb{C})$ from \cite[Appendix p. 363-365]{L}. We will present here simplified version of the definition of Young tableau and standard Young tableau for $Spin(2n,\mathbb{C})$ (respectively, $Sp(2n, \mathbb{C})$) associated to the weight $\lambda=2m\omega_1$ and $\lambda=2m\omega_{n}$ (respectively, $\lambda=2m\omega_1$), where $m\in \mathbb{N}$ (for more general see \cite[Appendix p. 363-365]{L}).

	\subsection{ Young tableau for $G=Spin(2n, \mathbb{C})$ and $\lambda=2m\omega_1,$ where $m \in \mathbb{N}$}\label{subsection2.2} Define $p_i = 4m$ for $i=1$ and $p_i=0$ for $2 \leq i \leq n.$ Associated to $\lambda,$ we define a partition $p(\lambda):=(p_{1},p_{2},\ldots,p_{n}).$ Then a Young diagram of shape $p(\lambda) = (p_1, \ldots, p_n)$ associated to $\lambda$ is denoted by $\Gamma$ consists of $p_1$ boxes in the first column, $p_2$ boxes  in the second column and so on. Since $p_i=0$ for all $2 \leq i \leq n,$ it follows that the Young diagram $\Gamma$ associated to $\lambda$ is of shape $p(\lambda) = (p_1),$ i.e., $\Gamma$ consists of $p_1$ boxes in a single column.
	
	A Young tableau of shape $\lambda$ is a Young diagram filled with integers between $1$ and $2n.$ 
	
	A Young tableau is said to be standard if the entries along the column is non-decreasing from top to bottom.
	
	A standard Young tableau $\Gamma$ of shape $p(\lambda)$ is said to be $Spin(2n,\mathbb{C})$-standard if the following conditions hold:
	
	\begin{itemize}
		\item[(1)] Let $r_i$ be the $i$-th row of $\Gamma$ enumerated from the bottom row. For $1 \leq i \leq \frac{p_1}{2}-1,$ let $r_{2i}$ be equal to $k_1$ and $r_{2i+1}$ be equal to $l_1.$ If $m \leq k_1 \leq m+1$ and $m \leq l_1 \leq m+1,$ then one has $k_1 \equiv l_1 \text{ mod } 2.$
		\item[(2)] The pairs $(r_{2i-1}, r_{2i})$ are trivial admissible pair i.e., $r_{2i-1}=r_{2i}$ for $i=1, 2, \ldots, \frac{p_1}{2}.$ 
	\end{itemize} 
	
	Given a $Spin(2n,\mathbb{C})$-standard Young tableau $\Gamma$ of shape $\lambda,$ we can associate a section $p_{\Gamma}\in H^0(G/P^{\alpha_{1}}, \mathcal{L}^{\otimes m}(2\omega_1))^{T_{G}}$
	which is called standard monomial of degree $m.$
	
	\subsection{Young tableau for $G=Spin(2n, \mathbb{C})$ and $\lambda=2m\omega_n,$ where $m \in \mathbb{N}$}\label{subsection2.1} Define $p_i = 2m$ for $1 \leq i \leq n.$ Associated to $\lambda,$ we define a partition $p(\lambda):=(p_{1},p_{2},\ldots,p_{n}).$ Then a Young diagram of shape $p(\lambda) = (p_1, \ldots, p_n)$ associated to $\lambda$ is denoted by $\Gamma$ consists of $p_1$ boxes in the first column, $p_2$ boxes  in the second column and so on.
	
	A Young tableau of shape $\lambda$ is a Young diagram filled with integers between $1$ and $2n.$ 
	
	A Young tableau is said to be standard if the entries along any column is non-decreasing from top to bottom and along any row is strictly increasing from left to right.
	
	A Young tableau $\Gamma$ of shape $p(\lambda)$ is said to be $Spin(2n,\mathbb{C})$-standard if $\Gamma$ is standard and if $r= (i_{1},i_{2},\ldots,i_{n})$ is a row of length $n$ such that if $i_{j}$ is an entry of the row, then $2n + 1-i_{j}$ is not an entry of this row. The number of integers greater than $n$ in a row of $\Gamma$ is even. 
	
	Given a $Spin(2n,\mathbb{C})$-standard Young tableau $\Gamma$ of shape $\lambda,$ we can associate a section $p_{\Gamma}\in H^0(G/P^{\alpha_{n}}, \mathcal{L}^{\otimes m}(2\omega_n))^{T_{G}}$
	which is called standard monomial of degree $m.$
	
	\subsection{Young tableau for $G=Sp(2n, \mathbb{C})$ and $\lambda=2m\omega_1,$ where $m \in \mathbb{N}$}\label{subsection2.3} Define $p_i = 4m$ for $i=1$ and $p_i=0$ for $2 \leq i \leq n.$ Associated to $\lambda,$ we define a partition $p(\lambda):=(p_{1},p_{2},\ldots,p_{n}).$ Then a Young diagram of shape $p(\lambda) = (p_1, \ldots, p_n)$ associated to $\lambda$ is denoted by $\Gamma$ consists of $p_1$ boxes in the first column, $p_2$ boxes  in the second column and so on. Since $p_i=0$ for all $2 \leq i \leq n,$ it follows that the Young diagram $\Gamma$ associated to $\lambda$ is of shape $p(\lambda) = (p_1),$ i.e., $\Gamma$ consists of $p_1$ boxes in a single column.
	
	A Young tableau of shape $\lambda$ is a Young diagram filled with integers between $1$ and $2n.$ 
	
	A Young tableau is said to be standard if the entries along the column is non-decreasing from top to bottom.
	
	A standard Young tableau $\Gamma$ of shape $p(\lambda)$ is said to be $Sp(2n,\mathbb{C})$-standard if the pairs $(r_{2i-1}, r_{2i})$ are trivial admissible pairs i.e., $r_{2i-1}=r_{2i}$ for $i=1, 2, \ldots, \frac{p_1}{2}.$ 
	
	Given a $Sp(2n,\mathbb{C})$-standard Young tableau $\Gamma$ of shape $\lambda,$ we can associate a section $p_{\Gamma}\in H^0(G/P^{\alpha_{1}}, \mathcal{L}^{\otimes m}(2\omega_1))^{T_{G}}$	 which is called standard monomial of degree $m.$
	
	\begin{example}
		For $G=Spin(8, \mathbb{C}),$ let $\lambda=2\omega_{4}.$ Then $(p_{1}, p_{2}, p_{3}, p_{4})=(2,2,2,2)$ and the corresponding Young diagram is the following.
	\end{example} 
	\ytableausetup{centertableaux}
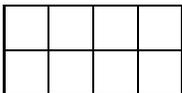
	\begin{ytableau}
		\none &  &  & & \\
		\none  &  & & &
	\end{ytableau}	
	
	\begin{remark}\label{zeroweight}
		For $1\le t\le 2n,$	let $c_{\Gamma}(t)$ denote the number of times $t$ appears in $\Gamma.$ 
		Then the weight of $\Gamma$ is defined by $\frac{1}{2}\sum_{j=1}^n(c_{\Gamma}(j)-c_{\Gamma}(2n+1-j))\epsilon_j$. A monomial $p_{\Gamma}\in H^0(G/P^{\alpha_{n}}, \mathcal{L}(\lambda))$ is $T_{G}$-invariant if and only if the weight of $\Gamma$ is zero. Therefore, $p_{\Gamma}$ is $T_{G}$-invariant if and only if $c_{\Gamma}(t)=c_{\Gamma}(2n+1-t)$ for all $1 \leq t \leq 2n$.
	\end{remark}
	
	\subsection{Straightening Law}
	\begin{theorem}
		Suppose that $p_{\beta_{1}}p_{\beta_{2}}$ is a quadratic monomial which is not standard. Then we have
		$p_{\beta_{1}}p_{\beta_{2}}=\sum\limits_{(\alpha)\in I^{2}}^{}\lambda_{(\alpha)}p_{\alpha_{1}}p_{\alpha_{2}};$ $\alpha_{1}\le \alpha_{2}$ and $\lambda_{(\alpha)}\neq0$ in $\mathbb{C},$
		where on the right side $(\alpha)$ runs over distinct element of $I^{2}$ of the form
		\begin{itemize}
			\item $\alpha_{1}\le \beta_{1},$ $\alpha_{1}\neq \beta_{1};$ $\alpha_{1}\le \beta_{2},$ $\alpha_{1}\neq \beta_{2}.$
			
			\item $\alpha_{2}\ge \beta_{1},$ $\alpha_{2}\neq \beta_{1};$ $\alpha_{2}\ge \beta_{2},$ $\alpha_{2}\neq \beta_{2}.$
			
		\end{itemize}
		
	\end{theorem}
	
	\begin{proof}
		See \cite[Corollary 1, p.144]{Ses}.
	\end{proof}
	The above theorem implies that given  a non-standard quadratic monomial $p_{\beta_{1}}p_{\beta_{2}}$ can be expressed as in the above as  a linear combination of standard monomials.
	
	To proceed further we recall some result on Pfaffian of a skew-symmetric matrix from \cite{DW} which we will use to determine the straightening law.
	
	The Pfaffian Pf$(A)$ of a skew-symmetric $n\times n$-matrix	
	\begin{equation}\label{eq2.1}
		A=\begin{pmatrix}
			0 & a_{12}&a _{13}&\cdots & a_{1n-1}&a_{1n}\\
			-a_{12} & 0&a _{23}&\cdots & a_{2n-1}&a_{2n}\\
			-a_{13}& -a_{23}&0&\cdots & a_{3n-1}&a_{3n} \\
			\vdots & \vdots&\vdots&\cdots & \vdots&\vdots \\
			-a_{1n-1}& -a_{2n-1}&-a_{3n-1}&\cdots & 0&a_{n-1n}\\
			-a_{1n}& -a_{2n}&-a _{3n}&\cdots & -a_{n-1n}&0\\
		\end{pmatrix}
	\end{equation} 
	with coefficients $a_{ij}(1\le i<j\le n)$ in $\mathbb{C}$ is defined to be Pf$(A)^{2}$=det$(A).$ Note that in case $n\equiv 1$ mod 2, we have Pf$(A)=0$ for all $A\in SkM_{n}.$ In case $n\equiv 0$ mod 2, one can describe the Pfaffian equivalently as a polynomial 
	\begin{center}
		Pf$(X_{12},X_{13},\ldots, X_{1n}, X_{23},\ldots, X_{2n},\ldots, X_{n-1n})\in \mathbb{Z}[X_{12}, X_{13},\ldots, X_{n-1n}]$
	\end{center} such that for each $A\in SkM_{n}$ of the form \ref{eq2.1} one has 
	
	\begin{center}
		Pf$(a_{12},a_{13},\ldots, a_{1n}, a_{23},\ldots, a_{2n},\ldots, a_{n-1n})^{2}$=det($A$),
	\end{center}
	normalized such that 
	
	\begin{equation}
		\text{Pf}(a_{12},a_{13},\ldots, a_{1n}, a_{23},\ldots, a_{2n},\ldots, a_{n-1n}):=1
	\end{equation}
	for $a_{ij}:=\bigg\{ \begin{matrix}  1, \text{~for~} j-i=\frac{n}{2}\\ 0 ,\text{~otherwise~}
	\end{matrix}.$
	
	To be more explicit, consider as above a skew-symmetric $n\times n$-matrix $A$ of the form \ref{eq2.1} with $n$ is even or odd and for each subset $I\subseteq \{1,2, \ldots, n\}$ of cardinality, say, $m.$ Let $A(I)$ denote the skew-symmetric  $m\times m$-matrix  one gets from $A$ by eliminating all rows and columns not indexed by indices from $I.$ Moreover, put 
	\begin{equation}
		P(I):=\text{Pf}(A(I))
	\end{equation} 
	for every $I\subseteq \{1,2,\ldots,n\},$ where as usual Pf($A(\phi)$):=1. Then the following result holds.
	
	\begin{theorem}\label{thm2.6}
		For any two subsets $I_{1}, I_{2}\subseteq \{1,2,\ldots, n\}$ of odd cardinality and elements $i_{1},i_{2},\ldots,i_{t}\in \{1,2,\ldots,n \}$ with $i_{1}<i_{2}<\cdots < i_{t}$ and $\{i_{1},\ldots, i_{t}\}=I_{1}\Delta I_{2}:=(I_{1}\setminus I_{2})\cup (I_{2}\setminus I_{1})$ one has \begin{equation}
			\sum_{\tau=1}^{t} (-1)^\tau\cdot P({I_{1}\Delta\{i_{\tau}\}})\cdot P({I_{2}\Delta\{i_{\tau}\}})=0.
		\end{equation}
	\end{theorem}
	\begin{proof}
		See \cite[Theorem 1, p.122]{DW}.
	\end{proof}
	
	\begin{remark}\label{rmk2.7}
		For $\underline{i}=(i_1, i_2, \ldots, i_n)\in I_{n,2n},$ we have $q_{\underline{i}}=P(\underline{i}(B))=\begin{ytableau}
			i_1 & i_2 & \cdots & i_n
		\end{ytableau}.$    
	\end{remark}
	
	\begin{remark}\label{rmk2.8}
		Let $I=\{i_1, i_2, \ldots, i_r\} \subset \{1, 2, \ldots, n\}$ be such that $i_1<i_2<\cdots <i_r.$ Let $\{j_{1},\ldots, j_{n-r}\}$ be the complements of $\{i_{1},\ldots, i_{r}\}$ in $\{1,2,\ldots, n\}$ such that $j_1<j_2<\ldots <j_{n-r}.$ Then we have $P(I)=q_{\underline{i}},$ where $\underline{i}=(j_1,j_2,\ldots, j_{n-r}, 2n+1-i_r, \ldots, 2n+1-i_1).$ 
	\end{remark}	
	
	Now, we work out straightening law for $Spin(8,\mathbb{C}).$ 
	
	\begin{example} For $G=Spin(8,\mathbb{C}),$ $T_{G}$-invariant tableau are $\ytableausetup{centertableaux}
		\begin{ytableau}
			1& 2 & 3 & 4 \\
			5 & 6 & 7& 8
		\end{ytableau},~$ $\ytableausetup{centertableaux}
		\begin{ytableau}
			1 & 2 & 5 & 6 \\
			3 & 4 & 7 & 8 
		\end{ytableau},~$ $\ytableausetup{centertableaux}
		\begin{ytableau}
			1 & 3 & 5 & 7 \\
			2 & 4 & 6 & 8
		\end{ytableau},~$  $\ytableausetup{centertableaux}
		\begin{ytableau}
			1& 4 & 6 & 7 \\
			2 & 3 & 5& 8
		\end{ytableau}.$ 
		Let $\underline{i_{1}}=(1,4,6,7)$ and $\underline{i_{2}}=(2, 3, 5,8).$ Then we have $\underline{i_{1}}(A)=\underline{i_{1}}(B)=(2,3)$ and $\underline{i_{2}}(A)=\underline{i_{2}}(A)=(1,4).$ Recall that the opposite big cell $U^-=SkM_{4},$ where 
		$SkM_{4}:=\Bigg\{Y=\begin{pmatrix}
			0&y_{12}&y_{13}&y_{14}\\
			-y_{12}&0&y_{23}&y_{24}\\
			-y_{13}&-y_{23}&0&y_{34}\\
			-y_{14}&-y_{24}&-y_{34}&0
		\end{pmatrix}: y_{ij} \in \mathbb{C} \text{~for all~} 1\le i<j\le 4\Bigg\}.$ 
		
		Let $I_{1}=\{1,2,4\}$ and $I_{2}=\{3\}.$ Then we have $I_{1}\Delta I_{2}=\{1,2,3,4\}.$ Thus by using Theorem \ref{thm2.6} we have 
		\begin{center}
			$P(I_{1}\Delta\{1\})\cdot P(I_{2}\Delta \{1\})-P(I_{1}\Delta\{2\})\cdot P(I_{2}\Delta \{2\})+P(I_{1}\Delta\{3\})\cdot P(I_{2}\Delta \{3\})-P(I_{1}\Delta\{4\})\cdot P(I_{2}\Delta \{4\})=0.$
		\end{center}
		
		Now note that by \cref{rmk2.8}, we have the following.
		\begin{itemize}
			\item $P(I_{1}\Delta\{1\})=q_{(1357)},$ $P(I_{2}\Delta\{1\})=q_{(2468)}.$
			\vspace{.2cm}
			
			\item	$P(I_{1}\Delta\{2\})=q_{(2358)},$ $P(I_{2}\Delta\{2\})=q_{(1467)}.$
			\vspace{.2cm}
			
			\item $P(I_{1}\Delta\{3\})=q_{(5678)},$ $P(I_{2}\Delta\{3\})=q_{(1234)}.$
			\vspace{.2cm}
			
			\item $P(I_{1}\Delta\{4\})=q_{(3478)},$ $P(I_{2}\Delta\{4\})=q_{(1256)}.$
		\end{itemize}

		Therefore, we have 
		\begin{center}
			$q_{(1,4,6,7)}q_{(2,3,5,8)}=q_{(1,2,3,4)}q_{(5,6,7,8)}-q_{(1,2,5,6)}q_{(3,4,7,8)}+q_{(1,3,5,7)}q_{(2,4,6,8)}.$
		\end{center}

		Hence, interchanging $q_{\underline{i}}$'s in Young tableau notation we have
		\begin{center}
			\begin{ytableau}
				1& 4 & 6 & 7 \\
				2 & 3 & 5& 8
			\end{ytableau}$~=\begin{ytableau}
				1& 2 & 3 & 4 \\
				5 & 6 & 7& 8
			\end{ytableau}$ $-~\ytableausetup{centertableaux}
			\begin{ytableau}
				1 & 2 & 5 & 6 \\
				3 & 4 & 7 & 8 
			\end{ytableau}$ $+~\ytableausetup{centertableaux}
			\begin{ytableau}
				1 & 3 & 5 & 7 \\
				2 & 4 & 6 & 8
			\end{ytableau}.$  
		\end{center}
		
	\end{example}
	
	\section{$G=Spin(8,\mathbb{C})$}\label{section4}
	
	Recall that by \cref{thm2.1}(i), $T_G$-linearized very ample line bundle $\mathcal{L}(2\omega_4)$ descends to a line bundle on the GIT quotient $T \backslash \backslash (G/P^{\alpha_4})^{ss}_{T_G}(\mathcal{L}(2\omega_4)).$ In this section, we prove that $T \backslash \backslash (G/P^{\alpha_4})^{ss}_{T_G}(\mathcal{L}(2\omega_4))$ is projectively normal and isomorphic to $(\mathbb{P}^2,\mathcal{O}_{\mathbb{P}^{2}}(1))$ as a polarized variety. 
	
	Consider $w=s_4(s_2s_3)(s_1s_2s_4).$ Note that $w$ is the longest element of $W^{P^{\alpha_{4}}},$ i.e., $X_G(w)=G/P^{\alpha_4}.$ Further, note that $ w(2\omega_4)<  0.$ Therefore,  by \cite[Lemma 3.1, p.276]{KNS}, $X_G(w)^s_{T_G}(\mathcal{L}(2\omega_4)) \neq \emptyset.$ In one line notation we have $w$ is equal to $(5,6,7,8).$ 
	
	Let $X= T_G \backslash \backslash (G/P^{\alpha_4})^{ss}_{T_G}(\mathcal{L}(2\omega_4)).$ Let $R=\bigoplus\limits_{k\in\mathbb{Z}_{\geq 0}}R_k,$ where $R_{k}=H^0(G/P^{\alpha_4}, \mathcal{L}^{\otimes k}(2\omega_{4}))^{T_G}.$ Then we have $X= Proj(R).$
	We note that $R_{k}$'s  are finite dimensional complex vector spaces.
	
	Let  $\Gamma_{1}=\ytableausetup{centertableaux}
	\begin{ytableau}
		1& 2 & 3 & 4 \\
		5 & 6 & 7& 8
	\end{ytableau},$ $\Gamma_{2} =\ytableausetup{centertableaux}
	\begin{ytableau}
		1 & 2 & 5 & 6 \\
		3 & 4 & 7 & 8 
	\end{ytableau},$ $\Gamma_{3}=\ytableausetup{centertableaux}
	\begin{ytableau}
		1 & 3 & 5 & 7 \\
		2 & 4 & 6 & 8
	\end{ytableau}$ and
	$\Gamma_{4}=\ytableausetup{centertableaux}
	\begin{ytableau}
		1 & 4 & 6 & 7 \\
		2 & 3 & 5 & 8
	\end{ytableau}.$ 
	
	Then $\{\Gamma_1, \Gamma_2, \Gamma_3\}$ forms a basis of $R_1.$
	\begin{lemma}\label{lemma3.1}
		The graded $\mathbb{C}$-algebra $R$ is generated by $R_1.$ 
	\end{lemma}
	
	\begin{proof}
		
		Let $\Gamma$ be a standard Young tableau corresponding to a standard monomial in $R_k$. It is enough to prove that there is a subtableau $\Gamma^{'}$ of $\Gamma$ such that the associated monomial $p_{\Gamma^{'}}$ lies in $R_1$.
		
		The Young tableau $\Gamma$ has $2k$ rows and $4$ columns with strictly increasing rows from left to right and non-decreasing columns from top to bottom. Since $\Gamma$ is $T_G$-invariant, by \cref{zeroweight} we have, \begin{equation} c_{\Gamma}(t)=c_{\Gamma}(9-t) \text{ for all } 1 \leq t \leq 8,\end{equation} where $c_{\Gamma}(t)$ denotes the number of times $t$ appears in $\Gamma$.
		
		Note that all possible rows in $\Gamma$ are $(1,2,3,4),$ $(1,3,5,7),$ $(1,2,5,6),$ $(1,4,6,7),$ $(2,3,5,8),$ $(2,4,6,8),$ $(3,4,7,8),$  $(5,6,7,8).$ 
		
		Assume that $(1,2,3,4),$ $(1,3,5,7),$ $(1,2,5,6),$ $(1,4,6,7),$ $(2,3,5,8),$ $(2,4,6,8),$ $(3,4,7,8)$ and $(5,6,7,8)$ appear in $\Gamma,$ $a_1,$ $a_2,$ $a_3,$ $a_4,$ $b_1,$ $b_2,$ $b_3$ and $b_4$ number of times respectively. 
		
		Since $c_{\Gamma}(1)=c_{\Gamma}(8),$ we have 
		\begin{equation}\label{eq3.2}
			a_1+a_2+a_3+a_4=b_1+b_2+b_3+b_4.
		\end{equation}
		
		Since $c_{\Gamma}(2)=c_{\Gamma}(7),$ we have 
		\begin{equation}\label{eq3.3}
			a_1+a_3+b_1+b_2=a_2+a_4+b_3+b_4.
		\end{equation} 
		
		Since $c_{\Gamma}(3)=c_{\Gamma}(6),$ we have 
		\begin{equation}\label{eq3.4}
			a_1+a_2+b_1+b_3=a_3+a_4+b_2+b_4.
		\end{equation}
		
		Since $c_{\Gamma}(4)=c_{\Gamma}(5),$ we have 
		\begin{equation}\label{eq3.5}
			a_1+a_4+b_2+b_3=a_2+a_3+b_1+b_4.
		\end{equation}
		
		By \cref{eq3.2}$-$\cref{eq3.3} we have 
		\begin{equation}\label{eq3.6}
			a_2+a_4=b_1+b_2.
		\end{equation}
		
		By \cref{eq3.4}$-$\cref{eq3.5} we have 
		\begin{equation}\label{eq3.7}
			a_2-a_4=b_2-b_1.
		\end{equation}
		
		Solving \cref{eq3.6} and \cref{eq3.7} we have 
		\begin{equation}\label{eq3.8}
			a_2=b_2 \text{ and } a_4=b_1.
		\end{equation}
		
		From \cref{eq3.2} and \cref{eq3.8} we have
		\begin{equation}\label{eq3.9}
			a_1+a_3=b_3+b_4.
		\end{equation}
		
		From \cref{eq3.4} and \cref{eq3.8} we have
		\begin{equation}\label{eq3.10}
			a_1+b_3=a_3+b_4.
		\end{equation}
		
		Solving \cref{eq3.9} and \cref{eq3.10} we have
		\begin{equation}
			a_3=b_3 \text{ and } a_1=b_4.
		\end{equation}
		
		Hence, we have $a_2=b_2, a_4=b_1, a_3=b_3$ and $a_1=b_4$. Therefore, we get one of the tableaux in $\{\Gamma_1, \Gamma_2, \Gamma_3, \Gamma_4\}$ as a subtableauof $\Gamma$.
	\end{proof}
	
	\begin{corollary}
		The GIT quotient $X$ is projectively normal with respect to the descent of the $T_G$-linearized very ample line bundle $\mathcal{L}(2\omega_4).$
	\end{corollary}
	\begin{proof}
		Since $G/P^{\alpha_4}$ is normal, $(G/P^{\alpha_4})^{ss}_{T_G}(\mathcal{L}(2\omega_{4}))$ is also normal. Thus, ${T_G} \backslash \backslash (G/P^{\alpha_4})^{ss}_{T_G}(\mathcal{L}(2\omega_{4}))$ is a normal variety.  On the other hand, by \cref{lemma3.1}, the homogeneous coordinate ring of $X$ is generated by elements of degree one. Therefore, the GIT quotient ${T_G} \backslash \backslash (G/P^{\alpha_4})^{ss}_{T_G}(\mathcal{L}(2\omega_{4}))$ is projectively normal with respect to the  descent of the $T_G$-linearized very ample line bundle $\mathcal{L}(2\omega_4).$
	\end{proof}
	Let $w_1=s_3(s_1s_2s_4)$ and $w_2=(s_2s_3)(s_1s_2s_4).$ In one line notation $w_1$ is $(2,4,6,8)$ and $w_2$ is $(3,4,7,8).$
	\begin{corollary}\label{cor3.3}
		The GIT quotient $T_G \backslash \backslash (X_G(w_i))^{ss}_{T_G}(\mathcal{L}(2\omega_4))$ is projectively normal with respect to the  descent of the $T_G$-linearized very ample line bundle $\mathcal{L}(2\omega_4)$ for $i=1,2.$ 
	\end{corollary}
	\begin{proof}
		Since the restriction map $\phi: H^0(G/P^{\alpha_4}, \mathcal{L}^{\otimes k}(2\omega_4)) \to H^0(X_G(w_i), \mathcal{L}^{\otimes k}(2\omega_4))$ is surjective and $T_G$ is linearly reductive, $\tilde{\phi}: H^0(G/P^{\alpha_4}, \mathcal{L}^{\otimes k}(2\omega_4))^{T_G} \to H^0(X_G(w_i), \mathcal{L}^{\otimes k}(2\omega_4))^{T_G}$ such that $f \mapsto f|_{X_G(w_i)}$ is surjective. So, by \cref{lemma3.1}, $H^0(X_G(w_i), \mathcal{L}^{\otimes k}(2\omega_4))^{T_G}$ is generated by $H^0(X_G(w_i), \mathcal{L}(2\omega_4))^{T_G}$. Therefore, $T_G \backslash \backslash (X_G(w_i))^{ss}_{T_G}(\mathcal{L}(2\omega_4))$ is projectively normal with respect to the  descent of the $T_G$-linearized very ample line bundle $\mathcal{L}(2\omega_4).$
	\end{proof}
	Note that since $(G/P^{\alpha_4})^s_{T_G}(\mathcal{L}(2\omega_4)) \neq \emptyset$, dim$(X)=l(w)-4=2.$   
	\begin{proposition}
		The GIT quotient $X$ is isomorphic to $(\mathbb{P}^2, \mathcal{O}_{\mathbb{P}^2}(1))$ as a polarized variety. 
		
	\end{proposition}
	\begin{proof}
		
		By \cref{lemma3.1}, $R$ is generated by $\Gamma_1, \Gamma_2, \Gamma_3.$ On the other hand, dimension of $X$ is $2.$ Thus, $R$ is the polynomial algebra generated by $\Gamma_1, \Gamma_2, \Gamma_3.$ Therefore, $X$ is isomorphic to $(\mathbb{P}^2, \mathcal{O}_{\mathbb{P}^2}(1))$ as a polarized variety. 	
	\end{proof}
	
	\begin{corollary}
		\begin{enumerate}
			\item[(i)] The GIT quotient $T_G \backslash \backslash (X_G(w_1))^{ss}_{T_G}(\mathcal{L}(2\omega_4))$ is point.
			
			\item[(ii)] The GIT quotient $T_G \backslash \backslash (X_G(w_2))^{ss}_{T_G}(\mathcal{L}(2\omega_4))$ is $(\mathbb{P}^1, \mathcal{O}_{\mathbb{P}^1}(1))$ as a polarized variety.
		\end{enumerate}
	\end{corollary}
	\begin{proof}
		Proof of (i): Let $T_G \backslash \backslash (X_G(w_1))^{ss}_{T_G}(\mathcal{L}(2\omega_4)) =Proj(A),$ where $A=\bigoplus_{k\in\mathbb{Z}_{\geq 0}}A_k$ and $A_k=H^0(X_G(w_1), \mathcal{L}^{\otimes k}(2\omega_4))^{T_G}.$ By \cref{cor3.3}, $A$ is generated by $\Gamma_3.$ Thus, $A$ is the polynomial algebra generated by $\Gamma_3.$ Therefore, ${T_G} \backslash \backslash (X_G(w_1))^{ss}_{T_G}(\mathcal{L}(2\omega_4))$ is a point. 	
		
		Proof of (ii): Note that $T_G \backslash \backslash (X_G(w_2))^{ss}_{T_G}(\mathcal{L}(2\omega_4))=Proj(B),$ where $B=\bigoplus_{k\in\mathbb{Z}_{\geq 0}}B_k$ and $B_k=H^0(X_G(w_2), \mathcal{L}^{\otimes k}(2\omega_4))^{T_G}.$ By \cref{cor3.3}, $B$ is generated by $ \Gamma_2, \Gamma_3$. Note that $\Gamma_{2}\Gamma_{3}$ is a standard monomial. Thus, $B$ is the polynomial algebra generated by $ \Gamma_2, \Gamma_3.$ Therefore,  ${T_G} \backslash \backslash (X_G(w_2))^{ss}_{T_G}(\mathcal{L}(2\omega_4))$ is $(\mathbb{P}^1, \mathcal{O}_{\mathbb{P}^1}(1))$ as a polarized variety. 
	\end{proof}
	
	\section{$G=Spin(8n,\mathbb{C})$ $(n \geq 2)$} \label{section5}
	
	Let $X_G(w)$ be a Schubert variety in $G/P^{\alpha_{4n}}$ $(n \geq 2)$ admitting semi-stable point with respect to the $T_G$-linearized very ample line bundle $\mathcal{L}(2\omega_{4n}).$ Let $R=\bigoplus_{k \in \mathbb{Z}_{\geq 0}}R_k,$ where $R_k=H^0(X_G(w), \mathcal{L}^{\otimes k}(2\omega_{4n}))^{T_G}.$ Note that $R_k$'s are finite dimensional complex vector spaces. In general, it is very hard to determine whether $R$ is generated by $R_1$ or not as it involves various explicit calculations. In this section, for some Schubert varieties $X_G(w)$ in $G/P^{\alpha_{4n}}$ admitting semi-stable point with respect to the $T_G$-linearized very ample line bundle $\mathcal{L}(2\omega_{4n})$ we prove that the corresponding graded algebra $R$ is not generated by $R_1,$ but it is generated by $R_1$ and $R_2.$ 
	
	We recall that there exists a unique $w_1 \in W^{P^{\alpha_{4n}}}$ such that $w_1(2\omega_{4n})\leq 0$.  Note that $w_1=\prod_{i=2n}^{1}\tau_{2i-1},$ where $\tau_{2i-1}=s_{2i-1}s_{2i}\cdots s_{4n-2}s_{4n-1}$ for $i$ even and $\tau_{2i-1}=s_{2i-1}s_{2i}\cdots s_{4n-2}s_{4n}$ for  $i$ odd.
	
	Note that in one line notation $w_1$ is $(2,4,6,\ldots,4n-6,4n-4,4n-2,4n,4n+2,4n+4,4n+6,4n+8,\ldots,~8n).$ 
	Now we consider the following $w_i$'s such that $w_1 \leq w_i$ for all $2 \leq i \leq 6:$ 
	
	$w_2=s_{4n-4}w_1, w_3=s_{4n-2}w_{1}, w_{4}=s_{4n-2}s_{4n-4}w_{1}, w_{5}=s_{4n}s_{4n-2}w_{1}, w_{6}=s_{4n-4}s_{4n}s_{4n-2}w_{1}.$ 
	
	Since $w_1 \leq w_{i}$ and $w_{1}(2\omega_{4n})\leq 0,$ we have $w_{i}(2\omega_{4n}) \leq 0$ for all $2 \leq i \leq 6.$ Thus, by \cite[Lemma 2.1, p.470]{KP}, $X_G(w_i)^{ss}_{T_G}(\mathcal{L}(2\omega_{4n}))$ is non-empty for all $1 \leq i \leq 6.$ In one line notation $w_i$'s $(2 \leq i \leq 6)$ are the following:  
	\begin{itemize} 
		\item $w_2~~=~(2,4,6,\ldots,4n-6,4n-3,4n-2,4n,4n+2,4n+5,4n+6,4n+8,\ldots,~8n)$
		\item $w_3~~=~(2,4,6,\ldots,4n-6,4n-4,4n-1,4n,4n+3,4n+4,4n+6,4n+8,\ldots,~8n)$
		\item $w_4~~=~(2,4,6,\ldots,4n-6,4n-3,4n-1,4n,4n+3,4n+5,4n+6,4n+8,\ldots,~8n)$
		\item  $w_5=(2,4,6,\ldots,4n-6,4n-4,4n+1,4n+2,4n+3,4n+4,4n+6,4n+8,\ldots,8n)$
		\item  $w_{6}=(2,4,6,\ldots,4n-6,4n-3,4n+1,4n+2,4n+3,4n+5,4n+6,4n+8,\ldots,8n).$
	\end{itemize} 
	
	Then we have the following sub-diagram of the Bruhat lattice $W^{P^{\alpha_{4n}}}$.
	
	\begin{center}
		\begin{tikzpicture}[scale=.7]
			\node (e) at (0,0)  {$w_1$};
			\node (z) at (-2,-1.2) {$s_{4n-4}$};
			\node (y) at (2,-1.2) {$s_{4n-2}$};
			\node (g) at (-3,-3) {$w_2$};
			\node (b) at (3,-3) {$w_3$};
			\node (w) at (5.5,-4.8) {$s_{4n}$};
			\node (v) at (-2.5,-4.8) {$s_{4n-2}$};
			\node (u) at (2.5,-4.8) {$s_{4n-4}$};
			\node (d) at (0,-6) {$w_4$};
			\node (p) at (6,-6) {$w_5$};
			\node (s) at (5.5,-7.6) {$s_{4n-4}$};
			\node (q) at (.5,-7.6) {$s_{4n}$};
			\node (a) at (3,-9) {$w_6$};
			\node (z) at (3,-10) {Figure $1$};
			\draw (e) -- (g) -- (d) -- (a) -- (p) -- (b) -- (e);
			\draw (b) -- (d);
		\end{tikzpicture}
	\end{center}

	Let $X= T_G \backslash \backslash (X_G(w_{6}))^{ss}_{T_G}(\mathcal{L}(2\omega_{4n})).$ Let $R=\bigoplus_{k \in \mathbb{Z}_{\geq 0}} R_k,$ where
	$R_{k}=H^0(X_G(w_{6}), \mathcal{L}^{\otimes k}(2\omega_{4n}))^{T_G}.$ Then we have $X=Proj(R).$ 
	
	Let 
	
	\tiny{\ytableausetup{boxsize=3.5em} \[X_1= \ytableausetup{centertableaux}
		\begin{ytableau}
			1 & 3 & 5  & \cdots & 4n-7 & 4n-5 & 4n-3 & 4n-1 & 4n+1 & 4n+3 & 4n+5 & 4n+7 & 4n+9 & \cdots  & 8n-3 & 8n-1\\
			2 & 4 & 6 & \cdots & 4n-6 & 4n-4 & 4n-2 & 4n & 4n+2 & 4n+4 & 4n+6 & 4n+8 & 4n+10 & \cdots & 8n-2 & 8n 
		\end{ytableau} \]}
	
	\tiny{\ytableausetup{boxsize=3.5em} \[X_2= \ytableausetup{centertableaux}
		\begin{ytableau}
			1 & 3 & 5  & \cdots & 4n-7 & 4n-5 & 4n-4 & 4n-1 & 4n+1 & 4n+3 & 4n+4 & 4n+7 & 4n+9 & \cdots & 8n-3 & 8n-1\\
			2 & 4 & 6 & \cdots & 4n-6 & 4n-3 & 4n-2 & 4n & 4n+2 & 4n+5 & 4n+6 & 4n+8 & 4n+10 & \cdots & 8n-2 & 8n 
		\end{ytableau} \]}
	
	\tiny{\ytableausetup{boxsize=3.5em} \[X_3= \ytableausetup{centertableaux}
		\begin{ytableau}
			1 & 3 & 5  & \cdots & 4n-7 & 4n-5 & 4n-3 & 4n-2 & 4n+1 & 4n+2 & 4n+5 & 4n+7 & 4n+9 & \cdots  & 8n-3 & 8n-1\\
			2 & 4 & 6 & \cdots & 4n-6 & 4n-4 & 4n-1 & 4n & 4n+3 & 4n+4 & 4n+6 & 4n+8 & 4n+10 & \cdots & 8n-2 & 8n 
		\end{ytableau} \]}
	\tiny{\ytableausetup{boxsize=3.5em} \[X_4= \ytableausetup{centertableaux}
		\begin{ytableau}
			1 & 3 & 5  & \cdots & 4n-7 & 4n-5 & 4n-4 & 4n-2 & 4n+1 & 4n+2 & 4n+4 & 4n+7 & 4n+9 & \cdots  & 8n-3 & 8n-1\\
			2 & 4 & 6 & \cdots & 4n-6 & 4n-3 & 4n-1 & 4n & 4n+3 & 4n+5 & 4n+6 & 4n+8 & 4n+10 & \cdots & 8n-2 & 8n 
		\end{ytableau} \]}
	\tiny{\ytableausetup{boxsize=3.5em} \[X_5= \ytableausetup{centertableaux}
		\begin{ytableau}
			1 & 3 & 5  & \cdots & 4n-7 & 4n-5 & 4n-3 & 4n-2 & 4n-1 & 4n & 4n+5 & 4n+7 & 4n+9 & \cdots  & 8n-3 & 8n-1\\
			2 & 4 & 6 & \cdots & 4n-6 & 4n-4 & 4n+1 & 4n+2 & 4n+3 & 4n+4 & 4n+6 & 4n+8 & 4n+10 & \cdots & 8n-2 & 8n 
		\end{ytableau} \]}
	\tiny{\ytableausetup{boxsize=3.5em} \[X_6= \ytableausetup{centertableaux}
		\begin{ytableau}
			1 & 3 & 5  & \cdots & 4n-7 & 4n-5 & 4n-4 & 4n-2 & 4n-1 & 4n & 4n+4 & 4n+7 & 4n+9 & \cdots  & 8n-3 & 8n-1\\
			2 & 4 & 6 & \cdots & 4n-6 & 4n-3 & 4n+1 & 4n+2 & 4n+3 & 4n+5 & 4n+6 & 4n+8 & 4n+10 & \cdots & 8n-2 & 8n 
		\end{ytableau} \]}
	
	\tiny{\ytableausetup{boxsize=3.5em} \[Y_1= \ytableausetup{centertableaux}
		\begin{ytableau}
			1& 3 & 5  & \cdots & 4n-7 & 4n-5 & 4n-4 & 4n-3 & 4n-2 & 4n-1 & 4n+1 & 4n+7 & 4n+9 & \cdots & 8n-3 & 8n-1\\
			1& 3 & 5 & \cdots & 4n-7 & 4n-5 & 4n-2 & 4n & 4n+2 & 4n+4 & 4n+5 & 4n+7 & 4n+9 & \cdots & 8n-3 & 8n-1\\
			2 & 4 & 6 & \cdots & 4n-6 & 4n-4 & 4n-1 & 4n & 4n+3 & 4n+4 & 4n+6 & 4n+8 & 4n+10 & \cdots & 8n-2 & 8n\\
			2 & 4 & 6 & \cdots & 4n-6 & 4n-3 & 4n+1 & 4n+2 & 4n+3 & 4n+5 & 4n+6 & 4n+8 & 4n+10 & \cdots & 8n-2 & 8n\\
		\end{ytableau},\]}
	\tiny{\ytableausetup{boxsize=3.5em} \[Y_2= \ytableausetup{centertableaux}
		\begin{ytableau}
			1& 3 & 5  & \cdots & 4n-7 & 4n-5 & 4n-4 & 4n-3 & 4n-2 & 4n & 4n+2 & 4n+7 & 4n+9 & \cdots & 8n-3 & 8n-1\\
			1& 3 & 5 & \cdots & 4n-7 & 4n-5 & 4n-2 & 4n-1 & 4n+1 & 4n+4 & 4n+5 & 4n+7 & 4n+9 & \cdots & 8n-3 & 8n-1\\
			2 & 4 & 6 & \cdots & 4n-6 & 4n-4 & 4n-1 & 4n & 4n+3 & 4n+4 & 4n+6 & 4n+8 & 4n+10 & \cdots & 8n-2 & 8n\\
			2 & 4 & 6 & \cdots & 4n-6 & 4n-3 & 4n+1 & 4n+2 & 4n+3 & 4n+5 & 4n+6 & 4n+8 & 4n+10 & \cdots & 8n-2 & 8n\\
		\end{ytableau},\]}
	\tiny{\ytableausetup{boxsize=3.5em} \[Y_3= \ytableausetup{centertableaux}
		\begin{ytableau}
			1& 3 & 5  & \cdots & 4n-7 & 4n-5 & 4n-4 & 4n-3 & 4n-1 & 4n & 4n+3 & 4n+7 & 4n+9 & \cdots & 8n-3 & 8n-1\\
			1& 3 & 5 & \cdots & 4n-7 & 4n-5 & 4n-2 & 4n-1 & 4n+1 & 4n+4 & 4n+5 & 4n+7 & 4n+9 & \cdots & 8n-3 & 8n-1\\
			2 & 4 & 6 & \cdots & 4n-6 & 4n-4 & 4n-2 & 4n & 4n+2 & 4n+4 & 4n+6 & 4n+8 & 4n+10 & \cdots & 8n-2 & 8n\\
			2 & 4 & 6 & \cdots & 4n-6 & 4n-3 & 4n+1 & 4n+2 & 4n+3 & 4n+5 & 4n+6 & 4n+8 & 4n+10 & \cdots & 8n-2 & 8n\\
		\end{ytableau},\]}
	\tiny{\ytableausetup{boxsize=3.5em} \[Y_4= \ytableausetup{centertableaux}
		\begin{ytableau}
			1& 3 & 5  & \cdots & 4n-7 & 4n-5 & 4n-4 & 4n-3 & 4n+1 & 4n+2 & 4n+3 & 4n+7 & 4n+9 & \cdots & 8n-3 & 8n-1\\
			1& 3 & 5 & \cdots & 4n-7 & 4n-5 & 4n-2 & 4n-1 & 4n+1 & 4n+4 & 4n+5 & 4n+7 & 4n+9 & \cdots & 8n-3 & 8n-1\\
			2 & 4 & 6 & \cdots & 4n-6 & 4n-4 & 4n-2 & 4n & 4n+2 & 4n+4 & 4n+6 & 4n+8 & 4n+10 & \cdots & 8n-2 & 8n\\
			2 & 4 & 6 & \cdots & 4n-6 & 4n-3 & 4n-1 & 4n & 4n+3 & 4n+5 & 4n+6 & 4n+8 & 4n+10 & \cdots & 8n-2 & 8n\\
		\end{ytableau},\]}
	
	\tiny{\ytableausetup{boxsize=3.5em} \[Z_1= \ytableausetup{centertableaux}
		\begin{ytableau}
			1& 3 & 5  & \cdots & 4n-7 & 4n-5 & 4n-4 & 4n-3 & 4n-2 & 4n-1 & 4n+1 & 4n+7 & 4n+9 & \cdots & 8n-3 & 8n-1\\
			1& 3 & 5 & \cdots & 4n-7 & 4n-5 & 4n-4 & 4n-1 & 4n+1 & 4n+3 & 4n+4 & 4n+7 & 4n+9 & \cdots & 8n-3 & 8n-1\\
			1& 3 & 5 & \cdots & 4n-7 & 4n-5 & 4n-2 & 4n & 4n+2 & 4n+4 & 4n+5 & 4n+7 & 4n+9 & \cdots & 8n-3 & 8n-1\\
			2 & 4 & 6 & \cdots & 4n-6 & 4n-4 & 4n-2 & 4n & 4n+2 & 4n+4 & 4n+6 & 4n+8 & 4n+10 & \cdots & 8n-2 & 8n\\
			2 & 4 & 6 & \cdots & 4n-6 & 4n-3 & 4n-1 & 4n & 4n+3 & 4n+5 & 4n+6 & 4n+8 & 4n+10 & \cdots & 8n-2 & 8n\\
			2 & 4 & 6 & \cdots & 4n-6 & 4n-3 & 4n+1 & 4n+2 & 4n+3 & 4n+5 & 4n+6 & 4n+8 & 4n+10 & \cdots & 8n-2 & 8n\\
		\end{ytableau},\]}
	\tiny{\ytableausetup{boxsize=3.5em} \[Z_2= \ytableausetup{centertableaux}
		\begin{ytableau}
			1& 3 & 5  & \cdots & 4n-7 & 4n-5 & 4n-4 & 4n-3 & 4n-2 & 4n & 4n+2 & 4n+7 & 4n+9 & \cdots & 8n-3 & 8n-1\\
			1 & 3 & 5 & \cdots & 4n-7 & 4n-5 & 4n-4 & 4n-1 & 4n+1 & 4n+3 & 4n+4 & 4n+7 & 4n+9 & \cdots & 8n-3 & 8n-1\\
			1 & 3 & 5 & \cdots & 4n-7 & 4n-5 & 4n-2 & 4n-1 & 4n+1 & 4n+4 & 4n+5 & 4n+7 & 4n+9 & \cdots & 8n-3 & 8n-1\\
			2 & 4 & 6 & \cdots & 4n-6 & 4n-4 & 4n-2 & 4n & 4n+2 & 4n+4 & 4n+6 & 4n+8 & 4n+10 & \cdots & 8n-2 & 8n\\
			2 & 4 & 6 & \cdots & 4n-6 & 4n-3 & 4n-1 & 4n & 4n+3 & 4n+5 & 4n+6 & 4n+8 & 4n+10 & \cdots & 8n-2 & 8n\\
			2 & 4 & 6 & \cdots & 4n-6 & 4n-3 & 4n+1 & 4n+2 & 4n+3 & 4n+5 & 4n+6 & 4n+8 & 4n+10 & \cdots & 8n-2 & 8n\\
		\end{ytableau},\]}
	
	\normalsize{
		
		Note that $X_i$'s form a basis of $R_1.$ Further, note that $Y_i$'s (respectively, $Z_i$'s) are standard monomials of $R_2$ (respectively, $R_3$) which are not in $R_1 \otimes R_1$ (respectively, not in $R_1 \otimes R_2$). 
		\begin{lemma}\label{lem5.1}
			The graded algebra $R$ is generated by $X_i$'s, $Y_j$'s and $Z_l$'s  as a $\mathbb{C}$-algebra.
	\end{lemma}}
	
	\begin{proof}
		
		Let $f \in R_k$ be a standard monomial. We claim that $f = f_1f_2,$ where either $f_1$ or $f_2$ is in $R_1$ or $R_2$ or $R_3$.
		
		The Young diagram associated to $f$ has shape $p = (p_1,  p_2, \ldots, p_{4n}) = (2k, 2k, \ldots, 2k)$. So, the Young tableau $\Gamma$ associated to $f$ has $2k$ number of rows and $4n$ number of columns such that rows are filled up with strictly increasing integers from left to right and columns are filled up with non-decreasing integers from top to bottom. Since $f$ is $T_G$-invariant, by \cref{zeroweight} we have \begin{equation}\label{eq11.1} c_\Gamma(t)=c_\Gamma(8n+1-t) 
		\end{equation}
		for all $1 \leq t \leq 8n,$ where $c_{\Gamma}(t)$ denotes the number of times $t$ appears in $\Gamma.$ 
		
		Let $r_i$ be the $i$-th row of the tableau. Let $E_{i,j}$  be the $(i,j)$-th entry of $\Gamma.$ Let $N_{t,j}$ denote the number of times $t$ appears in the $j$-th column.  
		Note that each row of $\Gamma$ contains either $i$ or $2n + 1 - i$ for all $1 \leq i \leq 2n.$
		Further, by using \cref{eq11.1}, we have \begin{equation}\label{eq5.2}
			c_{\Gamma}(t)=k
		\end{equation} for all $1 \leq t \leq 8n$.
		
		Recall that in one line notation $w_6$ is $(2,4,6,\ldots,4n-6,4n-3,4n+1,4n+2,4n+3,4n+5,4n+6,4n+8,\ldots,8n).$ Since $r_{2k} \leq w_6,$ we have $E_{2k,j} \leq 2j$ for all $1 \leq j \leq 2n-3.$ Note that for $1 \leq j \leq 2n-3,$ we have $E_{i,j}=2j$ for all $k+1 \leq i \leq 2k.$ Thus, for $1 \leq j \leq 2n-2,$ we have $E_{i,j}=2j-1$ for all $1 \leq i \leq k.$ 
		Since $E_{k,2n-2}=4n-5$ and $E_{2k,2n-2} \leq 4n-3,$ we have $E_{2k,2n-2}$ is either $4n-4$ or $4n-3$.
		
		Note that since each row has even number of integer entries greater than $4n$ and each row of $\Gamma$ has $4n$ number of boxes, it follows that each row has even number of integers less than or equal to $4n.$ $(*)$
		
		Hence, by $(*),$  $E_{2k,2n-1} \neq 4n.$ Further, since $4n-2 \leq E_{2k,2n-1} \leq 4n+1,$ it follows that $E_{2k,2n-1}\in \{4n-2, 4n-1, 4n+1\}.$  Also, by $(*),$ $E_{2k,2n} \neq 4n+1.$ Further, since $4n \leq E_{2k,2n} \leq 4n+2,$ $E_{2k,2n}$ is either $4n$ or $4n+2.$ Thus, the following are the possibilities for $r_{2k}:$ 
		\begin{itemize}
			\item $(2,4,6,\ldots,4n-6,4n-4,4n-2,~~4n~~,4n+2,4n+4,4n+6,4n+8,\ldots,8n)$ 
			\item $(2,4,6,\ldots,4n-6,4n-3,4n-2,~~4n~~,4n+2,4n+5,4n+6,4n+8,\ldots,8n)$
			\item $(2,4,6,\ldots,4n-6,4n-4,4n-1,~~4n~~,4n+3,4n+4,4n+6,4n+8,\ldots,8n)$
			\item $(2,4,6,\ldots,4n-6,4n-3,4n-1,~~4n~~,4n+3,4n+5,4n+6,4n+8,\ldots,8n)$ 
			\item $(2,4,6,\ldots,4n-6,4n-4,4n+1,4n+2,4n+3,4n+4,4n+6,4n+8,\ldots,8n)$
			\item $(2,4,6,\ldots,4n-6,4n-3,4n+1,4n+2,4n+3,4n+5,4n+6,4n+8,\ldots,8n).$
		\end{itemize} 
		
		Now, observe that in each of the following cases for $1 \leq j \leq 2n-3$ (respectively, $1 \leq j \leq 2n-2$), we have $E_{i,j}=2j$ (respectively, $E_{i,j}=2j-1$) for all $k+1 \leq i \leq 2k$ (respectively, $1 \leq i \leq k$).   
		
		\textbf{Case I:} Assume that $r_{2k}=(2,4,6,\ldots,4n-6,4n-4,4n-2,4n,4n+2,4n+4,4n+6,4n+8,\ldots,8n).$ Then for $1 \leq j \leq 2n,$ we have $E_{i,j}=2j-1$ (respectively, $E_{i,j}=2j$) for all $1 \leq i \leq k$ (respectively, for all $k+1 \leq i \leq 2k$). Therefore, $r_1=(1,3,5,\ldots,4n-7,4n-5,4n-3,4n-1,4n+1,4n+3,4n+5,4n+7,\ldots,8n-1)$ and hence, $X_{1}$ is a factor of $f.$
		
		\textbf{Case II:} Assume that $r_{2k}=(2,4,6,\ldots,4n-6,4n-3,4n-2,4n,4n+2,4n+5,4n+6,4n+8,\ldots,8n).$ Since $E_{k,2n-2}=4n-5$ and $N_{4n-3,2n-2}\geq 1,$ we have $N_{4n-4,2n-2} \leq k-1.$ Thus, $E_{1,2n-1}=4n-4.$ Further, note that since $E_{2k,2n-2}=4n-3,$ we have $4n-4 \leq E_{i,2n-2} \leq 4n-3$ for all $k+1 \leq i \leq 2k.$ Since $E_{2k,2n-1}=4n-2,$ we have $4n-4 \leq E_{i,2n-1} \leq 4n-2$ for all $1 \leq i \leq 2k.$ Thus, $\sum_{i=4n-4}^{4n-2}\sum_{j=2n-2}^{2n-1}N_{i,j}=3k=\sum_{i=4n-4}^{4n-2}c_{\Gamma}(i).$ Hence, $E_{1,2n}=4n-1.$ Therefore, $r_1=(1,3,5,\ldots,4n-7,4n-5,4n-4,4n-1,4n+1,4n+3,4n+4,4n+7,\ldots,8n-1)$ and hence, $X_{2}$ is a factor of $f.$  
		
		\textbf{Case III:} Assume that $r_{2k}=(2,4,6,\ldots,4n-6,4n-4,4n-1,4n,4n+3,4n+4,4n+6,4n+8,\ldots,8n).$ Then by using the argument similar to Case II we get $X_{3}$ is a factor of $f.$ 
		
		\textbf{Case IV:} Assume that $r_{2k}=(2,4,6,\ldots,4n-6,4n-3,4n-1,4n,4n+3,4n+5,4n+6,4n+8,\ldots,8n).$ Since $N_{4n-3,2n-2}\geq 1,$ we have $N_{4n-4,2n-2} \leq k-1.$ Thus, $E_{1,2n-1}=4n-4.$ Since $E_{2k,2n}=4n,$ we have $\sum_{i=1}^{2k}\sum_{j=1}^{4n}N_{i,j}=4nk=\sum_{i=1}^{4n} c_{\Gamma}(i).$  Thus all the integers between $1$ and $4n$ appear up to $2n$-th column. Hence, $E_{i,2n}=4n$ for all $k+1 \leq i \leq 2k.$ Since $E_{2k,2n-1}=4n-1,$ we have $4n-3 \leq E_{1,2n} \leq 4n-2.$ Thus, $r_1$ is either $(1,3,5,\ldots,4n-7,4n-5,4n-4,4n-2,4n+1,4n+2,4n+4,4n+7,\ldots,8n-1)$ or $(1,3,5,\ldots,4n-7,4n-5,4n-4,4n-3,4n+1,4n+3,4n+8,4n+9,\ldots,8n-1).$  If $r_1=(1,3,5,\ldots,4n-7,4n-5,4n-4,4n-2,4n+1,4n+2,4n+4,4n+7,\ldots,8n-1),$ then $r_1$ and $r_{2k}$ together gives a factor $X_{4}$ of $f.$ If $r_1=(1,3,5,\ldots,4n-7,4n-5,4n-4,4n-3,4n+1,4n+2,4n+3,4n+7,4n+9,\ldots,8n-1),$ then $E_{k+1,2n-2}=4n-4.$ Therefore, $E_{k,2n-1}$ can not be $4n-4$ and $4n-3.$ Hence, $E_{k,2n-1}=4n-2.$ Therefore, $r_k=(1,3,5,\ldots,4n-7,4n-5,4n-2,4n-1,4n+1,4n+4,4n+5,4n+7,\ldots,8n-1)$ and $r_{k+1}=(2,4,6,\ldots,4n-6,4n-4,4n-2,4n,4n+2,4n+4,4n+6,4n+8,\ldots,8n).$ Therefore, $r_1, r_{k}, r_{k+1}, r_{2k}$ together gives a factor $Y_4$ of $f.$
		
		\textbf{Case V:} Assume that $r_{2k}=(2,4,6,\ldots,4n-6,4n-4,4n+1,4n+2,4n+3,4n+4,4n+6,4n+8,\ldots,8n).$ Then by using the argument similar to Case II we get $X_{5}$ as a factor of $f.$ 
		
		\textbf{Case VI:} Assume that $r_{2k}=(2,4,6,\ldots,4n-6,4n-3,4n+1,4n+2,4n+3,4n+5,4n+6,4n+8,\ldots,8n).$ Since $E_{2k,2n-2}=4n-3,$ we have $E_{1,2n-1}=4n-4.$ Thus, $E_{1,2n+1} \geq 4n-2.$ If $E_{1,2n+1} \geq 4n+1,$ then $\sum_{i=4n+1}^{8n}\sum_{j=2n-1}^{4n}N_{i,j} \geq 4nk+2,$ which is a contradiction to \cref{eq5.2}. Also, if $E_{1,2n+1}=4n,$ then the number of integers less than or equal to $4n$ appearing in $r_1$ is odd, which is a contradiction. Hence, $E_{1,2n+1}$ is either $4n-2$ or $4n-1.$ If $E_{1,2n+1}=4n-2,$ then $E_{1,2n}=4n-3.$ If $E_{1,2n+1}=4n-1,$ then $E_{1,2n}$ is either $4n-3$ or $4n-2.$ Thus $r_1$ has following possibilities: 
		\begin{itemize}
			\item $(1,3,5,\ldots,4n-7,4n-5,4n-4,4n-2,4n-1,4n,4n+4,4n+7,\ldots,8n-1)$
			\item $(1,3,5,\ldots,4n-7,4n-5,4n-4,4n-3,4n-1,4n,4n+3,4n+7,\ldots,8n-1)$
			\item $(1,3,5,\ldots,4n-7,4n-5,4n-4,4n-3,4n-2,4n-1,4n+1,4n+7,\ldots,8n-1)$
			\item $(1,3,5,\ldots,4n-7,4n-5,4n-4,4n-3,4n-2,4n,4n+2,4n+7,\ldots,8n-1).$
		\end{itemize}  
		
		\textbf{Subcase (i):} If $r_1=(1,3,5,\ldots,4n-7,4n-5,4n-4,4n-2,4n-1,4n,4n+4,4n+7,\ldots,8n-1),$ then $X_6$ is a factor of $f.$
		
		\textbf{Subcase (ii):} If $r_{1}=(1,3,5,\ldots,4n-7,4n-5,4n-4,4n-3,4n-1,4n,4n+3,4n+7,\ldots,8n-1),$ then $N_{4n-3,2n-2} \leq k-1.$ Hence, $E_{k+1, 2n-2}=4n-4.$ We claim that $r_{k}$ has exactly $2n$ number of entries less than or equal to $4n.$ If $r_k$ has $2n+2$ number of entries less than or equal to $4n,$ then by counting the number of integers we have $r_{i}$ (for all $k+1 \leq i \leq 2k$) has $2n-2$ number of entries less than or equal to $4n.$ Then $r_{i}$ ($k+1 \leq i \leq 2k$) does not contain $4n-2.$ Hence, $4n-2$ appear in the rows $r_i$ ($2 \leq i \leq k$), which is a contradiction to \cref{eq5.2}. If $r_{k}$ has $2n-2$ number of entries less than or equal to $4n,$ then $\sum_{i=1}^{4n}\sum_{j=1}^{2n+2}N_{i,j} \leq (k-1)(2n+2)+(k+1)(2n-2)=4nk-4,$ which is a contradiction to \cref{eq5.2}. Thus, $r_{k}$ has exactly $2n$ number of entries less than or equal to $4n.$ Then $E_{k,2n-1}\in\{4n-4, 4n-3, 4n-2, 4n-1\}.$ Note that $E_{k,2n-1}$ can not be  $4n-4$ and $4n-3,$ otherwise, it contradicts the number of appearance of $4n-4$ and $4n-3.$ If $E_{k,2n-1}=4n-1,$ then $4n-2$ appears in the rows $r_{i}$ ($2 \leq i \leq k-1$), which is a contradiction to \cref{eq5.2}. Thus, $E_{k,2n-1}=4n-2.$ Hence, $r_k$ has the following possibilities:
		\begin{itemize}
			\item $(1,3,5,\ldots,4n-7,4n-5,4n-2,4n-1,4n+1,4n+4,4n+5,4n+7,\ldots,8n-1)$
			\item $(1,3,5,\ldots,4n-7,4n-5,4n-2,4n,4n+2,4n+4,4n+5,4n+7,\ldots,8n-1).$
		\end{itemize} 
		Since $r_{k}$ has exactly $2n$ number of entries less than or equal to $4n,$ $r_{k+1}$ has at most $2n$ number of entries less than or equal to $4n.$ If $r_{k+1}$ has $2n-2$ number of entries less than or equal to $4n,$ then $\sum_{i=1}^{4n}\sum_{j=1}^{2n+2}N_{i,j} \leq (k-1)(2n+2)+2n+(k)(2n-2)=4nk-2,$ which is a contradiction to \cref{eq5.2}. Thus, $r_{k+1}$ has exactly $2n$ number of entries less than or equal to $4n.$ Now we consider $r_k$ and $r_{k+1}$ case by case. 
		
		If $r_k=(1,3,5,\ldots,4n-7,4n-5,4n-2,4n-1,4n+1,4n+4,4n+5,4n+7,\ldots,8n-1),$ then $r_{k+1}$ has the following possibilities: 
		\begin{itemize}
			\item $(2,4,6,\ldots,4n-6,4n-4,4n-2,4n-1,4n+1,4n+4,4n+6,4n+8,\ldots,8n)$
			\item $(2,4,6,\ldots,4n-6,4n-4,4n-1,4n,4n+3,4n+4,4n+6,4n+8,\ldots,8n)$
			\item $(2,4,6,\ldots,4n-6,4n-4,4n-2,4n,4n+2,4n+4,4n+6,4n+8,\ldots,8n).$
		\end{itemize}
		If $r_{k+1}=(2,4,6,\ldots,4n-6,4n-4,4n-2,4n-1,4n+1,4n+4,4n+6,4n+8,\ldots,8n),$ then there exists at least one row $r_i$ ($k+2 \leq i \leq 2k-1$) which contains $4n$. Let $k_1$ be the number of rows between $r_{k+2}$ and $r_{2k-1}$ which contains $4n.$ Thus, $ k-k_1 \geq 1.$ Hence, $k-k_1$ is the number of rows between $r_{1}$ and $r_{k-1}$ which contains $4n.$ Then $\sum_{i=1}^{4n}\sum_{j=1}^{2n+2}N_{i,j} \geq (2n+2)(k-k_1)+2n(k_1)+2n+2n(k_1)+(2n)(k-k_1-1)=4nk+2(k-k_1)>4nk,$ which is a contradiction to \cref{eq5.2}. Therefore, $r_{k+1}$ can not be $(2,4,6,\ldots,4n-6,4n-4,4n-2,4n-1,4n+1,4n+4,4n+6,4n+8,\ldots,8n.$ If $r_{k+1}=(2,4,6,\ldots,4n-6,4n-4,4n-1,4n,4n+3,4n+4,4n+6,4n+8,\ldots,8n),$ then $4n-2$ can appear at most $k-1$ times (namely, between $r_2$ and $r_{k}$), which is a contradiction to \cref{eq5.2}. Therefore, $r_{k+1}$ can not be $(2,4,6,\ldots,4n-6,4n-4,4n-1,4n,4n+3,4n+4,4n+6,4n+8,\ldots,8n).$
		If $r_{k+1}=(2,4,6,\ldots,4n-6,4n-4,4n-2,4n,4n+2,4n+4,4n+6,4n+8,\ldots,8n),$ then $r_1,r_{k}, r_{k+1}$ and $r_{2k}$ together gives a factor $Y_3$ of $f.$  
		
		If $r_k=(1,3,5,\ldots,4n-7,4n-5,4n-2,4n,4n+2,4n+4,4n+5,4n+7,\ldots,8n-1),$ then $r_{k+1}$ has following possibilities:
		\begin{itemize}
			\item $(2,4,6,\ldots,4n-6,4n-4,4n-2,4n,4n+2,4n+4,4n+6,4n+8,\ldots,8n),$
			\item $(2,4,6,\ldots,4n-6,4n-4,4n-1,4n,4n+3,4n+4,4n+6,4n+8,\ldots,8n).$
		\end{itemize} 
		Assume that $r_{k+1}=(2,4,6,\ldots,4n-6,4n-4,4n-2,4n,4n+2,4n+4,4n+6,4n+8,\ldots,8n).$ Let $k_1$ be the number of rows between $r_{k+1}$ and $r_{2k-1}$ which contains $4n$ and $k_2$ be the number of rows between $r_1$ and $r_{k-1}$ whose $2n+2$-th entry is $4n$. Thus, $\sum_{i=1}^{4n}\sum_{j=1}^{2n+2}N_{i,j} = (2n+2)k_2+2n(k-k_2)+2nk_1+(2n-2)(k-k_1)=(4n-2)k+2(k_1+k_2),$ which is the total number of appearance of the integers between $1$ and $4n.$ Hence, $k_1+k_2=k.$ Then $N_{4n,2n}+N_{4n,2n+2}\geq k_1+k_2+1$ which is a contradiction to \cref{eq5.2}. Therefore, $r_{k+1}$ can not be $(2,4,6,\ldots,4n-6,4n-4,4n-2,4n,4n+2,4n+4,4n+6,4n+8,\ldots,8n).$ If $r_{k+1}=(2,4,6,\ldots,4n-6,4n-4,4n-1,4n,4n+3,4n+4,4n+6,4n+8,\ldots,8n),$ then $4n-2$ can appear at most $k-1$ times in $\Gamma$ (namely, between $r_{2}$ and $r_k$). Hence, $\sum_{j=1}^{2n+2}N_{4n-2,j} \leq k-1,$ which is a contradiction to \cref{eq5.2}. Therefore, $r_{k+1}$ can not be $(2,4,6,\ldots,4n-6,4n-4,4n-1,4n,4n+3,4n+4,4n+6,4n+8,\ldots,8n).$
		
		\textbf{Subcase (iii):} If $r_1=(1,3,5,\ldots,4n-7,4n-5,4n-4,4n-3,4n-2,4n-1,4n+1,4n+7,\ldots,8n-1),$ then $N_{4n-3,2n-2} \leq k-1.$ Thus, $E_{k+1,2n-2}=4n-4.$ We claim that $r_{k}$ has exactly $2n$ number of entries less than or equal to $4n.$ If $r_k$ has $2n+2$ number of entries less than or equal to $4n,$ then by counting the number of integers between $1$ and $4n$ we have $r_{i}$ (for $k+1 \leq i \leq 2k$) has $2n-2$ number of entries less than or equal to $4n.$ Then $r_{i}$ ($k+1 \leq i \leq 2k$) does not contain $4n.$ Hence, $4n$ can appear in the rows $r_i$ ($2 \leq i \leq k$), which is a contradiction to \cref{eq5.2}. If $r_{k}$ has $2n-2$ number of entries less than or equal to $4n,$ then $\sum_{i=1}^{4n}\sum_{j=1}^{2n+2}N_{i,j} \leq (k-1)(2n+2)+(k+1)(2n-2)=4nk-4,$ which is a contradiction to \cref{eq5.2}. Thus, $r_{k}$ has exactly $2n$ number of entries less than or equal to $4n.$ Then $E_{k,2n-1}\in\{4n-4, 4n-3, 4n-2, 4n-1\}.$ Note that $E_{k,2n-1}$ can not be  $4n-4$ and $4n-3,$ otherwise, it contradicts the number of appearance of $4n-4$ and $4n-3.$ If $E_{k,2n-1}=4n-1,$ then $4n-2$ appears in the rows $r_{i}$ ($1 \leq i \leq k-1$), which is a contradiction to \cref{eq5.2}. Thus, $E_{k,2n-1}=4n-2.$ Hence, $r_k$ has the following possibilities:
		\begin{itemize}
			\item $(1,3,5,\ldots,4n-7,4n-5,4n-2,4n-1,4n+1,4n+4,4n+5,4n+7,\ldots,8n-1)$
			\item $(1,3,5,\ldots,4n-7,4n-5,4n-2,4n,4n+2,4n+4,4n+5,4n+7,\ldots,8n-1).$
		\end{itemize}
		Since $r_{k}$ has exactly $2n$ number of entries less than or equal to $4n,$ $r_{k+1}$ has at most $2n$ number of entries less than or equal to $4n.$ If $r_{k+1}$ has $2n-2$ number of entries less than or equal to $4n,$ then $\sum_{i=1}^{4n}\sum_{j=1}^{2n+2}N_{i,j} \leq (k-1)(2n+2)+2n+k(2n-2)=4nk-2,$ which is a contradiction to \cref{eq5.2}. Thus, $r_{k+1}$ has exactly $2n$ number of entries less than or equal to $4n.$ Now we consider $r_k$ and $r_{k+1}$ case by case. 
		
		Assume that $r_{k}=(1,3,5,\ldots,4n-7,4n-5,4n-2,4n-1,4n+1,4n+4,4n+5,4n+7,\ldots,8n-1).$ Let $k_1$ be the number of rows between $r_2$ and $r_{k-1}$ which contains $4n.$ Then $k-k_1$ be the number of rows between $r_{k+1}$ and $r_{2k-1}$ which contains $4n.$ Then $\sum_{i=1}^{4n}\sum_{j=1}^{2n+2}N_{i,j} \geq (2n+2)(k_1+1)+2n(k-k_1-1)+2n(k-k_1)+(2n-2)k_1=4nk+2,$ which is a contradiction to \cref{eq5.2}. 
		
		If $r_k=(1,3,5,\ldots,4n-7,4n-5,4n-2,4n,4n+2,4n+4,4n+5,4n+7,\ldots,8n-1),$  then $r_{k+1}$ has the following possibilities: 
		\begin{itemize}
			\item $(2,4,6,\ldots,4n-6,4n-4,4n-2,4n,4n+2,4n+4,4n+6,4n+8,\ldots,8n)$
			\item $(2,4,6,\ldots,4n-6,4n-4,4n-1,4n,4n+3,4n+4,4n+6,4n+8,\ldots,8n).$ 
		\end{itemize} If $r_{k+1}=(2,4,6,\ldots,4n-6,4n-4,4n-1,4n,4n+3,4n+4,4n+6,4n+8,\ldots,8n),$ then $r_1,$ $r_k,$ $r_{k+1}$ and $r_{2k}$ together gives a factor $Y_1$ of $f.$ Assume that $r_{k+1}=(2,4,6,\ldots,4n-6,4n-4,4n-2,4n,4n+2,4n+4,4n+6,4n+8,\ldots,8n).$ Since $r_k$ does not contain $4n-1,$ there exists $l$ for some $k+2 \leq l \leq 2k-1$ such that $r_{l}$ contains $4n-1.$ Hence, $r_l$ is either $(2,4,6,\ldots,4n-6,4n-4,4n-1,4n,4n+3,4n+4,4n+6,4n+8,\ldots,8n)$ or $(2,4,6,\ldots,4n-6,4n-3,4n-1,4n,4n+3,4n+5,4n+6,4n+8,\ldots,8n)$. If $r_l=(2,4,6,\ldots,4n-6,4n-4,4n-1,4n,4n+3,4n+4,4n+6,4n+8,\ldots,8n),$ then $r_1, r_{k}, r_{l}, r_{2k}$ gives a factor $Y_1$ of $f$.
		Now, we consider $r_l=(2,4,6,\ldots,4n-6,4n-3,4n-1,4n,4n+3,4n+5,4n+6,4n+8,\ldots,8n).$
		
		Assume that $k$ be even. Then $r_{\frac{k}{2}+1}$ is one of the following:
		\begin{itemize}
			\item $(1,3,5,\ldots,4n-7,4n-5,4n-4,4n-3,4n-2,4n-1,4n+1,4n+7,\ldots,8n-1)$
			\item $(1,3,5,\ldots,4n-7,4n-5,4n-4,4n-3,4n-2,4n,4n+2,4n+7,\ldots,8n-1)$
			\item $(1,3,5,\ldots,4n-7,4n-5,4n-4,4n-3,4n-1,4n,4n+3,4n+7,\ldots,8n-1)$
			\item $(1,3,5,\ldots,4n-7,4n-5,4n-4,4n-3,4n+1,4n+2,4n+3,4n+7,\ldots,8n-1)$
			\item $(1,3,5,\ldots,4n-7,4n-5,4n-4,4n-2,4n-1,4n,4n+4,4n+7,\ldots,8n-1)$
			\item $(1,3,5,\ldots,4n-7,4n-5,4n-4,4n-2,4n+1,4n+2,4n+4,4n+7,\ldots,8n-1)$
			\item $(1,3,5,\ldots,4n-7,4n-5,4n-4,4n-1,4n+1,4n+2,4n+3,4n+8,\ldots,8n-1)$
			\item $(1,3,5,\ldots,4n-7,4n-5,4n-4,4n,4n+2,4n+3,4n+4,4n+7,\ldots,8n-1)$
			\item $(1,3,5,\ldots,4n-7,4n-5,4n-3,4n-2,4n-1,4n,4n+5,4n+7,\ldots,8n-1)$
			\item $(1,3,5,\ldots,4n-7,4n-5,4n-3,4n-2,4n+1,4n+2,4n+5,4n+7,\ldots,8n-1)$
			\item $(1,3,5,\ldots,4n-7,4n-5,4n-3,4n-1,4n+1,4n+3,4n+5,4n+7,\ldots,8n-1)$
			\item $(1,3,5,\ldots,4n-7,4n-5,4n-3,4n,4n+2,4n+3,4n+5,4n+7,\ldots,8n-1)$
			\item $(1,3,5,\ldots,4n-7,4n-5,4n-2,4n-1,4n+1,4n+4,4n+5,4n+7,\ldots,8n-1)$
			\item $(1,3,5,\ldots,4n-7,4n-5,4n-2,4n,4n+2,4n+4,4n+5,4n+7,\ldots,8n-1).$
		\end{itemize}

		If $r_{\frac{k}{2}+1}$ is one of the following:
		\begin{itemize}
			\item $(1,3,5,\ldots,4n-7,4n-5,4n-4,4n-3,4n-2,4n-1,4n+1,4n+7,\ldots,8n-1),$
			\item $(1,3,5,\ldots,4n-7,4n-5,4n-4,4n-3,4n-2,4n,4n+2,4n+7,\ldots,8n-1),$
			\item $(1,3,5,\ldots,4n-7,4n-5,4n-4,4n-3,4n-1,4n,4n+3,4n+7,\ldots,8n-1),$
			\item $(1,3,5,\ldots,4n-7,4n-5,4n-4,4n-3,4n+1,4n+2,4n+3,4n+7,\ldots,8n-1),$
		\end{itemize}
		then $\sum_{i=4n-4}^{4n-3}\sum_{j=2n-2}^{2n}(N_{i,j}) \geq 2k+2,$ which is a contradiction to \cref{eq5.2}. 
		
		If $r_{\frac{k}{2}+1}=(1,3,5,\ldots,4n-7,4n-5,4n-4,4n-2,4n-1,4n,4n+4,4n+7,\ldots,8n-1),$ then $r_{\frac{k}{2}+1}, r_{2k}$ gives a factor $X_6$ of $f.$
		
		If $r_{\frac{k}{2}+1}=(1,3,5,\ldots,4n-7,4n-5,4n-4,4n-2,4n+1,4n+2,4n+4,4n+7,\ldots,8n-1),$ then $r_{\frac{k}{2}+1}, r_{l}$ gives a factor $X_4$ of $f.$
		
		If $r_{\frac{k}{2}+1}=(1,3,5,\ldots,4n-7,4n-5,4n-4,4n-1,4n+1,4n+2,4n+3,4n+8,\ldots,8n-1),$ then $r_1,r_{\frac{k}{2}+1},r_{k}, r_{k+1}, r_{l},r_{2k}$ gives a factor $Z_1$ of $f.$
		
		If $r_{\frac{k}{2}+1}=(1,3,5,\ldots,4n-7,4n-5,4n-4,4n,4n+2,4n+3,4n+4,4n+7,\ldots,8n-1),$ then we consider $r_{\frac{3k}{2}}.$ Note that $E_{\frac{3k}{2}, 2n-2}$ can not be $4n-4,$ otherwise it contradicts the number of appearances of $4n-4$ in $\Gamma.$ Thus, $r_{\frac{3k}{2}}$ is one of the following:
		\begin{itemize}
			\item 	$ (2,4,6,\ldots,4n-6,4n-3,4n+1,4n+2,4n+3,4n+5,4n+6,4n+8,\ldots,8n),$
			\item $(2,4,6,\ldots,4n-6,4n-3,4n-1,4n,4n+3,4n+5,4n+6,4n+8,\ldots,8n),$
			\item $(2,4,6,\ldots,4n-6,4n-3,4n-2,4n,4n+2,4n+5,4n+6,4n+8,\ldots,8n)$.
		\end{itemize}
		If $r_{\frac{3k}{2}}=(2,4,6,\ldots,4n-6,4n-3,4n+1,4n+2,4n+3,4n+5,4n+6,4n+8,\ldots,8n),$ then $E_{i,2n}=4n+2$ for all $\frac{3k}{2} \leq i \leq 2k$ and $E_{i,2n+1}=4n+2$ for all $\frac{k}{2}+1 \leq i \leq k+1.$ Thus $N_{4n+2,2n}+N_{4n+2,2n+1} \geq k+2,$ which is a contradiction to \cref{eq5.2}.  If $r_{\frac{3k}{2}}=(2,4,6,\ldots,4n-6,4n-3,4n-1,4n,4n+3,4n+5,4n+6,4n+8,\ldots,8n)$ or $(2,4,6,\ldots,4n-6,4n-3,4n-2,4n,4n+2,4n+5,4n+6,4n+8,\ldots,8n),$ then $E_{i,2n}=4n$ for all $\frac{k}{2}+1 \leq i \leq \frac{3k}{2}.$ Thus, the total number of integers less than or equal to $4n$ appearing between $r_{\frac{k}{2}+1}$ and $r_{2k}$ is $2nk+\frac{k}{2}(2n-2)=3nk-k.$  Hence,  the total number of integers less than or equal to $4n$ appearing between $r_{1}$ and $r_{\frac{k}{2}}$ is $4nk-(3nk-k)=nk+k=\frac{k}{2}(2n+2).$ Thus, each row between $r_1$ and $r_{\frac{k}{2}}$ contains exactly $2n+2$ number of entries less than or equal to $4n.$ Therefore,  $r_i=(1,3,5,\ldots,4n-7,4n-5,4n-4,4n-3,4n-2,4n-1,4n+1,4n+7,\ldots,8n-1)$ for all $1 \leq i \leq \frac{k}{2}.$ Thus, $N_{4n-4,2n-2}+N_{4n-3,2n-2}+N_{4n-4,2n+1}+N_{4n-3,2n} \geq 2k+1,$ which is a contradiction to \cref{eq5.2}.
		
		If $r_{\frac{k}{2}+1}=(1,3,5,\ldots,4n-7,4n-5,4n-3,4n-2,4n-1,4n,4n+5,4n+7,\ldots,8n-1),$ then we consider $r_{\frac{3k}{2}}.$ Note that $E_{\frac{3k}{2}, 2n-2}$ can not be $4n-3,$ otherwise it contradicts the number of appearances of $4n-4$ in $\Gamma.$ Thus, $r_{\frac{3k}{2}}$ is one of the following:
		\begin{itemize}
			\item 	$ (2,4,6,\ldots,4n-6,4n-4,4n-2,4n,4n+2,4n+4,4n+6,4n+8,\ldots,8n)$
			\item $ (2,4,6,\ldots,4n-6,4n-4,4n-1,4n,4n+2,4n+4,4n+6,4n+8,\ldots,8n).$
		\end{itemize}
		If $r_{\frac{3k}{2}}=(2,4,6,\ldots,4n-6,4n-4,4n-2,4n,4n+2,4n+4,4n+6,4n+8,\ldots,8n),$ then $\sum_{i=4n-4}^{4n-2}\sum_{j=2n-2}^{2n}N_{i,j}\geq k+\frac{3k}{2}+\frac{k}{2}+1=3k+1,$ which is a contradiction to \cref{eq5.2}. If $r_{\frac{3k}{2}}=(2,4,6,\ldots,4n-6,4n-4,4n-1,4n,4n+2,4n+4,4n+6,4n+8,\ldots,8n),$ then $r_1, r_{k}, r_{\frac{3k}{2}}, r_{2k}$ gives a factor $Y_1$ of $f.$
		
		If $r_{\frac{k}{2}+1}=(1,3,5,\ldots,4n-7,4n-5,4n-3,4n-2,4n+1,4n+2,4n+5,4n+7,\ldots,8n-1),$  then $r_{\frac{3k}{2}}$ is either $(2,4,6,\ldots,4n-6,4n-4,4n-2,4n,4n+2,4n+4,4n+6,4n+8,\ldots,8n)$ or $(2,4,6,\ldots,4n-6,4n-4,4n-1,4n,4n+3,4n+4,4n+6,4n+8,\ldots,8n).$  Further, If $r_{\frac{3k}{2}}=(2,4,6,\ldots,4n-6,4n-4,4n-2,4n,4n+2,4n+4,4n+6,4n+8,\ldots,8n),$ then $\sum_{i=4n-4}^{4n-2}N_{i,2n-2}+N_{i,2n-1}+N_{i,2n} \geq 3k+1,$ which is a contradiction to \cref{eq5.2}. On the other hand, if $r_{\frac{3k}{2}}=(2,4,6,\ldots,4n-6,4n-4,4n-1,4n,4n+3,4n+4,4n+6,4n+8,\ldots,8n),$ then $r_{\frac{k}{2}+1}, r_{\frac{3k}{2}}$ gives a factor $X_{3}$  of $f.$ 
		
		If $r_{\frac{k}{2}+1}=(1,3,5,\ldots,4n-7,4n-5,4n-3,4n-1,4n+1,4n+3,4n+5,4n+7,\ldots,8n-1),$ then $r_{\frac{k}{2}+1}, r_{k+1}$ gives a factor $X_2$ of $f.$ 
		
		If $r_{\frac{k}{2}+1}=(1,3,5,\ldots,4n-7,4n-5,4n-3,4n,4n+2,4n+3,4n+5,4n+7,\ldots,8n-1),$ then $r_{\frac{3k}{2}}$ is either $(2,4,6,\ldots,4n-6,4n-4,4n-2,4n,4n+2,4n+4,4n+6,4n+8,\ldots,8n)$ or $(2,4,6,\ldots,4n-6,4n-4,4n-1,4n,4n+3,4n+4,4n+6,4n+8,\ldots,8n)$. In both cases $E_{i,2n}=4n$ for all $\frac{k}{2}+1 \leq i \leq \frac{3k}{2}.$ Thus, $r_i=(1,3,5,\ldots,4n-7,4n-5,4n-4,4n-3,4n-2,4n-1,4n+1,4n+7,\ldots,8n-1)$ for all $1 \leq i \leq \frac{k}{2}.$ Thus, $N_{4n-4,2n-2}+N_{4n-3,2n-2}+N_{4n-4,2n+1}+N_{4n-3,2n} \geq 2k+1,$ which is a contradiction to \cref{eq5.2}
		
		If $r_{\frac{k}{2}+1}=(1,3,5,\ldots,4n-7,4n-5,4n-2,4n-1,4n+2,4n+3,4n+5,4n+7,\ldots,8n-1),$ then $r_{\frac{3k}{2}}$ is either  $(2,4,6,\ldots,4n-6,4n-4,4n-2,4n,4n+2,4n+4,4n+6,4n+8,\ldots,8n)$ or $(2,4,6,\ldots,4n-6,4n-4,4n-1,4n,4n+3,4n+4,4n+6,4n+8,\ldots,8n)$. If $r_{\frac{3k}{2}}=(2,4,6,\ldots,4n-6,4n-4,4n-2,4n,4n+2,4n+4,4n+6,4n+8,\ldots,8n),$ then $N_{4n-2,2n-1}+N_{4n-2,2n+1} \geq k+1,$ which is a contradiction to \cref{eq5.2}. If $r_{\frac{3k}{2}}=(2,4,6,\ldots,4n-6,4n-4,4n-1,4n,4n+3,4n+4,4n+6,4n+8,\ldots,8n),$ then $r_1, r_{k}, r_{\frac{3k}{2}}, r_{2k}$ gives a factor $Y_1$ of $f.$
		
		If $r_{\frac{k}{2}+1}=(1,3,5,\ldots,4n-7,4n-5,4n-2,4n,4n+2,4n+4,4n+5,4n+7,\ldots,8n-1),$ then $r_{\frac{3k}{2}}$ is either $(2,4,6,\ldots,4n-6,4n-4,4n-2,4n,4n+2,4n+4,4n+6,4n+8,\ldots,8n)$ or $(2,4,6,\ldots,4n-6,4n-4,4n-1,4n,4n+3,4n+4,4n+6,4n+8,\ldots,8n).$  If $r_{\frac{3k}{2}}=(2,4,6,\ldots,4n-6,4n-4,4n-2,4n,4n+2,4n+4,4n+6,4n+8,\ldots,8n)$ then $N_{4n-2,2n-1}+N_{4n-2,2n+1} \geq k+1,$ which is a contradiction to \cref{eq5.2}. If $r_{\frac{3k}{2}}=(2,4,6,\ldots,4n-6,4n-4,4n-1,4n,4n+3,4n+4,4n+6,4n+8,\ldots,8n),$ then $r_1, r_{k}, r_{\frac{3k}{2}}, r_{2k}$ gives a factor $Y_1$  of $f.$
		
		If $k+1$ is even, then the proof is similar. 
		
		\textbf{Subcase (iv):} Assume that $r_{1}=(1,3,5,\ldots,4n-7,4n-5,4n-4,4n-3,4n-2,4n,4n+2,4n+7,\ldots,8n-1).$ Then the proof is similar to the proof of Subcase (iii).  	
	\end{proof}
	
	\begin{lemma}\label{lem5.2}
		The following relations among $X_i$'s $(1 \leq i \leq 6),$ $Y_j$'s ($1 \leq j \leq 4$) hold in $R_2.$
		\begin{itemize}
			\item[(i)] $X_4X_5-X_3X_6+Y_2-Y_1=0.$
			\item[(ii)] $X_2X_5-X_1X_6+Y_3-Y_1=0.$
			\item[(iii)] $X_2X_3-X_1X_4+Y_4-Y_1=0.$
		\end{itemize}
	\end{lemma}
	\begin{proof} Proof of (i) : Note that
		
		$X_4X_5=$
		
		\tiny{\ytableausetup{boxsize=3.5em} \ytableausetup{centertableaux}
			\begin{ytableau}
				1 & 3 & 5  & \cdots & 4n-7 & 4n-5 & 4n-3 & 4n-2 & 4n-1 & 4n & 4n+5 & 4n+7 & 4n+9 & \cdots  & 8n-3 & 8n-1\\
				1 & 3 & 5  & \cdots & 4n-7 & 4n-5 & 4n-4 & 4n-2 & 4n+1 & 4n+2 & 4n+4 & 4n+7 & 4n+9 & \cdots  & 8n-3 & 8n-1\\
				2 & 4 & 6 & \cdots & 4n-6 & 4n-3 & 4n-1 & 4n & 4n+3 & 4n+5 & 4n+6 & 4n+8 & 4n+10 & \cdots & 8n-2 & 8n\\
				2 & 4 & 6 & \cdots & 4n-6 & 4n-4 & 4n+1 & 4n+2 & 4n+3 & 4n+4 & 4n+6 & 4n+8 & 4n+10 & \cdots & 8n-2 & 8n 
		\end{ytableau}}.
		
		\normalsize{Now we see that}
		
		\tiny{\ytableausetup{boxsize=3.5em} \ytableausetup{centertableaux}
			\begin{ytableau}
				1 & 3 & 5  & \cdots & 4n-7 & 4n-5 & 4n-3 & 4n-2 & 4n-1 & 4n & 4n+5 & 4n+7 & 4n+9 & \cdots  & 8n-3 & 8n-1\\
				1 & 3 & 5  & \cdots & 4n-7 & 4n-5 & 4n-4 & 4n-2 & 4n+1 & 4n+2 & 4n+4 & 4n+7 & 4n+9 & \cdots  & 8n-3 & 8n-1
		\end{ytableau}}
		
		\normalsize{is a non-standard monomial. We apply  \cref{thm2.6}, \cref{rmk2.7} and  \cref{rmk2.8} to express the above non-standard monomial as a linear combination of standard monomials.}
		
		Let $\underline{i}_{1}$=$(1, 3 ,5, \ldots, 4n-7, 4n-5, 4n-3, 4n-2, 4n-1, 4n,4n+5, 4n+7, 4n+9, \cdots, 8n-3, 8n-1),$ and  $\underline{i}_{2}$=$(1, 3 ,5, \cdots, 4n-7, 4n-5, 4n-4, 4n-2, 4n+1, 4n+2,4n+4, 4n+7, 4n+9, \ldots, 8n-3, 8n-1).$ Then $\underline{i}_{1}(B)=\{2,4,6,\ldots, 4n-8,4n-6, 4n-4\}$ and $\underline{i}_{2}(B)=\{2,4,6,\ldots, 4n-8,4n-6, 4n-3, 4n-1,4n\}$ are two subsets of $\{1,2,3,\ldots, 4n\}.$ 
		
		In order to apply Theorem \ref{thm2.6}, set $I_{1}=\{2,4,6,\ldots, 4n-8,4n-6\}$ and $I_{2}=\{2,4,6,\ldots,4n-8,4n-6,4n-4,4n-3,4n-1,4n\}.$ Note that $I_{1}\Delta I_{2}=\{4n-4,4n-3,4n-1,4n\}.$ Then we have 
		
		$P(I_{1}\Delta\{4n-4\})\cdot P(I_{2}\Delta \{4n-4\})-P(I_{1}\Delta\{4n-3\})\cdot P(I_{2}\Delta \{4n-3\})+P(I_{1}\Delta\{4n-1\})\cdot P(I_{2}\Delta \{4n-1\})-P(I_{1}\Delta\{4n\})\cdot P(I_{2}\Delta \{4n\})=0.$
		
		By using \cref{rmk2.8}, we have the following:
		
		\item $P(I_{1}\Delta\{4n-4\})=q_{\underline{i}_1}$ 
		\item $P(I_{2}\Delta \{4n-4\}=q_{\underline{i}_2}$ 
		\item $P(I_{1}\Delta\{4n-3\}=q_{(1, 3, 5, \ldots, 4n-7, 4n-5, 4n-4, 4n-2, 4n-1, 4n, 4n+4, 4n+7, \ldots,  8n-1)}$
		\item $P(I_{2}\Delta\{4n-3\}=q_{(1, 3, 5, \ldots, 4n-7, 4n-5, 4n-3, 4n-2, 4n+1, 4n+2, 4n+5, 4n+7, \ldots,  8n-1)}$
		\item $P(I_{1}\Delta\{4n-1\}=q_{(1, 3, 5, \ldots, 4n-7, 4n-5, 4n-4, 4n-3, 4n-2, 4n, 4n+2, 4n+7, \ldots,  8n-1)}$
		\item $P(I_{2}\Delta\{4n-1\}=q_{(1, 3, 5, \ldots, 4n-7, 4n-5, 4n-2, 4n-1, 4n+1, 4n+4, 4n+5, 4n+7, \ldots,  8n-1)}$
		\item $P(I_{1}\Delta\{4n\}=q_{(1, 3, 5, \ldots, 4n-7, 4n-5, 4n-4, 4n-3, 4n-2, 4n-1, 4n+1, 4n+7, \ldots,  8n-1)}$
		\item $P(I_{2}\Delta\{4n\}=q_{(1, 3, 5, \ldots, 4n-7, 4n-5, 4n-2, 4n, 4n+2, 4n+4, 4n+5, 4n+7, \ldots,  8n-1)}.$

		Therefore, we have the following straightening laws in $G/P^{\alpha_{4n}}$
		
		\tiny{\ytableausetup{boxsize=3.5em} \ytableausetup{centertableaux}
			\begin{ytableau}
				1 & 3 & 5  & \cdots & 4n-7 & 4n-5 & 4n-3 & 4n-2 & 4n-1 & 4n & 4n+5 & 4n+7 & 4n+9 & \cdots  & 8n-3 & 8n-1\\
				1 & 3 & 5  & \cdots & 4n-7 & 4n-5 & 4n-4 & 4n-2 & 4n+1 & 4n+2 & 4n+4 & 4n+7 & 4n+9 & \cdots  & 8n-3 & 8n-1
			\end{ytableau}\\
			$=$\hspace{.2cm}\begin{ytableau}
				1 & 3 & 5  & \cdots & 4n-7 & 4n-5 & 4n-4 & 4n-3 & 4n-2 & 4n-1 & 4n+1 & 4n+7 & 4n+9 & \cdots  & 8n-3 & 8n-1\\
				1 & 3 & 5  & \cdots & 4n-7 & 4n-5 & 4n-2 & 4n & 4n+2 & 4n+4 & 4n+5 & 4n+7 & 4n+9 & \cdots  & 8n-3 & 8n-1
			\end{ytableau}\\
			$-$~\hspace{.15cm}\begin{ytableau}
				1 & 3 & 5  & \cdots & 4n-7 & 4n-5 & 4n-4 & 4n-3 & 4n-2 & 4n & 4n+2 & 4n+7 & 4n+9 & \cdots  & 8n-3 & 8n-1\\
				1 & 3 & 5  & \cdots & 4n-7 & 4n-5 & 4n-2 & 4n-1 & 4n+1 & 4n+4 & 4n+5 & 4n+7 & 4n+9 & \cdots  & 8n-3 & 8n-1
			\end{ytableau}\\
			$+$~\hspace{.15cm}\begin{ytableau}
				1 & 3 & 5  & \cdots & 4n-7 & 4n-5 & 4n-4 & 4n-2 & 4n-1 & 4n & 4n+4 & 4n+7 & 4n+9 & \cdots  & 8n-3 & 8n-1\\
				1 & 3 & 5  & \cdots & 4n-7 & 4n-5 & 4n-3 & 4n-2 & 4n+1 & 4n+2 & 4n+5 & 4n+7 & 4n+9 & \cdots  & 8n-3 & 8n-1
			\end{ytableau}.}
		
		\normalsize{Similarly,} by using Theorem \ref{thm2.6}, Remark \ref{rmk2.7}, and Remark\ref{rmk2.8} we have
		
		\tiny{\ytableausetup{boxsize=3.5em} \ytableausetup{centertableaux}
			\begin{ytableau}
				2 & 4 & 6 & \cdots & 4n-6 & 4n-3 & 4n-1 & 4n & 4n+3 & 4n+5 & 4n+6 & 4n+8 & 4n+10 & \cdots & 8n-2 & 8n\\
				2 & 4 & 6 & \cdots & 4n-6 & 4n-4 & 4n+1 & 4n+2 & 4n+3 & 4n+4 & 4n+6 & 4n+8 & 4n+10 & \cdots & 8n-2 & 8n 
			\end{ytableau}\\
			$=$ \hspace{.15cm}\begin{ytableau}
				2 & 4 & 6 & \cdots & 4n-6 & 4n-4 & 4n-1 & 4n & 4n+3 & 4n+4 & 4n+6 & 4n+8 & 4n+10 & \cdots & 8n-2 & 8n\\
				2 & 4 & 6 & \cdots & 4n-6 & 4n-3 & 4n+1 & 4n+2 & 4n+3 & 4n+5 & 4n+6 & 4n+8 & 4n+10 & \cdots & 8n-2 & 8n 
			\end{ytableau}\\
			$+$  \hspace{.15cm}\begin{ytableau}
				2 & 4 & 6 & \cdots & 4n-6 & 4n-4 & 4n-3 & 4n-1 & 4n+1 & 4n+3 & 4n+6 & 4n+8 & 4n+10 & \cdots & 8n-2 & 8n \\
				2 & 4 & 6 & \cdots & 4n-6 & 4n & 4n+2 & 4n+3 & 4n+4 & 4n+5 & 4n+6 & 4n+8 & 4n+10 & \cdots & 8n-2 & 8n
			\end{ytableau}\\
			$-$  \hspace{.15cm}\begin{ytableau}
				2 & 4 & 6 & \cdots & 4n-6 & 4n-4 & 4n-3 & 4n & 4n+1 & 4n+3 & 4n+6 & 4n+8 & 4n+10 & \cdots & 8n-2 & 8n \\
				2 & 4 & 6 & \cdots & 4n-6 & 4n-1 & 4n+1 & 4n+3 & 4n+4 & 4n+5 & 4n+6 & 4n+8 & 4n+10 & \cdots & 8n-2 & 8n
			\end{ytableau}.}
		
		\normalsize{Since} we are working inside $X(w_{6})$, the above straightening law becomes
		
		\tiny{\ytableausetup{boxsize=3.5em} \ytableausetup{centertableaux}
			\begin{ytableau}
				2 & 4 & 6 & \cdots & 4n-6 & 4n-3 & 4n-1 & 4n & 4n+3 & 4n+5 & 4n+6 & 4n+8 & 4n+10 & \cdots & 8n-2 & 8n\\
				2 & 4 & 6 & \cdots & 4n-6 & 4n-4 & 4n+1 & 4n+2 & 4n+3 & 4n+4 & 4n+6 & 4n+8 & 4n+10 & \cdots & 8n-2 & 8n 
			\end{ytableau}\\
			$=$ \hspace{.15cm}\begin{ytableau}
				2 & 4 & 6 & \cdots & 4n-6 & 4n-4 & 4n-1 & 4n & 4n+3 & 4n+4 & 4n+6 & 4n+8 & 4n+10 & \cdots & 8n-2 & 8n\\
				2 & 4 & 6 & \cdots & 4n-6 & 4n-3 & 4n+1 & 4n+2 & 4n+3 & 4n+5 & 4n+6 & 4n+8 & 4n+10 & \cdots & 8n-2 & 8n 
		\end{ytableau}}
		
		\normalsize{Therefore,} by using the above straightening laws we have
		
		$X_4X_5=$
		
		\tiny{\ytableausetup{boxsize=3.5em} \ytableausetup{centertableaux}
			\hspace{.2cm}\begin{ytableau}
				1 & 3 & 5  & \cdots & 4n-7 & 4n-5 & 4n-4 & 4n-3 & 4n-2 & 4n-1 & 4n+1 & 4n+7 & 4n+9 & \cdots  & 8n-3 & 8n-1\\
				1 & 3 & 5  & \cdots & 4n-7 & 4n-5 & 4n-2 & 4n & 4n+2 & 4n+4 & 4n+5 & 4n+7 & 4n+9 & \cdots  & 8n-3 & 8n-1\\
				2 & 4 & 6 & \cdots & 4n-6 & 4n-4 & 4n-1 & 4n & 4n+3 & 4n+4 & 4n+6 & 4n+8 & 4n+10 & \cdots & 8n-2 & 8n\\
				2 & 4 & 6 & \cdots & 4n-6 & 4n-3 & 4n+1 & 4n+2 & 4n+3 & 4n+5 & 4n+6 & 4n+8 & 4n+10 & \cdots & 8n-2 & 8n 
			\end{ytableau}\\
			$-$~\hspace{.15cm}\hspace{.15cm}\begin{ytableau}
				1 & 3 & 5  & \cdots & 4n-7 & 4n-5 & 4n-4 & 4n-3 & 4n-2 & 4n & 4n+2 & 4n+7 & 4n+9 & \cdots  & 8n-3 & 8n-1\\
				1 & 3 & 5  & \cdots & 4n-7 & 4n-5 & 4n-2 & 4n-1 & 4n+1 & 4n+4 & 4n+5 & 4n+7 & 4n+9 & \cdots  & 8n-3 & 8n-1\\
				2 & 4 & 6 & \cdots & 4n-6 & 4n-4 & 4n-1 & 4n & 4n+3 & 4n+4 & 4n+6 & 4n+8 & 4n+10 & \cdots & 8n-2 & 8n\\
				2 & 4 & 6 & \cdots & 4n-6 & 4n-3 & 4n+1 & 4n+2 & 4n+3 & 4n+5 & 4n+6 & 4n+8 & 4n+10 & \cdots & 8n-2 & 8n 
			\end{ytableau}\\
			$+$~\hspace{.15cm}\hspace{.15cm}\begin{ytableau}
				1 & 3 & 5  & \cdots & 4n-7 & 4n-5 & 4n-4 & 4n-2 & 4n-1 & 4n & 4n+4 & 4n+7 & 4n+9 & \cdots  & 8n-3 & 8n-1\\
				1 & 3 & 5  & \cdots & 4n-7 & 4n-5 & 4n-3 & 4n-2 & 4n+1 & 4n+2 & 4n+5 & 4n+7 & 4n+9 & \cdots  & 8n-3 & 8n-1\\
				2 & 4 & 6 & \cdots & 4n-6 & 4n-4 & 4n-1 & 4n & 4n+3 & 4n+4 & 4n+6 & 4n+8 & 4n+10 & \cdots & 8n-2 & 8n\\
				2 & 4 & 6 & \cdots & 4n-6 & 4n-3 & 4n+1 & 4n+2 & 4n+3 & 4n+5 & 4n+6 & 4n+8 & 4n+10 & \cdots & 8n-2 & 8n 
		\end{ytableau}}
		
		\normalsize{$=Y_1-Y_2+X_3X_6.$}
		
		Proof of (ii): Note that 
		
		$X_2X_5=$
		
		\tiny{\ytableausetup{boxsize=3.5em}  \ytableausetup{centertableaux}
			\begin{ytableau}
				1 & 3 & 5  & \cdots & 4n-7 & 4n-5 & 4n-4 & 4n-1 & 4n+1 & 4n+3 & 4n+4 & 4n+7 & 4n+9 & \cdots & 8n-3 & 8n-1\\
				1 & 3 & 5  & \cdots & 4n-7 & 4n-5 & 4n-3 & 4n-2 & 4n-1 & 4n & 4n+5 & 4n+7 & 4n+9 & \cdots  & 8n-3 & 8n-1\\
				2 & 4 & 6 & \cdots & 4n-6 & 4n-3 & 4n-2 & 4n & 4n+2 & 4n+5 & 4n+6 & 4n+8 & 4n+10 & \cdots & 8n-2 & 8n\\
				2 & 4 & 6 & \cdots & 4n-6 & 4n-4 & 4n+1 & 4n+2 & 4n+3 & 4n+4 & 4n+6 & 4n+8 & 4n+10 & \cdots & 8n-2 & 8n 
		\end{ytableau}}
		
		\normalsize{By using Theorem \ref{thm2.6} and Remark \ref{rmk2.8}, we have the following straightening laws in $G/P^{\alpha_{4n}}$}
		
		\tiny{\ytableausetup{boxsize=3.5em}  \ytableausetup{centertableaux}
			\begin{ytableau}
				1 & 3 & 5  & \cdots & 4n-7 & 4n-5 & 4n-4 & 4n-1 & 4n+1 & 4n+3 & 4n+4 & 4n+7 & 4n+9 & \cdots & 8n-3 & 8n-1\\
				1 & 3 & 5  & \cdots & 4n-7 & 4n-5 & 4n-3 & 4n-2 & 4n-1 & 4n & 4n+5 & 4n+7 & 4n+9 & \cdots  & 8n-3 & 8n-1
			\end{ytableau}\\
			$=$\hspace{.15cm} \begin{ytableau}
				1 & 3 & 5  & \cdots & 4n-7 & 4n-5 & 4n-4 & 4n-2 & 4n-1 & 4n & 4n+4 & 4n+7 & 4n+9 & \cdots & 8n-3 & 8n-1\\
				1 & 3 & 5  & \cdots & 4n-7 & 4n-5 & 4n-3 & 4n-1 & 4n+1 & 4n+3 & 4n+5 & 4n+7 & 4n+9 & \cdots  & 8n-3 & 8n-1
			\end{ytableau}\\
			$-$\hspace{.15cm} \begin{ytableau}
				1 & 3 & 5  & \cdots & 4n-7 & 4n-5 & 4n-4 & 4n-3 & 4n-1 & 4n & 4n+3 & 4n+7 & 4n+9 & \cdots & 8n-3 & 8n-1\\
				1 & 3 & 5  & \cdots & 4n-7 & 4n-5 & 4n-2 & 4n-1 & 4n+1 & 4n+4 & 4n+5 & 4n+7 & 4n+9 & \cdots  & 8n-3 & 8n-1
			\end{ytableau}\\
			$+$\hspace{.15cm} \begin{ytableau}
				1 & 3 & 5  & \cdots & 4n-7 & 4n-5 & 4n-4 & 4n-3 & 4n-2 & 4n-1 & 4n+1 & 4n+7 & 4n+9 & \cdots & 8n-3 & 8n-1\\
				1 & 3 & 5  & \cdots & 4n-7 & 4n-5 & 4n-1 & 4n & 4n+3 & 4n+4 & 4n+5 & 4n+7 & 4n+9 & \cdots  & 8n-3 & 8n-1
		\end{ytableau}}
		
		\normalsize{and}
		
		\tiny{\ytableausetup{boxsize=3.5em}  \ytableausetup{centertableaux}
			\begin{ytableau}
				2 & 4 & 6 & \cdots & 4n-6 & 4n-3 & 4n-2 & 4n & 4n+2 & 4n+5 & 4n+6 & 4n+8 & 4n+10 & \cdots & 8n-2 & 8n\\
				2 & 4 & 6 & \cdots & 4n-6 & 4n-4 & 4n+1 & 4n+2 & 4n+3 & 4n+4 & 4n+6 & 4n+8 & 4n+10 & \cdots & 8n-2 & 8n
			\end{ytableau}\\
			$=$\hspace{.15cm} \begin{ytableau}
				2 & 4 & 6 & \cdots & 4n-6 & 4n-4 & 4n-3 & 4n-2 & 4n+1 & 4n+2 & 4n+6 & 4n+8 & 4n+10 & \cdots & 8n-2 & 8n\\
				2 & 4 & 6 & \cdots & 4n-6 & 4n & 4n+2 & 4n+3 & 4n+4 & 4n+5 & 4n+6 & 4n+8 & 4n+10 & \cdots & 8n-2 & 8n
			\end{ytableau}\\
			$-$\hspace{.15cm} \begin{ytableau}
				2 & 4 & 6 & \cdots & 4n-6 & 4n-4 & 4n-3 & 4n & 4n+2 & 4n+3 & 4n+6 & 4n+8 & 4n+10 & \cdots & 8n-2 & 8n\\
				2 & 4 & 6 & \cdots & 4n-6 & 4n-2 & 4n+1 & 4n+2 & 4n+4 & 4n+5 & 4n+6 & 4n+8 & 4n+10 & \cdots & 8n-2 & 8n
			\end{ytableau}\\
			$+$\hspace{.15cm} \begin{ytableau}
				2 & 4 & 6 & \cdots & 4n-6 & 4n-4 & 4n-2 & 4n & 4n+2 & 4n+4 & 4n+6 & 4n+8 & 4n+10 & \cdots & 8n-2 & 8n\\
				2 & 4 & 6 & \cdots & 4n-6 & 4n-3 & 4n+1 & 4n+2 & 4n+3 & 4n+5 & 4n+6 & 4n+8 & 4n+10 & \cdots & 8n-2 & 8n
		\end{ytableau}}
		
		\normalsize{Since} we are working in $X(w_6),$ the above straightening law becomes
		
		\tiny{\ytableausetup{boxsize=3.5em}  \ytableausetup{centertableaux}
			\begin{ytableau}
				2 & 4 & 6 & \cdots & 4n-6 & 4n-3 & 4n-2 & 4n & 4n+2 & 4n+5 & 4n+6 & 4n+8 & 4n+10 & \cdots & 8n-2 & 8n\\
				2 & 4 & 6 & \cdots & 4n-6 & 4n-4 & 4n+1 & 4n+2 & 4n+3 & 4n+4 & 4n+6 & 4n+8 & 4n+10 & \cdots & 8n-2 & 8n
			\end{ytableau}\\
			$=$ \hspace{.1cm} \begin{ytableau}
				2 & 4 & 6 & \cdots & 4n-6 & 4n-4 & 4n-2 & 4n & 4n+2 & 4n+4 & 4n+6 & 4n+8 & 4n+10 & \cdots & 8n-2 & 8n\\
				2 & 4 & 6 & \cdots & 4n-6 & 4n-3 & 4n+1 & 4n+2 & 4n+3 & 4n+5 & 4n+6 & 4n+8 & 4n+10 & \cdots & 8n-2 & 8n
		\end{ytableau}}
		
		\normalsize{Hence,} we have
		
		$X_2X_5=$
		
		\tiny{\ytableausetup{boxsize=3.5em}  \ytableausetup{centertableaux}
			\begin{ytableau}
				1 & 3 & 5  & \cdots & 4n-7 & 4n-5 & 4n-4 & 4n-2 & 4n-1 & 4n & 4n+4 & 4n+7 & 4n+9 & \cdots & 8n-3 & 8n-1\\
				1 & 3 & 5  & \cdots & 4n-7 & 4n-5 & 4n-3 & 4n-1 & 4n+1 & 4n+3 & 4n+5 & 4n+7 & 4n+9 & \cdots  & 8n-3 & 8n-1\\
				2 & 4 & 6 & \cdots & 4n-6 & 4n-4 & 4n-2 & 4n & 4n+2 & 4n+4 & 4n+6 & 4n+8 & 4n+10 & \cdots & 8n-2 & 8n\\
				2 & 4 & 6 & \cdots & 4n-6 & 4n-3 & 4n+1 & 4n+2 & 4n+3 & 4n+5 & 4n+6 & 4n+8 & 4n+10 & \cdots & 8n-2 & 8n
			\end{ytableau}\\
			$-$\hspace{.15cm} \begin{ytableau}
				1 & 3 & 5  & \cdots & 4n-7 & 4n-5 & 4n-4 & 4n-3 & 4n-1 & 4n & 4n+3 & 4n+7 & 4n+9 & \cdots & 8n-3 & 8n-1\\
				1 & 3 & 5  & \cdots & 4n-7 & 4n-5 & 4n-2 & 4n-1 & 4n+1 & 4n+4 & 4n+5 & 4n+7 & 4n+9 & \cdots  & 8n-3 & 8n-1\\
				2 & 4 & 6 & \cdots & 4n-6 & 4n-4 & 4n-2 & 4n & 4n+2 & 4n+4 & 4n+6 & 4n+8 & 4n+10 & \cdots & 8n-2 & 8n\\
				2 & 4 & 6 & \cdots & 4n-6 & 4n-3 & 4n+1 & 4n+2 & 4n+3 & 4n+5 & 4n+6 & 4n+8 & 4n+10 & \cdots & 8n-2 & 8n
			\end{ytableau}\\
			$+$\hspace{.15cm} \begin{ytableau}
				1 & 3 & 5  & \cdots & 4n-7 & 4n-5 & 4n-4 & 4n-3 & 4n-2 & 4n-1 & 4n+1 & 4n+7 & 4n+9 & \cdots & 8n-3 & 8n-1\\
				1 & 3 & 5  & \cdots & 4n-7 & 4n-5 & 4n-1 & 4n & 4n+3 & 4n+4 & 4n+5 & 4n+7 & 4n+9 & \cdots  & 8n-3 & 8n-1\\
				2 & 4 & 6 & \cdots & 4n-6 & 4n-4 & 4n-2 & 4n & 4n+2 & 4n+4 & 4n+6 & 4n+8 & 4n+10 & \cdots & 8n-2 & 8n\\
				2 & 4 & 6 & \cdots & 4n-6 & 4n-3 & 4n+1 & 4n+2 & 4n+3 & 4n+5 & 4n+6 & 4n+8 & 4n+10 & \cdots & 8n-2 & 8n
		\end{ytableau}}
		
		\normalsize{Since} we are working in $X(w_6),$ we have 
		
		\tiny{\ytableausetup{boxsize=3.5em}  \ytableausetup{centertableaux}
			\begin{ytableau}
				1 & 3 & 5  & \cdots & 4n-7 & 4n-5 & 4n-1 & 4n & 4n+3 & 4n+4 & 4n+5 & 4n+7 & 4n+9 & \cdots  & 8n-3 & 8n-1\\
				2 & 4 & 6 & \cdots & 4n-6 & 4n-4 & 4n-2 & 4n & 4n+2 & 4n+4 & 4n+6 & 4n+8 & 4n+10 & \cdots & 8n-2 & 8n
			\end{ytableau}\\
			$=$\hspace{.15cm} \begin{ytableau}
				1 & 3 & 5  & \cdots & 4n-7 & 4n-5 & 4n-2 & 4n & 4n+2 & 4n+4 & 4n+5 & 4n+7 & 4n+9 & \cdots  & 8n-3 & 8n-1\\
				2 & 4 & 6 & \cdots & 4n-6 & 4n-4 & 4n-1 & 4n & 4n+3 & 4n+4 & 4n+6 & 4n+8 & 4n+10 & \cdots & 8n-2 & 8n
		\end{ytableau}}
		
		\normalsize{Hence,} we have
		
		$X_2X_5=$
		
		\tiny{\ytableausetup{boxsize=3.5em}  \ytableausetup{centertableaux}
			\begin{ytableau}
				1 & 3 & 5  & \cdots & 4n-7 & 4n-5 & 4n-4 & 4n-2 & 4n-1 & 4n & 4n+4 & 4n+7 & 4n+9 & \cdots & 8n-3 & 8n-1\\
				1 & 3 & 5  & \cdots & 4n-7 & 4n-5 & 4n-3 & 4n-1 & 4n+1 & 4n+3 & 4n+5 & 4n+7 & 4n+9 & \cdots  & 8n-3 & 8n-1\\
				2 & 4 & 6 & \cdots & 4n-6 & 4n-4 & 4n-2 & 4n & 4n+2 & 4n+4 & 4n+6 & 4n+8 & 4n+10 & \cdots & 8n-2 & 8n\\
				2 & 4 & 6 & \cdots & 4n-6 & 4n-3 & 4n+1 & 4n+2 & 4n+3 & 4n+5 & 4n+6 & 4n+8 & 4n+10 & \cdots & 8n-2 & 8n
			\end{ytableau}\\
			$-$\hspace{.15cm} \begin{ytableau}
				1 & 3 & 5  & \cdots & 4n-7 & 4n-5 & 4n-4 & 4n-3 & 4n-1 & 4n & 4n+3 & 4n+7 & 4n+9 & \cdots & 8n-3 & 8n-1\\
				1 & 3 & 5  & \cdots & 4n-7 & 4n-5 & 4n-2 & 4n-1 & 4n+1 & 4n+4 & 4n+5 & 4n+7 & 4n+9 & \cdots  & 8n-3 & 8n-1\\
				2 & 4 & 6 & \cdots & 4n-6 & 4n-4 & 4n-2 & 4n & 4n+2 & 4n+4 & 4n+6 & 4n+8 & 4n+10 & \cdots & 8n-2 & 8n\\
				2 & 4 & 6 & \cdots & 4n-6 & 4n-3 & 4n+1 & 4n+2 & 4n+3 & 4n+5 & 4n+6 & 4n+8 & 4n+10 & \cdots & 8n-2 & 8n
			\end{ytableau}\\
			$+$\hspace{.15cm} \begin{ytableau}
				1 & 3 & 5  & \cdots & 4n-7 & 4n-5 & 4n-4 & 4n-3 & 4n-2 & 4n-1 & 4n+1 & 4n+7 & 4n+9 & \cdots & 8n-3 & 8n-1\\
				1 & 3 & 5  & \cdots & 4n-7 & 4n-5 & 4n-2 & 4n & 4n+2 & 4n+4 & 4n+5 & 4n+7 & 4n+9 & \cdots  & 8n-3 & 8n-1\\
				2 & 4 & 6 & \cdots & 4n-6 & 4n-4 & 4n-1 & 4n & 4n+3 & 4n+4 & 4n+6 & 4n+8 & 4n+10 & \cdots & 8n-2 & 8n\\
				2 & 4 & 6 & \cdots & 4n-6 & 4n-3 & 4n+1 & 4n+2 & 4n+3 & 4n+5 & 4n+6 & 4n+8 & 4n+10 & \cdots & 8n-2 & 8n
		\end{ytableau}}
		
		\normalsize{$=X_1X_6-Y_3+Y_1.$}
		
		Proof of (iii): Note that
		
		$ X_2X_3=$
		
		\tiny{\ytableausetup{boxsize=3.5em} \ytableausetup{centertableaux}
			\begin{ytableau}
				1 & 3 & 5  & \cdots & 4n-7 & 4n-5 & 4n-4 & 4n-1 & 4n+1 & 4n+3 & 4n+4 & 4n+7 & 4n+9 & \cdots & 8n-3 & 8n-1\\
				1 & 3 & 5  & \cdots & 4n-7 & 4n-5 & 4n-3 & 4n-2 & 4n+1 & 4n+2 & 4n+5 & 4n+7 & 4n+9 & \cdots  & 8n-3 & 8n-1\\
				2 & 4 & 6 & \cdots & 4n-6 & 4n-3 & 4n-2 & 4n & 4n+2 & 4n+5 & 4n+6 & 4n+8 & 4n+10 & \cdots & 8n-2 & 8n\\
				2 & 4 & 6 & \cdots & 4n-6 & 4n-4 & 4n-1 & 4n & 4n+3 & 4n+4 & 4n+6 & 4n+8 & 4n+10 & \cdots & 8n-2 & 8n 
		\end{ytableau} }
		
		\normalsize{By using Theorem \ref{thm2.6} and Remark \ref{rmk2.8}, we have the following straightening laws in $G/P^{\alpha_{4n}}$}
		
		\tiny{\ytableausetup{boxsize=3.5em} \ytableausetup{centertableaux}
			\begin{ytableau}
				1 & 3 & 5  & \cdots & 4n-7 & 4n-5 & 4n-4 & 4n-1 & 4n+1 & 4n+3 & 4n+4 & 4n+7 & 4n+9 & \cdots & 8n-3 & 8n-1\\
				1 & 3 & 5  & \cdots & 4n-7 & 4n-5 & 4n-3 & 4n-2 & 4n+1 & 4n+2 & 4n+5 & 4n+7 & 4n+9 & \cdots  & 8n-3 & 8n-1
			\end{ytableau}\\
			$=$\hspace{.15cm} \begin{ytableau}
				1 & 3 & 5  & \cdots & 4n-7 & 4n-5 & 4n-4 & 4n-2 & 4n+1 & 4n+2 & 4n+4 & 4n+7 & 4n+9 & \cdots & 8n-3 & 8n-1\\
				1 & 3 & 5  & \cdots & 4n-7 & 4n-5 & 4n-3 & 4n-1 & 4n+1 & 4n+3 & 4n+5 & 4n+7 & 4n+9 & \cdots  & 8n-3 & 8n-1
			\end{ytableau}\\
			$-$\hspace{.15cm} \begin{ytableau}
				1 & 3 & 5  & \cdots & 4n-7 & 4n-5 & 4n-4 & 4n-3 & 4n+1 & 4n+3 & 4n+4 & 4n+7 & 4n+9 & \cdots & 8n-3 & 8n-1\\
				1 & 3 & 5  & \cdots & 4n-7 & 4n-5 & 4n-2 & 4n-1 & 4n+1 & 4n+4 & 4n+5 & 4n+7 & 4n+9 & \cdots  & 8n-3 & 8n-1
			\end{ytableau}\\
			$+$\hspace{.15cm} \begin{ytableau}
				1 & 3 & 5  & \cdots & 4n-7 & 4n-5 & 4n-4 & 4n-3 & 4n-2 & 4n-1 & 4n+1 & 4n+7 & 4n+9 & \cdots & 8n-3 & 8n-1\\
				1 & 3 & 5  & \cdots & 4n-7 & 4n-5 & 4n+1 & 4n+2 & 4n+3 & 4n+4 & 4n+5 & 4n+7 & 4n+9 & \cdots  & 8n-3 & 8n-1
			\end{ytableau}.}
		
		\normalsize{and}
		
		\tiny{\ytableausetup{boxsize=3.5em} \ytableausetup{centertableaux}
			\begin{ytableau}
				2 & 4 & 6 & \cdots & 4n-6 & 4n-3 & 4n-2 & 4n & 4n+2 & 4n+5 & 4n+6 & 4n+8 & 4n+10 & \cdots & 8n-2 & 8n\\
				2 & 4 & 6 & \cdots & 4n-6 & 4n-4 & 4n-1 & 4n & 4n+3 & 4n+4 & 4n+6 & 4n+8 & 4n+10 & \cdots & 8n-2 & 8n 
			\end{ytableau} \\
			$=$ \hspace{.15cm}\begin{ytableau}
				2 & 4 & 6 & \cdots & 4n-6 & 4n-4 & 4n-2 & 4n & 4n+2 & 4n+4 & 4n+6 & 4n+8 & 4n+10 & \cdots & 8n-2 & 8n \\
				2 & 4 & 6 & \cdots & 4n-6 & 4n-3 & 4n-1 & 4n & 4n+3 & 4n+5 & 4n+6 & 4n+8 & 4n+10 & \cdots & 8n-2 & 8n
			\end{ytableau}\\
			$-$ \hspace{.15cm}\begin{ytableau}
				2 & 4 & 6 & \cdots & 4n-6 & 4n-4 & 4n-3 & 4n & 4n+2 & 4n+3 & 4n+6 & 4n+8 & 4n+10 & \cdots & 8n-2 & 8n\\
				2 & 4 & 6 & \cdots & 4n-6 & 4n-2 & 4n-1 & 4n & 4n+4 & 4n+5 & 4n+6 & 4n+8 & 4n+10 & \cdots & 8n-2 & 8n 
			\end{ytableau}\\
			$+$ \hspace{.15cm}\begin{ytableau}
				2 & 4 & 6 & \cdots & 4n-6 & 4n-4 & 4n-3 & 4n-2 & 4n-1 & 4n & 4n+6 & 4n+8 & 4n+10 & \cdots & 8n-2 & 8n\\
				2 & 4 & 6 & \cdots & 4n-6 & 4n & 4n+2 & 4n+3 & 4n+4 & 4n+5 & 4n+6 & 4n+8 & 4n+10 & \cdots & 8n-2 & 8n 
		\end{ytableau}}
		
		\normalsize{Since} we are working in $X(w_{6}),$ the above straightening law becomes
		
		\tiny{\ytableausetup{boxsize=3.5em} \ytableausetup{centertableaux}
			\begin{ytableau}
				2 & 4 & 6 & \cdots & 4n-6 & 4n-3 & 4n-2 & 4n & 4n+2 & 4n+5 & 4n+6 & 4n+8 & 4n+10 & \cdots & 8n-2 & 8n\\
				2 & 4 & 6 & \cdots & 4n-6 & 4n-4 & 4n-1 & 4n & 4n+3 & 4n+4 & 4n+6 & 4n+8 & 4n+10 & \cdots & 8n-2 & 8n 
			\end{ytableau} \\
			$=$ \hspace{.15cm}\begin{ytableau}
				2 & 4 & 6 & \cdots & 4n-6 & 4n-4 & 4n-2 & 4n & 4n+2 & 4n+4 & 4n+6 & 4n+8 & 4n+10 & \cdots & 8n-2 & 8n \\
				2 & 4 & 6 & \cdots & 4n-6 & 4n-3 & 4n-1 & 4n & 4n+3 & 4n+5 & 4n+6 & 4n+8 & 4n+10 & \cdots & 8n-2 & 8n
		\end{ytableau}}
		
		\normalsize{Hence,} $X_2X_3=$
		
		\tiny{\ytableausetup{boxsize=3.5em} \ytableausetup{centertableaux}
			\begin{ytableau}
				1 & 3 & 5  & \cdots & 4n-7 & 4n-5 & 4n-4 & 4n-2 & 4n+1 & 4n+2 & 4n+4 & 4n+7 & 4n+9 & \cdots & 8n-3 & 8n-1\\
				1 & 3 & 5  & \cdots & 4n-7 & 4n-5 & 4n-3 & 4n-1 & 4n+1 & 4n+3 & 4n+5 & 4n+7 & 4n+9 & \cdots  & 8n-3 & 8n-1\\
				2 & 4 & 6 & \cdots & 4n-6 & 4n-4 & 4n-2 & 4n & 4n+2 & 4n+4 & 4n+6 & 4n+8 & 4n+10 & \cdots & 8n-2 & 8n \\
				2 & 4 & 6 & \cdots & 4n-6 & 4n-3 & 4n-1 & 4n & 4n+3 & 4n+5 & 4n+6 & 4n+8 & 4n+10 & \cdots & 8n-2 & 8n
			\end{ytableau}\\
			$-$\hspace{.15cm} \begin{ytableau}
				1 & 3 & 5  & \cdots & 4n-7 & 4n-5 & 4n-4 & 4n-3 & 4n+1 & 4n+3 & 4n+4 & 4n+7 & 4n+9 & \cdots & 8n-3 & 8n-1\\
				1 & 3 & 5  & \cdots & 4n-7 & 4n-5 & 4n-2 & 4n-1 & 4n+1 & 4n+4 & 4n+5 & 4n+7 & 4n+9 & \cdots  & 8n-3 & 8n-1\\
				2 & 4 & 6 & \cdots & 4n-6 & 4n-4 & 4n-2 & 4n & 4n+2 & 4n+4 & 4n+6 & 4n+8 & 4n+10 & \cdots & 8n-2 & 8n \\
				2 & 4 & 6 & \cdots & 4n-6 & 4n-3 & 4n-1 & 4n & 4n+3 & 4n+5 & 4n+6 & 4n+8 & 4n+10 & \cdots & 8n-2 & 8n
			\end{ytableau}\\
			$+$\hspace{.15cm} \begin{ytableau}
				1 & 3 & 5  & \cdots & 4n-7 & 4n-5 & 4n-4 & 4n-3 & 4n-2 & 4n-1 & 4n+1 & 4n+7 & 4n+9 & \cdots & 8n-3 & 8n-1\\
				1 & 3 & 5  & \cdots & 4n-7 & 4n-5 & 4n+1 & 4n+2 & 4n+3 & 4n+4 & 4n+5 & 4n+7 & 4n+9 & \cdots  & 8n-3 & 8n-1\\
				2 & 4 & 6 & \cdots & 4n-6 & 4n-4 & 4n-2 & 4n & 4n+2 & 4n+4 & 4n+6 & 4n+8 & 4n+10 & \cdots & 8n-2 & 8n \\
				2 & 4 & 6 & \cdots & 4n-6 & 4n-3 & 4n-1 & 4n & 4n+3 & 4n+5 & 4n+6 & 4n+8 & 4n+10 & \cdots & 8n-2 & 8n
			\end{ytableau}.}
		
		\normalsize{Since} we are working in $X(w_{6}),$ \normalsize{by using Theorem \ref{thm2.6} and Remark \ref{rmk2.8}, we have }
		
		\tiny{\ytableausetup{boxsize=3.5em} \ytableausetup{centertableaux}
			\begin{ytableau}
				1 & 3 & 5  & \cdots & 4n-7 & 4n-5 & 4n+1 & 4n+2 & 4n+3 & 4n+4 & 4n+5 & 4n+7 & 4n+9 & \cdots  & 8n-3 & 8n-1\\
				2 & 4 & 6 & \cdots & 4n-6 & 4n-3 & 4n-1 & 4n & 4n+3 & 4n+5 & 4n+6 & 4n+8 & 4n+10 & \cdots & 8n-2 & 8n
			\end{ytableau}\\
			$=$  \hspace{.1cm} \begin{ytableau}
				1 & 3 & 5  & \cdots & 4n-7 & 4n-5 & 4n-1 & 4n & 4n+3 & 4n+4 & 4n+5 & 4n+7 & 4n+9 & \cdots  & 8n-3 & 8n-1\\
				2 & 4 & 6 & \cdots & 4n-6 & 4n-3 & 4n+1 & 4n+2 & 4n+3 & 4n+5 & 4n+6 & 4n+8 & 4n+10 & \cdots & 8n-2 & 8n 
		\end{ytableau}}
		
		\normalsize{Hence,} we have
		
		\tiny{\ytableausetup{boxsize=3.5em} \ytableausetup{centertableaux}
			\begin{ytableau}
				1 & 3 & 5  & \cdots & 4n-7 & 4n-5 & 4n-4 & 4n-3 & 4n-2 & 4n-1 & 4n+1 & 4n+7 & 4n+9 & \cdots & 8n-3 & 8n-1\\
				1 & 3 & 5  & \cdots & 4n-7 & 4n-5 & 4n+1 & 4n+2 & 4n+3 & 4n+4 & 4n+5 & 4n+7 & 4n+9 & \cdots  & 8n-3 & 8n-1\\
				2 & 4 & 6 & \cdots & 4n-6 & 4n-4 & 4n-2 & 4n & 4n+2 & 4n+4 & 4n+6 & 4n+8 & 4n+10 & \cdots & 8n-2 & 8n \\
				2 & 4 & 6 & \cdots & 4n-6 & 4n-3 & 4n-1 & 4n & 4n+3 & 4n+5 & 4n+6 & 4n+8 & 4n+10 & \cdots & 8n-2 & 8n
			\end{ytableau}\\
			$=$ \hspace{.15cm}\begin{ytableau}
				1 & 3 & 5  & \cdots & 4n-7 & 4n-5 & 4n-4 & 4n-3 & 4n-2 & 4n-1 & 4n+1 & 4n+7 & 4n+9 & \cdots & 8n-3 & 8n-1\\
				1 & 3 & 5  & \cdots & 4n-7 & 4n-5 & 4n-1 & 4n & 4n+3 & 4n+4 & 4n+5 & 4n+7 & 4n+9 & \cdots  & 8n-3 & 8n-1\\
				2 & 4 & 6 & \cdots & 4n-6 & 4n-4 & 4n-2 & 4n & 4n+2 & 4n+4 & 4n+6 & 4n+8 & 4n+10 & \cdots & 8n-2 & 8n \\
				2 & 4 & 6 & \cdots & 4n-6 & 4n-3 & 4n+1 & 4n+2 & 4n+3 & 4n+5 & 4n+6 & 4n+8 & 4n+10 & \cdots & 8n-2 & 8n 
		\end{ytableau}}
		
		\normalsize{Since} we are working in $X(w_{6}),$ \normalsize{by using Theorem \ref{thm2.6} and Remark \ref{rmk2.8}, we have } 
		
		\tiny{\ytableausetup{boxsize=3.5em} \ytableausetup{centertableaux}
			\begin{ytableau}
				1 & 3 & 5  & \cdots & 4n-7 & 4n-5 & 4n-1 & 4n & 4n+3 & 4n+4 & 4n+5 & 4n+7 & 4n+9 & \cdots  & 8n-3 & 8n-1\\
				2 & 4 & 6 & \cdots & 4n-6 & 4n-4 & 4n-2 & 4n & 4n+2 & 4n+4 & 4n+6 & 4n+8 & 4n+10 & \cdots & 8n-2 & 8n
			\end{ytableau}\\
			$=$  \hspace{.1cm} \begin{ytableau}
				1 & 3 & 5  & \cdots & 4n-7 & 4n-5 & 4n-2 & 4n & 4n+2 & 4n+4 & 4n+5 & 4n+7 & 4n+9 & \cdots  & 8n-3 & 8n-1\\
				2 & 4 & 6 & \cdots & 4n-6 & 4n-4 & 4n-1 & 4n & 4n+3 & 4n+4 & 4n+6 & 4n+8 & 4n+10 & \cdots & 8n-2 & 8n
		\end{ytableau}}
		
		\normalsize{Hence,} we have
		
		\tiny{\ytableausetup{boxsize=3.5em} \ytableausetup{centertableaux}
			\begin{ytableau}
				1 & 3 & 5  & \cdots & 4n-7 & 4n-5 & 4n-4 & 4n-3 & 4n-2 & 4n-1 & 4n+1 & 4n+7 & 4n+9 & \cdots & 8n-3 & 8n-1\\
				1 & 3 & 5  & \cdots & 4n-7 & 4n-5 & 4n+1 & 4n+2 & 4n+3 & 4n+4 & 4n+5 & 4n+7 & 4n+9 & \cdots  & 8n-3 & 8n-1\\
				2 & 4 & 6 & \cdots & 4n-6 & 4n-4 & 4n-2 & 4n & 4n+2 & 4n+4 & 4n+6 & 4n+8 & 4n+10 & \cdots & 8n-2 & 8n \\
				2 & 4 & 6 & \cdots & 4n-6 & 4n-3 & 4n-1 & 4n & 4n+3 & 4n+5 & 4n+6 & 4n+8 & 4n+10 & \cdots & 8n-2 & 8n
			\end{ytableau}\\
			$=$ \hspace{.1cm} \begin{ytableau}
				1 & 3 & 5  & \cdots & 4n-7 & 4n-5 & 4n-4 & 4n-3 & 4n-2 & 4n-1 & 4n+1 & 4n+7 & 4n+9 & \cdots & 8n-3 & 8n-1\\
				1 & 3 & 5  & \cdots & 4n-7 & 4n-5 & 4n-2 & 4n & 4n+2 & 4n+4 & 4n+5 & 4n+7 & 4n+9 & \cdots  & 8n-3 & 8n-1\\
				2 & 4 & 6 & \cdots & 4n-6 & 4n-4 & 4n-1 & 4n & 4n+3 & 4n+4 & 4n+6 & 4n+8 & 4n+10 & \cdots & 8n-2 & 8n\\
				2 & 4 & 6 & \cdots & 4n-6 & 4n-3 & 4n+1 & 4n+2 & 4n+3 & 4n+5 & 4n+6 & 4n+8 & 4n+10 & \cdots & 8n-2 & 8n 
		\end{ytableau}}
		
		\normalsize{Hence,} $X_2X_3=$
		
		\tiny{\ytableausetup{boxsize=3.3em} \ytableausetup{centertableaux}
			\begin{ytableau}
				1 & 3 & 5  & \cdots & 4n-7 & 4n-5 & 4n-4 & 4n-2 & 4n+1 & 4n+2 & 4n+4 & 4n+7 & 4n+9 & \cdots & 8n-3 & 8n-1\\
				1 & 3 & 5  & \cdots & 4n-7 & 4n-5 & 4n-3 & 4n-1 & 4n+1 & 4n+3 & 4n+5 & 4n+7 & 4n+9 & \cdots  & 8n-3 & 8n-1\\
				2 & 4 & 6 & \cdots & 4n-6 & 4n-4 & 4n-2 & 4n & 4n+2 & 4n+4 & 4n+6 & 4n+8 & 4n+10 & \cdots & 8n-2 & 8n \\
				2 & 4 & 6 & \cdots & 4n-6 & 4n-3 & 4n-1 & 4n & 4n+3 & 4n+5 & 4n+6 & 4n+8 & 4n+10 & \cdots & 8n-2 & 8n
			\end{ytableau}\\
			$-$\hspace{.10cm} \begin{ytableau}
				1 & 3 & 5  & \cdots & 4n-7 & 4n-5 & 4n-4 & 4n-3 & 4n+1 & 4n+3 & 4n+4 & 4n+7 & 4n+9 & \cdots & 8n-3 & 8n-1\\
				1 & 3 & 5  & \cdots & 4n-7 & 4n-5 & 4n-2 & 4n-1 & 4n+1 & 4n+4 & 4n+5 & 4n+7 & 4n+9 & \cdots  & 8n-3 & 8n-1\\
				2 & 4 & 6 & \cdots & 4n-6 & 4n-4 & 4n-2 & 4n & 4n+2 & 4n+4 & 4n+6 & 4n+8 & 4n+10 & \cdots & 8n-2 & 8n \\
				2 & 4 & 6 & \cdots & 4n-6 & 4n-3 & 4n-1 & 4n & 4n+3 & 4n+5 & 4n+6 & 4n+8 & 4n+10 & \cdots & 8n-2 & 8n
			\end{ytableau}\\
			$+$\hspace{.10cm} \begin{ytableau}
				1 & 3 & 5  & \cdots & 4n-7 & 4n-5 & 4n-4 & 4n-3 & 4n-2 & 4n-1 & 4n+1 & 4n+7 & 4n+9 & \cdots & 8n-3 & 8n-1\\
				1 & 3 & 5  & \cdots & 4n-7 & 4n-5 & 4n-2 & 4n & 4n+2 & 4n+4 & 4n+5 & 4n+7 & 4n+9 & \cdots  & 8n-3 & 8n-1\\
				2 & 4 & 6 & \cdots & 4n-6 & 4n-4 & 4n-1 & 4n & 4n+3 & 4n+4 & 4n+6 & 4n+8 & 4n+10 & \cdots & 8n-2 & 8n\\
				2 & 4 & 6 & \cdots & 4n-6 & 4n-3 & 4n+1 & 4n+2 & 4n+3 & 4n+5 & 4n+6 & 4n+8 & 4n+10 & \cdots & 8n-2 & 8n 
		\end{ytableau}}

		\normalsize{$=X_1X_4-Y_4+Y_1.$}
	\end{proof}
	
	\begin{lemma}\label{lem5.3}
		The following relations in $X_i$'s $(1 \leq i \leq 6),$ $Y_j$'s ($1 \leq j \leq 4$) hold in $R_2.$
		\begin{itemize}
			\item[(i)] $X_2Y_1-Z_1=0.$
			\item[(ii)] $X_2Y_2-Z_2=0.$
		\end{itemize}
	\end{lemma}
	\begin{proof}
		Proof of (i): Note that
		$X_2Y_1=$
		
		\tiny{\ytableausetup{boxsize=3.5em} \ytableausetup{centertableaux}
			\begin{ytableau}
				1& 3 & 5  & \cdots & 4n-7 & 4n-5 & 4n-4 & 4n-3 & 4n-2 & 4n-1 & 4n+1 & 4n+7 & 4n+9 & \cdots & 8n-3 & 8n-1\\
				1 & 3 & 5  & \cdots & 4n-7 & 4n-5 & 4n-4 & 4n-1 & 4n+1 & 4n+3 & 4n+4 & 4n+7 & 4n+9 & \cdots & 8n-3 & 8n-1\\
				1& 3 & 5 & \cdots & 4n-7 & 4n-5 & 4n-2 & 4n & 4n+2 & 4n+4 & 4n+5 & 4n+7 & 4n+9 & \cdots & 8n-3 & 8n-1\\
				2 & 4 & 6 & \cdots & 4n-6 & 4n-3 & 4n-2 & 4n & 4n+2 & 4n+5 & 4n+6 & 4n+8 & 4n+10 & \cdots & 8n-2 & 8n\\
				2 & 4 & 6 & \cdots & 4n-6 & 4n-4 & 4n-1 & 4n & 4n+3 & 4n+4 & 4n+6 & 4n+8 & 4n+10 & \cdots & 8n-2 & 8n\\
				2 & 4 & 6 & \cdots & 4n-6 & 4n-3 & 4n+1 & 4n+2 & 4n+3 & 4n+5 & 4n+6 & 4n+8 & 4n+10 & \cdots & 8n-2 & 8n\\
		\end{ytableau}}
		
		\normalsize{Since} we are working in $X(w_{6}),$ \normalsize{by using Theorem \ref{thm2.6} and Remark \ref{rmk2.8}, we have }
		
		\tiny{\ytableausetup{boxsize=3.5em} \ytableausetup{centertableaux}
			\begin{ytableau}
				2 & 4 & 6 & \cdots & 4n-6 & 4n-3 & 4n-2 & 4n & 4n+2 & 4n+5 & 4n+6 & 4n+8 & 4n+10 & \cdots & 8n-2 & 8n\\
				2 & 4 & 6 & \cdots & 4n-6 & 4n-4 & 4n-1 & 4n & 4n+3 & 4n+4 & 4n+6 & 4n+8 & 4n+10 & \cdots & 8n-2 & 8n
			\end{ytableau}\\
			$=$  \hspace{.1cm} \begin{ytableau}
				2 & 4 & 6 & \cdots & 4n-6 & 4n-4 & 4n-2 & 4n & 4n+2 & 4n+5 & 4n+6 & 4n+8 & 4n+10 & \cdots & 8n-2 & 8n\\
				2 & 4 & 6 & \cdots & 4n-6 & 4n-3 & 4n-1 & 4n & 4n+3 & 4n+4 & 4n+6 & 4n+8 & 4n+10 & \cdots & 8n-2 & 8n
		\end{ytableau}}
		
		\normalsize{Hence,} we have
		
		$X_2Y_1=$
		
		\tiny{\ytableausetup{boxsize=3.5em} \ytableausetup{centertableaux}
			\begin{ytableau}
				1& 3 & 5  & \cdots & 4n-7 & 4n-5 & 4n-4 & 4n-3 & 4n-2 & 4n-1 & 4n+1 & 4n+7 & 4n+9 & \cdots & 8n-3 & 8n-1\\
				1 & 3 & 5  & \cdots & 4n-7 & 4n-5 & 4n-4 & 4n-1 & 4n+1 & 4n+3 & 4n+4 & 4n+7 & 4n+9 & \cdots & 8n-3 & 8n-1\\
				1& 3 & 5 & \cdots & 4n-7 & 4n-5 & 4n-2 & 4n & 4n+2 & 4n+4 & 4n+5 & 4n+7 & 4n+9 & \cdots & 8n-3 & 8n-1\\
				2 & 4 & 6 & \cdots & 4n-6 & 4n-4 & 4n-2 & 4n & 4n+2 & 4n+5 & 4n+6 & 4n+8 & 4n+10 & \cdots & 8n-2 & 8n\\
				2 & 4 & 6 & \cdots & 4n-6 & 4n-3 & 4n-1 & 4n & 4n+3 & 4n+4 & 4n+6 & 4n+8 & 4n+10 & \cdots & 8n-2 & 8n\\
				2 & 4 & 6 & \cdots & 4n-6 & 4n-3 & 4n+1 & 4n+2 & 4n+3 & 4n+5 & 4n+6 & 4n+8 & 4n+10 & \cdots & 8n-2 & 8n
		\end{ytableau}}
		
		\normalsize{$=Z_1$.}
		
		Proof of (ii): Note that
		
		$X_2Y_2=$
		
		\tiny{\ytableausetup{boxsize=3.5em} \ytableausetup{centertableaux}
			\begin{ytableau}
				1& 3 & 5  & \cdots & 4n-7 & 4n-5 & 4n-4 & 4n-3 & 4n-2 & 4n & 4n+2 & 4n+7 & 4n+9 & \cdots & 8n-3 & 8n-1\\
				1 & 3 & 5  & \cdots & 4n-7 & 4n-5 & 4n-4 & 4n-1 & 4n+1 & 4n+3 & 4n+4 & 4n+7 & 4n+9 & \cdots & 8n-3 & 8n-1\\
				1& 3 & 5 & \cdots & 4n-7 & 4n-5 & 4n-2 & 4n-1 & 4n+1 & 4n+4 & 4n+5 & 4n+7 & 4n+9 & \cdots & 8n-3 & 8n-1\\
				2 & 4 & 6 & \cdots & 4n-6 & 4n-4 & 4n-1 & 4n & 4n+3 & 4n+4 & 4n+6 & 4n+8 & 4n+10 & \cdots & 8n-2 & 8n\\
				2 & 4 & 6 & \cdots & 4n-6 & 4n-3 & 4n-2 & 4n & 4n+2 & 4n+5 & 4n+6 & 4n+8 & 4n+10 & \cdots & 8n-2 & 8n\\
				2 & 4 & 6 & \cdots & 4n-6 & 4n-3 & 4n+1 & 4n+2 & 4n+3 & 4n+5 & 4n+6 & 4n+8 & 4n+10 & \cdots & 8n-2 & 8n\\
			\end{ytableau}.}
		
		\normalsize{Since} we are working in $X(w_{6}),$ \normalsize{by using Theorem \ref{thm2.6} and Remark \ref{rmk2.8}, we have }
		
		\tiny{\ytableausetup{boxsize=3.5em} \ytableausetup{centertableaux}
			\begin{ytableau}
				2 & 4 & 6 & \cdots & 4n-6 & 4n-3 & 4n-2 & 4n & 4n+2 & 4n+5 & 4n+6 & 4n+8 & 4n+10 & \cdots & 8n-2 & 8n\\
				2 & 4 & 6 & \cdots & 4n-6 & 4n-4 & 4n-1 & 4n & 4n+3 & 4n+4 & 4n+6 & 4n+8 & 4n+10 & \cdots & 8n-2 & 8n
			\end{ytableau}\\
			$=$  \hspace{.1cm} \begin{ytableau}
				2 & 4 & 6 & \cdots & 4n-6 & 4n-4 & 4n-2 & 4n & 4n+2 & 4n+5 & 4n+6 & 4n+8 & 4n+10 & \cdots & 8n-2 & 8n\\
				2 & 4 & 6 & \cdots & 4n-6 & 4n-3 & 4n-1 & 4n & 4n+3 & 4n+4 & 4n+6 & 4n+8 & 4n+10 & \cdots & 8n-2 & 8n
		\end{ytableau}}
		
		\normalsize{$X_2Y_2=$}
		
		\tiny{\ytableausetup{boxsize=3.5em}  \ytableausetup{centertableaux}
			\begin{ytableau}
				1& 3 & 5  & \cdots & 4n-7 & 4n-5 & 4n-4 & 4n-3 & 4n-2 & 4n & 4n+2 & 4n+7 & 4n+9 & \cdots & 8n-3 & 8n-1\\
				1 & 3 & 5  & \cdots & 4n-7 & 4n-5 & 4n-4 & 4n-1 & 4n+1 & 4n+3 & 4n+4 & 4n+7 & 4n+9 & \cdots & 8n-3 & 8n-1\\
				1& 3 & 5 & \cdots & 4n-7 & 4n-5 & 4n-2 & 4n-1 & 4n+1 & 4n+4 & 4n+5 & 4n+7 & 4n+9 & \cdots & 8n-3 & 8n-1\\
				2 & 4 & 6 & \cdots & 4n-6 & 4n-4 & 4n-2 & 4n & 4n+2 & 4n+5 & 4n+6 & 4n+8 & 4n+10 & \cdots & 8n-2 & 8n\\
				2 & 4 & 6 & \cdots & 4n-6 & 4n-3 & 4n-1 & 4n & 4n+3 & 4n+4 & 4n+6 & 4n+8 & 4n+10 & \cdots & 8n-2 & 8n\\
				2 & 4 & 6 & \cdots & 4n-6 & 4n-3 & 4n+1 & 4n+2 & 4n+3 & 4n+5 & 4n+6 & 4n+8 & 4n+10 & \cdots & 8n-2 & 8n\\
			\end{ytableau}.}
		
		\normalsize{$=Z_2.$}
		
	\end{proof}
	
	\begin{lemma}\label{lemma5.4}
		$R$ is generated by $X_i$ $(1 \leq i \leq 6)$ and $Y_1$ as a $\mathbb{C}$-algebra.
	\end{lemma}
	\begin{proof}
		The proof follows from \cref{lem5.1}, \cref{lem5.2} and \cref{lem5.3}.
	\end{proof}
	\begin{corollary}
		Let $w\in W^{P^{\alpha_{4n}}}$ be such that $w_{6}\le w.$ Then corresponding graded ring $R$ is not generated by $R_1.$
	\end{corollary}
	\begin{proof}
		Since $T_G$ is linearly reductive, the restriction map $$\phi: H^{0}(X_G(w), \mathcal{L}^{\otimes k}(2m\omega_{4n}))^{T_G} \longrightarrow H^{0}(X_G(w_6), \mathcal{L}^{\otimes k}(2m\omega_{4n}))^{T_G}$$ is surjective for all $k \geq 1.$ So, by Lemma \ref{lemma5.4}, $R$ is not generated by degree one elements.
	\end{proof}

	Now, we recall that by \cref{thm2.1}(ii), the very ample line bundle $\mathcal{L}(4\omega_{4n})$ descends to a line bundle on the GIT quotient $T_G \backslash \backslash (G/P^{\alpha_{4n}})^{ss}_{T_G}(\mathcal{L}(4\omega_{4n})).$ We prove that the GIT quotient $T_G \backslash \backslash (X_G(w_6))^{ss}_{T_G}(\mathcal{L}(4\omega_{4n}))$ is projectively normal with respect to the descent of the $T_G$-linearized very ample line bundle $\mathcal{L}(4\omega_{4n}).$
	\begin{theorem}\label{thm5.1}
		The homogeneous coordinate ring of $T_G \backslash \backslash (X_G(w_6))^{ss}_{T_G}(\mathcal{L}(4\omega_{4n}))$ is generated by elements of degree one. 
	\end{theorem}
	\begin{proof}
		Let $f \in H^{0}(X_G(w_6), \mathcal{L}^{\otimes k}(4\omega_{4n}))^{T_G}=H^{0}(X_G(w_6), \mathcal{L}^{\otimes 2k}(2\omega_{4n}))^{T_G}.$ Then by \cref{lemma5.4}, we have $$f = \sum {a_{\underline{m},n_1}}X^{\underline{m}}Y_1^{n_1},$$ where $\underline{m}=(m_1,m_2,m_3,m_4,m_5,m_6) \in \mathbb{Z}_{\geq 0}^6$ and $n_1 \in \mathbb{Z}_{\ge 0}$ such that $m_1+m_2+m_3+m_4+m_5+m_6+2n_1=2k,$ $X^{\underline{m}}$ denotes $X_1^{m_1}X_2^{m_2}X_3^{m_3}X_4^{m_4}X_5^{m_5}X_6^{m_6}$ and $a_{\underline{m}, n_1}$'s are non-zero scalars.
		
		Now to prove that the homogeneous coordinate ring of $T_G \backslash \backslash (X_G(w_6))^{ss}_{T_G}(\mathcal{L}(4\omega_{4n}))$ is generated by $H^0(X_G(w_6), \mathcal{L}(4\omega_{4n}))^{T_G}$ as a $\mathbb{C}$-algebra, it is enough to show that for each $f$ as above and each monomial appearing in the expression of $f$ is a product of $k$ elements of $H^0(X_G(w_6), \mathcal{L}(4\omega_{4n}))^{T_G}.$
		
		Consider the monomial $X^{\underline{m}}Y_{1}^{n_1}$ in the expression of $f.$ Note that $m_1+m_2+m_3+m_4+m_5+m_6$ is an even integer. Thus, $X^{\underline{m}}$ can be written as $\prod_{(i,j)}X_iX_j,$ where the number of pairs $(i,j)$ is $k-n_1$ and repetitions of $X_i$'s are allowed. Therefore, $X^{\underline{m}}$ can be written as product of $k-n_1$ number of monomials in $H^0(X_G(w_6), \mathcal{L}(4\omega_{4n}))^{T_G}.$ On the other hand, since $Y_1\in H^0(X_G(w_6), \mathcal{L}(4\omega_{4n}))^{T_G},$ $X^{\underline{m}}Y_{1}^{n_1}$ is a product of $k$ elements of $H^0(X_G(w_6), \mathcal{L}(4\omega_{4n}))^{T_G}.$
	\end{proof} 
	\begin{corollary}
		The GIT quotient $T_G \backslash \backslash (X_G(w_6))^{ss}_{T_G}(\mathcal{L}(4\omega_{4n}))$ is projectively normal with respect to the descent of the $T_G$-linearized very ample line bundle $\mathcal{L}(4\omega_{4n}).$
	\end{corollary}
	\begin{proof}
		Since $X_G(w_6)$ is normal and  $X_G(w_6)^{ss}_{T_G}(\mathcal{L}(4\omega_{4n}))$ is  open in $X_{G}(w_{6}),$ $X_G(w_6)^{ss}_{T_G}(\mathcal{L}(4\omega_{4n}))$ is also normal. Hence, ${T_G} \backslash \backslash (X_G(w_6))^{ss}_{T_G}(\mathcal{L}(4\omega_{4n}))$ is a normal variety.  Therefore, by using \cref{thm5.1}, it follows that $T_G \backslash \backslash (X_G(w_6))^{ss}_{T_G}(\mathcal{L}(4\omega_{4n}))$ is projectively normal with respect to the descent of the $T_G$-linearized very ample line bundle $\mathcal{L}(4\omega_{4n}).$
	\end{proof}
	\begin{remark}\label{rmk5.3}
		The GIT quotient $T_G \backslash \backslash (X_G(w_6))^{ss}_{T_G}(\mathcal{L}(4m\omega_{4n}))$ $(m \geq 1)$ is projectively normal with respect to the descent of the $T_G$-linearized very ample line bundle $\mathcal{L}(4m\omega_{4n}).$
	\end{remark}

	\begin{corollary}\label{cor5.4}
		The GIT quotient $T_G \backslash \backslash (X_G(v))^{ss}_{T_G}(\mathcal{L}(4m\omega_{4n}))$ is  projectively normal for all $v \in W^{P^{\alpha_{4n}}}$ such that $w_1 \leq v\le w_{6}$ with respect to the descent of the $T_G$-linearized very ample line bundle $\mathcal{L}(4m\omega_{4n}).$ 
	\end{corollary}
	
	\begin{proof}
		Since $T_G$ is linearly reductive, the restriction map $$\phi: H^{0}(X_G(w_{6}), \mathcal{L}^{\otimes k}(4m\omega_{4n}))^{T_G} \longrightarrow H^{0}(X_G(v), \mathcal{L}^{\otimes k}(4m\omega_{4n}))^{T_G}$$ is surjective for all $k \geq 1$.
		
		So, by Remark \ref{rmk5.3}, $T_G \backslash \backslash (X_G(v))^{ss}_{T_G}(\mathcal{L}(4m\omega_{4n}))$ is projectively normal with respect to the descent of the $T_G$-linearized very ample line bundle $\mathcal{L}(4m\omega_{4n}).$
	\end{proof}

	\begin{proposition} We have
		\begin{itemize}
			\item[(i)] The GIT quotient $T_G\backslash\backslash(X_G(w_1))^{ss}_{T_G}(\mathcal{L}(4\omega_{4n}))$ is a point.
			\item[(ii)] The GIT quotient $T_G\backslash\backslash(X_G(w_2))^{ss}_{T_G}(\mathcal{L}(4\omega_{4n}))$ is isomorphic to $(\mathbb{P}^{1},\mathcal{O}_{\mathbb{P}^{1}}(2))$ as a polarized variety.  
			\item[(iii)] The GIT quotient $T_G\backslash\backslash(X_G(w_3))^{ss}_{T_G}(\mathcal{L}(4\omega_{4n}))$ is isomorphic to $(\mathbb{P}^{1},\mathcal{O}_{\mathbb{P}^{1}}(2))$ as a polarized variety.
			\item[(iv)] The GIT quotient $T_G\backslash\backslash(X_G(w_4))^{ss}_{T_G}(\mathcal{L}(4\omega_{4n}))$ is isomorphic to $(\mathbb{P}^{3},\mathcal{O}_{\mathbb{P}^{3}}(2))$ as a polarized variety.
			\item[(v)] The GIT quotient $T_G\backslash\backslash(X_G(w_5))^{ss}_{T_G}(\mathcal{L}(4\omega_{4n}))$ is isomorphic to $(\mathbb{P}^{2},\mathcal{O}_{\mathbb{P}^{2}}(2))$ as a polarized variety.
		\end{itemize}	
	\end{proposition}
	\begin{proof} Note that $Y_{1}=0$ on $X_G(w_{i})$ for all $1 \leq i \leq 5.$

		Proof of (i):  Since $T_G$ is linearly reductive, the restriction map $$\phi: H^{0}(X_G(w_6), \mathcal{L}^{\otimes k}(4\omega_{4n}))^{T_G} \\ \longrightarrow H^{0}(X_G(w_{1}), \mathcal{L}^{\otimes k}(4\omega_{4n}))^{T_G}$$ is surjective for all $k \geq 1$. Thus, by \cref{lemma5.4}, any standard monomial in $H^{0}(X_G(w_1),\mathcal{L}^{\otimes k}(4\omega_{4n}))^{T_G}$ is of the form $X_1^{2k}.$ Hence, $T_G\backslash\backslash(X_G(w_1))^{ss}_{T_G}(\mathcal{L}(4\omega_{4n}))=Proj(\mathbb{C}[X_1^2]).$ 
		
		Proof of (ii):  Since $T_G$ is linearly reductive, the restriction map $$\phi: H^{0}(X_G(w_6), \mathcal{L}^{\otimes k}(4\omega_{4n}))^{T_G} \longrightarrow H^{0}(X_G(w_{2}), \mathcal{L}^{\otimes k}(4\omega_{4n}))^{T_G}$$ is surjective for all $k \geq 1$. Thus, by \cref{lemma5.4}, $X_1^2, X_1X_2, X_2^2$ are standard monomials in $H^{0}(X_G(w_2),\mathcal{L}(4\omega_{4n}))^{T_G}.$ 
		
		Note that by \cref{cor5.4}, $T_G \backslash \backslash (X_G(w_2))^{ss}_{T_G}(\mathcal{L}(4\omega_{4n}))$ is projectively normal. Therefore, there is a surjective homomorphism of $\mathbb{C}$-algebras $$\phi: \mathbb{C}[z_0,z_1,z_2] \longrightarrow \oplus_{k \in \mathbb{Z}_{\geq 0}}H^{0}(X_G(w_2), \mathcal{L}^{\otimes k}(4\omega_{4n}))^{T_G},$$ defined by
		$$z_0 \mapsto X_1^2,\\ z_1 \mapsto X_1X_2 \text{ and }\\ z_2 \mapsto X_2^2.$$ 
		
		Let $I$ be the ideal of $\mathbb{C}[z_0,z_1,z_2]$ generated by the relation
		\begin{equation}\label{5.1}
			z_0z_2=z_1^2.
		\end{equation} 
		
		Clearly, $I \subseteq \ker\phi$ and $\phi$ induces a homomorphism of $\mathbb{C}$-algebras $$\tilde{\phi}: \mathbb{C}[z_0,z_1,z_2]/I \longrightarrow \oplus_{k \in \mathbb{Z}_{\geq 0}}H^{0}(X_G(w_2), \mathcal{L}^{\otimes k}(4\omega_{4n}))^{T_G}.$$Now, we show that $\tilde{\phi}$ is an isomorphism. To complete the proof we use  \eqref{5.1} as a reduction system. The process is of replacing a monomial $M$ in $z_i$'s which is divisible by a term $L_i$'s on the left hand side of the reduction rule $L_i=R_i$ by $(M/L_i)R_i,$ where $R_i$ is the right hand side of the reduction rule. We show that the diamond lemma of ring theory holds for this reduction system (see \cite{bergman}). That is any monomial in $z_i$'s reduces after applying the reduction rule to a unique expression in $z_i$'s, in which no term is divisible by a term appearing on the left hand side of the above reduction system.
		
		Since we have only one reduction rule, it is enough to check for $z_0z_1z_2.$ Note that $z_0z_1z_2=(z_0z_2)z_1=z_1^3,$ for which no further reduction is possible. Therefore, $\tilde{\phi}$ is an isomorphism. Hence, the GIT quotient  $T_G\backslash\backslash(X_G(w_2))^{ss}_{T_G}(\mathcal{L}(4\omega_{4n}))$ is isomorphic to $Proj(\frac{\mathbb{C}[z_0, z_1, z_2]}{(z_1^2-z_0z_2)}).$ Thus, the GIT quotient $T_G\backslash\backslash(X_G(w_2))^{ss}_{T_G}(\mathcal{L}(4\omega_{4n}))$ is isomorphic to  $(\mathbb{P}^1, \mathcal{O}_{\mathbb{P}^1}(2))$ as a polarized variety.
		
		Proof of (iii) is similar to the proof of (ii).
		
		Proof of (iv):  Since $T_G$ is linearly reductive, the restriction map $$\phi: H^{0}(X_G(w_6), \mathcal{L}^{\otimes k}(4\omega_{4n}))^{T_G} \\ \longrightarrow H^{0}(X_G(w_{4}), \mathcal{L}^{\otimes k}(4\omega_{4n}))^{T_G}$$ is surjective for all $k \geq 1$. Thus, by \cref{lem5.1}, $X_1^2,$ $X_2^2,$ $X_3^2,$ $X_4^2,$ $X_1X_2,$ $X_1X_3,$ $X_1X_4,$ $X_2X_4,$ $X_3X_4,$ $Y_4$ are standard monomials in $H^0(X_G(w_4), \mathcal{L}(4\omega_{4n}))^{T_G}.$ Further, by \cref{lem5.2}, $X_2X_3=X_1X_4-Y_4.$  
		
		Note that by \cref{cor5.4}, $T_G \backslash \backslash (X_G(w_4))^{ss}_{T_G}(\mathcal{L}(4\omega_{4n}))$ is projectively normal. Let $A=\mathbb{C}[z_1,z_2,z_3,z_4,z_5,z_6,z_7,z_8,z_9,z_{10}]$ be the polynomial algebra with variables $z_i$'s ($1 \leq i \leq 10$). Then there is a surjective homomorphism $$\phi: A \rightarrow \oplus_{k \in \mathbb{Z}_{\geq 0}}H^{0}(X_G(w_4), \mathcal{L}^{\otimes k}(4\omega_{4n}))^{T_G},$$ defined by \begin{equation*}
			\begin{split}z_1 \mapsto X_1^2\\ z_2 \mapsto X_2^2\\ z_3 \mapsto X_3^2\\ z_4 \mapsto X_4^2\\ z_5 \mapsto X_1X_2\\ z_6 \mapsto X_1X_3\\ z_7 \mapsto X_1X_4\\ z_8 \mapsto X_2X_3\\ z_9 \mapsto X_2X_4\\ z_{10} \mapsto X_3X_4.
			\end{split}
		\end{equation*} 
		
		Let $I$ be the ideal of $A$ generated by the following relations:
		\begin{equation}\label{4.4}
			z_1z_2=z_5^2
		\end{equation}
		\begin{equation}\label{4.5}
			z_1z_3=z_6^2
		\end{equation}
		\begin{equation}\label{4.6}
			z_1z_4=z_7^2
		\end{equation}
		\begin{equation}\label{4.7}
			z_1z_8=z_5z_6
		\end{equation}
		\begin{equation}\label{4.8}
			z_1z_9=z_5z_7
		\end{equation} 
		\begin{equation}\label{4.9}
			z_1z_{10}=z_6z_7
		\end{equation}
		\begin{equation}\label{4.10}
			z_2z_3=z_8^2
		\end{equation}
		\begin{equation}\label{4.11}
			z_2z_4=z_9^2
		\end{equation}
		\begin{equation}\label{4.12}
			z_2z_6=z_5z_8
		\end{equation}
		\begin{equation}\label{4.13}
			z_2z_7=z_5z_9
		\end{equation}
		\begin{equation}\label{4.14}
			z_2z_{10}=z_8z_9
		\end{equation}
		\begin{equation}\label{4.15}
			z_3z_4=z_{10}^2
		\end{equation}
		\begin{equation}\label{4.16}
			z_3z_5=z_6z_8
		\end{equation}
		\begin{equation}\label{4.17}
			z_3z_7=z_6z_{10}
		\end{equation}
		\begin{equation}\label{4.18}
			z_3z_9=z_8z_{10}
		\end{equation}
		\begin{equation}\label{4.19}
			z_4z_5=z_7z_9
		\end{equation}
		\begin{equation}\label{4.20}
			z_4z_6=z_7z_{10}
		\end{equation}
		\begin{equation}\label{4.21}
			z_4z_8=z_9z_{10}
		\end{equation}
		\begin{equation}\label{4.22}
			z_5z_{10}=z_6z_9
		\end{equation}
		\begin{equation}\label{4.23}
			z_6z_{9}=z_7z_8
		\end{equation}

		Clearly, $I \subseteq \ker\phi$ and $\phi$ induces a homomorphism of $\mathbb{C}$-algebras $$\tilde{\phi}: A/ I \longrightarrow\bigoplus\limits_{k \in \mathbb{Z}_{\geq 0}}H^{0}(X_{G}(w_4), \mathcal{L}^{\otimes k}(4\omega_{4n}))^T.$$ Now, we show that $\tilde{\phi}$ is an isomorphism. 
		
		To complete the proof we use the above relations as a reduction system. We show that diamond lemma (see \cite{bergman}) holds for this reduction system by looking at the reduction of the following minimal overlapping ambiguities: \begin{center}$z_1z_2z_3, z_1z_2z_4, z_1z_2z_6, z_1z_2z_7, z_1z_2z_8, z_1z_2z_9, z_1z_2z_{10}, z_1z_3z_4, z_1z_3z_5, z_1z_3z_7, z_1z_3z_8, z_1z_3z_9,
			z_1z_3z_{10},$\\ $z_1z_4z_5, z_1z_4z_6, z_1z_4z_8, z_1z_4z_9, z_1z_4z_{10}, z_1z_5z_{10}, z_1z_6z_9, z_1z_8z_9, z_1z_8z_{10}, z_1z_9z_{10}, z_2z_3z_4, z_2z_3z_5, z_2z_3z_6,$\\ $z_2z_3z_7,  z_2z_3z_9, z_2z_3z_{10}, z_2z_4z_5, z_2z_4z_6, z_2z_4z_7, z_2z_4z_8, z_2z_4z_{10}, z_2z_5z_{10}, z_2z_6z_7, z_2z_6z_9, z_2z_6z_{10}, z_2z_7z_{10},$\\ $z_3z_4z_5, z_3z_4z_6, z_3z_4z_7, z_3z_4z_8, z_3z_4z_9, z_3z_5z_7, z_3z_5z_9, z_3z_5z_{10}, z_3z_6z_9, z_3z_7z_9, z_4z_5z_6, z_4z_5z_8, z_4z_5z_{10},$\\ $z_4z_6z_8, z_4z_6z_9. \hspace{15cm}$
		\end{center}
		Note that $z_1z_2z_3=(z_1z_2)z_3=z_3z_5^2 ~(\text{using \cref{4.4}})=(z_3z_5)z_5=z_5z_6z_8 ~(\text{using \cref{4.16}}).$ 
		
		Again, $z_1z_2z_3=z_1(z_2z_3)=z_1z_8^2 ~(\text{using \cref{4.10}})=(z_1z_8)z_8=z_5z_6z_8 ~(\text{using \cref{4.7}})$
		
		Also, $z_1z_2z_3=(z_1z_3)z_2=z_2z_6^2 ~(\text{using \cref{4.5}})=(z_2z_6)z_6=z_5z_6z_8 ~(\text{using \cref{4.12}}).$
		
		Therefore, the reduction is unique. 
		
		Likewise, we can show that the remaining overlapping ambiguities reduce to a unique expression in $z_i$'s after applying different reduction rules. Therefore, $\tilde{\phi}$ is an isomorphism. 
		
		Consider the embedding $$\psi: \mathbb{P}^3 \hookrightarrow \mathbb{P}^9$$ given by $$(x_1, x_2, x_3, x_4) \mapsto (x_1^2, x_2^2, x_3^2, x_4^2, x_1x_2, x_1x_3, x_1x_4, x_2x_3, x_2x_4, x_3x_4).$$ Let $J$ be the homogeneous ideal of $A$ generated by $2 \times 2$ minors of the matrix $\begin{pmatrix}
			z_1 & z_5 & z_6 & z_7 \\
			z_5 & z_2 & z_8 & z_9 \\
			z_6 & z_8 & z_3 & z_{10} \\
			z_7 & z_9 & z_{10} & z_4 \\
		\end{pmatrix},$ where $z_i$'s are the homogeneous coordinates of $\mathbb{P}^9.$ Note that $\psi(\mathbb{P}^3)$ is given by $J.$ Further, note that $I \subseteq J$ and $dim (A/I)=dim(A/J)=4$ (computed using Macaulay2 \cite{GS}). Hence, the GIT quotient $T_G\backslash\backslash(X_G(w_4))^{ss}_{T_G}(\mathcal{L}(4\omega_{4n}))$ is isomorphic to $(\mathbb{P}^3, \mathcal{O}_{\mathbb{P}^3}(2))$ as a polarized variety. 
		
		Proof of (v):  Since $T_G$ is linearly reductive, the restriction map $$\phi: H^{0}(X_G(w_6), \mathcal{L}^{\otimes k}(4\omega_{4n}))^{T_G} \\ \longrightarrow H^{0}(X_G(w_{5}), \mathcal{L}^{\otimes k}(4\omega_{4n}))^{T_G}$$ is surjective for all $k \geq 1$. Thus, by \cref{lem5.1}, $X_1^2,$ $X_3^2,$ $X_5^2,$ $X_1X_3,$ $X_1X_5,$ $X_3X_5$ are standard monomials in $H^0(X_G(w_4), \mathcal{L}(4\omega_{4n}))^{T_G}.$ 
		
		Note that by \cref{cor5.4}, $T_G \backslash \backslash (X_G(w_5))^{ss}_{T_G}(\mathcal{L}(4\omega_{4n}))$ is projectively normal.  Let $A=\mathbb{C}[z_1,z_2,z_3,z_4,z_5,z_6]$ be the polynomial algebra with variables $z_i$'s. Then there is a surjective homomorphism $$\phi: A \longrightarrow \oplus_{k \in \mathbb{Z}_{\geq 0}}H^{0}(X_G(w_5), \mathcal{L}^{\otimes k}(4\omega_{4n}))^{T_G},$$ defined by  $$z_1\mapsto X_1^2$$ $$z_2\mapsto X_3^2$$ $$z_3\mapsto X_5^2$$ $$z_4\mapsto X_1X_3$$ $$z_5\mapsto X_1X_5$$  $$z_6\mapsto X_3X_5.$$ Let $I$ be the ideal of $A$ generated by the following relations: 
		\begin{equation}\label{4.24}
			z_1z_2=z_4^2
		\end{equation}
		\begin{equation}\label{4.25}
			z_1z_3=z_5^2
		\end{equation}
		\begin{equation}\label{4.26}
			z_2z_3=z_6^2
		\end{equation}
		\begin{equation}\label{4.27}
			z_1z_6=z_4z_5
		\end{equation}
		\begin{equation}\label{4.28}
			z_2z_5=z_4z_6
		\end{equation}
		\begin{equation}\label{4.29}
			z_3z_4=z_5z_6.
		\end{equation}
		
		Clearly, $I\subseteq \ker\phi,$ and $\phi$ induces a homomorphism 
		$$\tilde{\phi}: A/I \longrightarrow\bigoplus\limits_{k \in \mathbb{Z}_{\geq 0}}H^{0}(X_G(w_5), \mathcal{L}^{\otimes k}(4\omega_{4n}))^{T_G}.$$
		
		Now, we show that $\tilde{\phi}$ is an isomorphism. 
		
		To complete the proof we use  the above relations as a reduction system. We show that diamond lemma (see \cite{bergman}) holds for this reduction system by looking at the reduction of the minimal overlapping ambiguities: $z_1z_2z_3, z_1z_2z_5, z_1z_2z_6, z_1z_3z_4, z_1z_3z_6, z_2z_3z_4, z_2z_3z_5.$
		
		Note that $z_1z_2z_3=z_3(z_1z_2)=z_3z_4^2 ~(\text{using \cref{4.24}})=(z_3z_4)z_4=z_4z_5z_6 ~(\text{using \cref{4.29}}).$
		
		Again $z_1z_2z_3=z_1(z_2z_3)=z_1z_6^2 ~(\text{using \cref{4.26}})=(z_1z_6)z_6=z_4z_5z_6 ~(\text{using \cref{4.27}}).$
		
		Also, $z_1z_2z_3=z_2(z_1z_3)=z_2z_5^2 ~(\text{using \cref{4.25}})=(z_2z_5)z_5=z_4z_5z_6 ~(\text{using \cref{4.28}}).$

		Likewise, we can show that the remaining overlapping ambiguities:
		
		$z_1z_2z_5, z_1z_2z_6, z_1z_3z_4, z_1z_3z_6, z_2z_3z_4, z_2z_3z_5$ reduce to unique expressions in $z_i$'s  namely, $z_4^2z_5,$ $z_4^2z_6,$ $z_4z_5^2,$ $z_5^2z_6,$ $z_4z_6^2,$ $z_5z_6^2$ after applying different reduction rules. Therefore, $\tilde{\phi}$ is an isomorphism. 
		
		Consider the embedding $$\psi: \mathbb{P}^2 \hookrightarrow \mathbb{P}^5$$ given by $$(x_1, x_2, x_3) \mapsto (x_1^2, x_2^2, x_3^2, x_1x_2, x_1x_3, x_2x_3).$$ Let $J$ be the homogeneous ideal of $A$ generated by $2 \times 2$ minors of the matrix $\begin{pmatrix}
			z_1 & z_4 & z_5  \\
			z_4 & z_2 & z_6  \\
			z_5 & z_6 & z_3 \\
		\end{pmatrix},$ where $z_i$'s are the homogeneous coordinates of $\mathbb{P}^5.$ Note that $\psi(\mathbb{P}^2)$ is given by $J.$ Further, note that $I \subseteq J$ and $dim (A/I)=dim(A/J)=3$ (computed using Macaulay2 \cite{GS}). Thus, the GIT quotient $T_G\backslash\backslash(X_G(w_5))^{ss}_{T_G}(\mathcal{L}(4\omega_{4n}))$ is isomorphic to $(\mathbb{P}^2, \mathcal{O}_{\mathbb{P}^2}(2))$ as a polarized variety.
		
	\end{proof}	
	
	\section{Projective normality}
	\subsection{$G=Spin(2n,\mathbb{C})(n\ge 4)$}	
	Let $G=Spin(2n,\mathbb{C})$ and $T_{G}$ be a maximal torus of $G.$ Let $P^{\alpha_{1}} (\supset T_G)$ be the maximal parabolic subgroup of $G$ corresponding to $\alpha_{1}.$ 
	In this section, we study the GIT quotient of $G/P^{\alpha_1}$ with respect to the descend of the $T_G$-linearized very ample line bundle $\mathcal{L}(2\omega_1).$ 
	
	Let $w=(s_{1}s_{2}\cdots s_{n-2})(s_{n-1}s_{n})(s_{n-2}s_{n-3} \cdots s_{1}) \in W^{P^{\alpha_1}}.$ Then note that $w$ is the minimal coset representative of the longest element of $W$ in $W/W^{P^{\alpha_1}}.$ Thus, we have $X(w)=G/P^{\alpha_1}.$ 
	
	Let $X= T_{G}\backslash \backslash (G/P^{\alpha_1})^{ss}_{T_G}(\mathcal{L}(2\omega_1)).$
	Let $R=\bigoplus\limits_{k\in\mathbb{Z}_{\geq 0}}R_k,$ where $R_{k}=H^0(G/P^{\alpha_1}, \mathcal{L}^{\otimes k}(2\omega_{1}))^{T_{G}}.$ Then we have $X=Proj(R).$ 
	
	Since $w(2\omega_1)<0,$ by \cite[Lemma 3.1, p.276]{KNS}, it follows that $(G/P^{\alpha_1})^{s}_{T_G} (\mathcal{L}(2\omega_{1}))$ is non-empty.  Therefore, the dimension of $X$ is $2n-2-n(=n-2).$

	For $1 \leq j \leq n-1,$  let \small{\begin{equation*}X_{j}=\ytableausetup{boxsize=3.3em}\begin{ytableau}
				j \\ j \\ \scriptstyle{2n+1-j} \\ \scriptstyle{2n+1-j} \\
			\end{ytableau}.\end{equation*}} \normalsize{Then} $X_j$'s form a basis of $R_1.$ Further, we have the following lemma.
	\begin{lemma}\label{lemma7.1}
		The graded $\mathbb{C}$-algebra $R$ is generated by $R_{1}.$
	\end{lemma}
	
	\begin{proof}
		Let $f \in R_k$ be a standard monomial. Then we show that $f = f_1 f_2,$ where $f_1$ of $f_2$ is in $R_1$.
		
		Recall that by \cref{subsection2.2}, the Young diagram $\Gamma$ associated to $f$ has the shape $p = (p_1) = (4k)$. We also denote $\Gamma$ for the Young tableau associated to the Young diagram $\Gamma.$ So, $\Gamma$ has $4k$ rows and $1$ column with strictly increasing integers from top row to bottom row.
		
		Since $f$ is $T_G$-invariant, by \cref{zeroweight}, we have \begin{equation}\label{7.1}c_\Gamma(t)=c_\Gamma(2n+1-t) 
		\end{equation}
		for all $1 \leq t \leq n,$ where $c_{\Gamma}(t)$ is the number of $t$ appearing in $\Gamma.$ Let $r_i$ be the $i$-th row of $\Gamma$ enumerated from the bottom row. We also denote the entry in the $r_i$-th row of $\Gamma$ by $r_i.$ Note that the admissible pair $(r_{2i-1}, r_{2i})$ is trivial, i.e.,  $r_{2i-1}=r_{2i}$ for all $1 \leq i \leq 2k.$ Further, by \cref{subsection2.2}(1), $\Gamma$ satisfies:  $r_{2i} \equiv r_{2i+1} \text{mod } 2,$ if $r_{2i}, r_{2i+1} \in \{n, n+1\}$ for some $1 \leq i \leq 2k-1.$ Thus, $r_{i} \notin \{n,n+1\}$ for all $1 \leq i \leq 4k.$ 
		
		Now assume that $r_1=r_2(=2n+1-j)$. Then by \cref{7.1}, $j$ appears in $\Gamma.$ Further, since $2n+1-j$ is the maximal integer appearing in $\Gamma$, it follows that $j$ is the minimal integer appearing in $\Gamma.$ Thus, we have $r_{4k-1}=r_{4k}(=j).$ Further, note that $j \neq n,$ as $r_i \notin \{n, n+1\}.$ Therefore, $X_{j}$ is a factor of $f$ for some $1 \leq j \leq n-1.$     
	\end{proof}
	\begin{corollary}
		The GIT quotient $X$ is projectively normal with respect to the descent of the $T_G$-linearized very ample line bundle $\mathcal{L}(2\omega_1)$ and is isomorphic to the projective space $(\mathbb{P}^{n-2}, \mathcal{O}_{\mathbb{P}^{n-2}}(1))$ as a polarized variety.
	\end{corollary}	
	\begin{proof}
		By \cref{lemma7.1}, $R$ is generated by $X_j$ for all $1 \leq j \leq n-1.$ Therefore, $X$ is projectively normal. On the other hand, dimension of $X$ is $n-2.$ Thus, $R$ is the polynomial algebra generated by $X_j$ for $1 \leq j \leq n-1.$ Therefore, $X$ is isomorphic to the projective space $(\mathbb{P}^{n-2}, \mathcal{O}_{\mathbb{P}^{n-2}}(1))$ as a polarized variety.
	\end{proof}
	
	\subsection{$G=Sp(2n,\mathbb{C})$ $(n \geq 2)$} Let $G=Sp(2n,\mathbb{C})$ $(n \geq 2)$ and $T_{G}$ be a maximal torus of $G.$ Let $P^{\alpha_{1}} (\supset T_G)$ be the maximal parabolic subgroup of $G$ corresponding to $\alpha_{1}.$  In this section, we study the GIT quotient of $G/P^{\alpha_1}$ with respect to the descend of the $T_G$-linearized very ample line bundle $\mathcal{L}(2\omega_1).$ 
	
	Let $w=(s_{1}s_{2}\cdots s_{n-1})(s_{n})(s_{n-1}s_{n-2} \cdots s_{1}) \in W^{P^{\alpha_1}}.$ Then note that $w$ is the minimal coset representative of the longest element of $W$ in $W/W^{P^{\alpha_1}}.$ Thus, we have $X(w)=G/P^{\alpha_1}.$ 
	
	Let $X= T_{G}\backslash \backslash (G/P^{\alpha_1})^{ss}_{T_G}(\mathcal{L}(2\omega_1)).$
	Let $R=\bigoplus\limits_{k\in\mathbb{Z}_{\geq 0}}R_k,$ where $R_{k}=H^0(G/P^{\alpha_1}, \mathcal{L}^{\otimes k}(2\omega_{1}))^{T_{G}}.$ Then we have $X=Proj(R).$ 
	
	Since $w(2\omega_1)<0,$ by \cite[Lemma 3.1, p.276]{KNS}, it follows that $(G/P^{\alpha_1})^{s}_T (\mathcal{L}(2\omega_{1})$ is non-empty.  Therefore, the dimension of $X$ is $2n-1-n(=n-1).$

	For $1 \leq j \leq n,$ let \small{\begin{equation*}X_{j}=
			\ytableausetup{boxsize=3.5em}
			\begin{ytableau}
				j \\ j  \\ \scriptstyle{2n+1-j}\\ \scriptstyle{2n+1-j} \\
			\end{ytableau}.\end{equation*}} \normalsize{Note} that $X_j$'s form a basis of $R_1.$ Further, we have the following lemma.
	\begin{lemma}\label{lemma7.3}
		The graded $\mathbb{C}$-algebra $R$ is generated by $R_{1}.$
	\end{lemma}
	
	\begin{proof}
		Let $f \in R_k$ be a standard monomial. We show that $f = f_1 f_2,$ where $f_1$ or $f_2$ is in $R_1$.
		
		Recall that by \cref{subsection2.3}, the Young diagram $\Gamma$ associated to $f$ has the shape $p = (p_1) = (4k)$. We also denote $\Gamma$ for the Young tableau associated to the Young diagram $\Gamma.$ So, $\Gamma$ has $4k$ rows and $1$ column with strictly increasing integers from top row to bottom row.
		
		Since $f$ is $T_G$-invariant, by \cref{zeroweight}, we have \begin{equation}\label{7.2}c_\Gamma(t)=c_\Gamma(2n+1-t) 
		\end{equation}
		for all $1 \leq t \leq n,$ where $c_{\Gamma}(t)$ is the number of $t$ appearing in $\Gamma.$ Let $r_i$ be the $i$-th row of $\Gamma$ enumerated from the bottom row. We also denote the entry in the $r_i$-th row of $\Gamma$ by $r_i.$ Note that the admissible pair $(r_{2i-1}, r_{2i})$ is trivial, i.e.,  $r_{2i-1}=r_{2i}$ for all $1 \leq i \leq 2k.$   
		
		Now assume that $r_1=r_2(=2n+1-j)$. Then by \cref{7.2}, $j$ appears in $\Gamma.$ Further, since $2n+1-j$ is the maximal integer appearing in $\Gamma$, it follows that $j$ is the minimal integer appearing in $\Gamma.$ Thus, we have $r_{4k-1}=r_{4k}(=j).$ Therefore, $X_{j}$ is a factor of $f$ for some $1 \leq j \leq n.$

	\end{proof}
	\begin{corollary}
		The GIT quotient $X$ is projectively normal with respect to the descent of the $T_G$-linearized very ample line bundle $\mathcal{L}(2\omega_1)$ and is isomorphic to the projective space $(\mathbb{P}^{n-1}, \mathcal{O}_{\mathbb{P}^{n-1}}(1))$ as a polarized variety.
	\end{corollary}	
	\begin{proof}
		By \cref{lemma7.3}, $R$ is generated by $X_j$ for all $1 \leq j \leq n.$ On the other hand, dimension of $X$ is $n-1.$ Thus, $R$ is the polynomial algebra generated by $X_j$ for all $1 \leq j \leq n.$ Therefore, $X$ is isomorphic to  $(\mathbb{P}^{n-1}, \mathcal{O}_{\mathbb{P}^{n-1}}(1))$ as a polarized variety. 
	\end{proof}

	{\bf Acknowledgements:} We thank Professor S. Senthamarai Kannan for his constant encouragement and various helpful suggestion.

\end{document}